\PassOptionsToPackage{dvipsnames}{xcolor}
\documentclass[11pt]{amsart}
\usepackage{verbatim}
\usepackage{graphicx}
\usepackage{bbding}
\usepackage[shortlabels]{enumitem}
\setlist[enumerate,1]{label={\upshape(\arabic*)},leftmargin=1cm}
\usepackage[margin=1in]{geometry}
\usepackage[all,arc]{xy}
\SelectTips{lu}{10}
\usepackage{tikz}
\usetikzlibrary{cd}
\usetikzlibrary{calc,intersections,through,backgrounds}
\usepackage{pgfplots}
\pgfplotsset{compat=1.5.1}
\usepackage{amssymb}
\usepackage[mathscr]{euscript} 

\numberwithin{equation}{section}
\newtheorem{thmx}{Theorem}
\newtheorem{thm}[equation]{Theorem}

\newtheorem{lem}[equation]{Lemma}
\newtheorem*{lem*}{Lemma}
\newtheorem{cor}[equation]{Corollary}
\newtheorem{conj}[equation]{Conjecture}
\newtheorem{prop}[equation]{Proposition}

\newtheorem*{prop*}{Proposition}

\theoremstyle{definition}
\newtheorem{defn}[equation]{Definition}
\newtheorem{exmp}[equation]{Example}

\theoremstyle{remark}
\newtheorem{ex}[equation]{Exercise}
\newtheorem{rem}[equation]{Remark} 

\newtheorem{claim}[equation]{Claim}

\newcommand{\al}{\alpha}
\newcommand{\A}{\mathscr{A}}
\newcommand{\B}{\mathscr{B}}

\newcommand{\cm}{\text{\tiny\Checkmark}}

\newcommand{\de}{\delta}
\newcommand{\ds}{\displaystyle}

\newcommand{\e}{\emptyset}

\newcommand{\ep}{{\varepsilon}}
\newcommand{\F}{\mathscr{F}}
\newcommand{\fin}{{\it fin}}

\newcommand{\g}{{\gamma}}

\newcommand{\lb}{\lambda}

\newcommand{\m}{\medskip}

\newcommand{\n}{\noindent}
\newcommand{\N}{{\mathbb N}}

\newcommand{\R}{\mathbb{R}}

\newcommand{\s}{{\sigma}}

\newcommand{\sm}{\setminus}

\newcommand{\sq}{\subseteq}
\newcommand{\sz}{\scriptsize}
\newcommand{\ta}{\theta}

\newcommand{\ts}{\textstyle}

\newcommand{\wt}{\widetilde}

\title{Repeated integrals of increasing functions}

\author{Maxim R.~Burke}

\address{School of Mathematical and Computational Sciences, University of Prince Edward Island
Charlottetown PE, Canada C1A 4P3}

\email{burke@upei.ca}

\thanks{The first author's research supported by NSERC. The author thanks Yinhe Peng and the Academy of Mathematics and Systems Science of the Chinese Academy of Sciences for their hospitality in the fall of 2024 when some of this research was carried out. The author also thanks Bill Weiss for a helpful discussion on early results related to this paper.}

\subjclass{Primary 26A241, 46E10, 26A48; Secondary 54C05, 26D05, 26A51, 41A30.}%
\keywords{repeated integrals, increasing function, convex function, continuous transversal, topology on $C^\infty$ functions, comonotone approximation}%

\author{Maleeha Haris}
\email{mharis16666@upei.ca}

\author{Madhavendra}
\email{mmadhavendra@upei.ca}

\date{}

\begin{document}

\begin{abstract}
Motivated by a problem on comonotone approximation of $C^n$ functions by entire functions, for increasing functions $f\colon[0,1]\to[0,1]$, we characterize the possible values of $(a,b,c)$, where $a=I(f)(1)$, $b=I^2(f)(1)$, $c=I^3(f)(1)$ ($I$ is the integral operator $I(f)(x)=\int_0^xf(t)\,dt$), as those which satisfy the conditions $0\leq a\leq 1$, $a^2/2\leq b\leq a/2$, $2b^2\leq 3ac$,
$a^2 + 4b^2 + 6c\leq 6ac +2ab+2b$, and $0\leq c\leq a/6$.
Our main theorem states that if $a,b,c$ are real numbers for which the inequalities are strict, then there is a function $f$ satisfying $a=I(f)(1)$, $b=I^2(f)(1)$, $c=I^3(f)(1)$ which is $C^\infty$ with $f(0)=0$, $f(1)=1$, $Df(x)>0$ for $0<x<1$, and whose derivatives $D^jf(0)$ and $D^jf(1)$, $j\geq 1$, are arbitrary as long as they are consistent with the increasing nature of $f$.
The construction of $f$ proceeds by starting with a continuous parametrization $s\mapsto \rho_s\in C^\infty([0,1])$ defined on an open subset of $\R^4$, and composing with successive continuous transversals through the open set to fix the values of $I^j(\rho_s)(1)$ for $j=0,1,2,3$.

Addressing the aforementioned problem on comonotone approximation, we examine the set $V_n\sq\R^{2(n+1)}$ of possible values $D^jf(0)$, $D^jf(1)$, $j=0,\dots,n$,
of the derivatives of a $C^n$ function at the endpoints when $D^nf$ is increasing but not constant. We make a conjecture about the nature of this set and prove our conjecture for $n\leq 3$ as a consequence of the theorem mentioned above.
\end{abstract}

\date{\today}

\maketitle

%
%

\section{Introduction}

This paper has its origins in the study of comonotone approximation by entire functions of $C^n$ functions having piecewise monotone derivatives.
A function $f\colon\R\to\R$ is \emph{piecewise monotone} if there is a closed discrete set $K\sq \R$ such that $f$ is monotone on the components of the complement of $K$. Two functions $f,g$ are \emph{comonotone} if there is a closed discrete set $K\sq \R$ such that $f$ and $g$ have the same monotonicity on the components of the complement of $K$.

It is known (see \cite{Bu2019}, Theorem A) that if $n$ is a nonnegative integer and for some $m>n$, $f\colon \mathbb{R}\to\mathbb{R}$ is a $C^m$ function such that $D^{n+1}f$ has no flat points (i.e., for each $x\in\R$, there is an integer $i\in[n+1,m]$ such that $D^if(x)\not=0$), then for any closed discrete set $E\sq\R$, and for any positive continuous function $\varepsilon\colon\R\to\R$, there is a function $g\colon\R\to\R$ which is the restriction of an entire function, whose derivatives approximate those of $f$ within $\ep$ with interpolation on $E$, in symbols, for all $x\in \R$ we have
\begin{itemize}
\item
$|D^ig(x) - D^if(x)|<\varepsilon(x)$, $0\leq i\leq m$;

\item
$D^ig(x) = D^if(x)$ when $x\in E$, $0\leq i\leq m$;
\end{itemize}
and each $D^kf$ is
comonotone with $D^kg$, $k=0,\dots,n$.

The following theorem is a partial answer to the question of what can be said if $m=n$, i.e., $f$ may not have derivatives of order larger than $n$.
In the statement, a \emph{platform} of a continuous function $f\colon\R\to\R$ is a maximal interval on which $f$ is constant, and a compact platform $I=[a,b]$ is a \emph{turning platform} if for some $\ep>0$, the values of $f$ on $(a-\ep,b+\ep)\sm I$ are either all larger than the constant value $k$ of $f$ on $I$, or all smaller than $k$. If $D^nf$ has a platform unbounded above, then $A^{n}_{\max}=A^{n}_{\max}(f)$ denotes this platform, and $A^n_{\max}=\e$ otherwise. Similarly, $A^{n}_{\min}=A^{n}_{\min}(f)$ denotes the platform of $D^nf$ which is unbounded below, if there is one, and $A^n_{\min}=\e$ otherwise.

The theorem makes an assumption $(P_n)$ which we discuss shortly.

\begin{thm}[\cite{Bu2025}]\label{t:C.n.Hoischen.C.n.fcts}
Assume $(P_n)$. Let $f\colon \R\to\R$ be a $C^n$ function such that $D^jf$ is piecewise monotone for $0\leq j\leq n$, and $D^nf$ is not constant. Suppose there is a closed discrete set $E\sq \R$, disjoint from $A^n_{\min}\cup A^n_{\max}$, having no more than one point on any platform of $D^nf$, and having exactly one point on each turning platform of $D^jf$, $0\leq j\leq n$.
Let $\ep\colon \R\to\R$ be a positive continuous function.
Then there is a function $g\colon \R\to\R$ which is the restriction of an entire function and satisfies the following conditions for $x\in \R$.
\begin{enumerate}
\item\label{p:C.n.Hoischen:1}
$|D^jg(x)-D^jf(x)|<\ep(x)$, $0\leq j\leq n$.

\item\label{p:C.n.Hoischen:2}
$D^jg(x)=D^jf(x)$ when $x\in E$, $0\leq j\leq n$.

\item\label{p:C.n.Hoischen:3}
$D^jg(x)\not=0$ when $x\not\in E$, $0<j\leq n+1$.
\end{enumerate}
\end{thm}

On a component $I$ of the complement of $E$, for $0\leq j\leq n$, $D^jg$ is monotone by \ref{p:C.n.Hoischen:3}, and $D^jf$ is also monotone on $I$ because $E$ has a point on each turning platform of $D^jf$. From \ref{p:C.n.Hoischen:2} it follows that these monotonicities are the same for $D^jg$ and $D^jf$ if $I$ is bounded (so both endpoints are in $E$), and when $I$ is not bounded the same conclusion follows from \ref{p:C.n.Hoischen:2} and \ref{p:C.n.Hoischen:1} if we take $\ep(x)$ so that it vanishes at $\pm\infty$.

To state the assumption $(P_n)$, we first define the following set of points $b=(b_0,\dots,b_n)\in\R^{n+1}$.
\begin{align*}
W_n & = \{b\in\R^{n+1}: \text{there is an $f\in C^n[0,1]$ with $D^nf$ increasing\footnotemark\
and nonconstant such that} \\
& \rule{2.5cm}{0cm}\text{$D^jf(0)=0$ and $D^jf(1)=b_j$ for all $j=0,\dots,n$}\}.
\end{align*}
\footnotetext{For functions $f\colon I\to \R$ on an interval $I$, we use the word increasing in its non-strict sense, i.e., $f$ is increasing if $x\leq y$ implies $f(x)\leq f(y)$ (for all $x,y\in I$). Similarly for the word decreasing.}
$(P_n)$ is the statement that $W_n$ is open in $\R^{n+1}$ and that for each $b\in W_n$ we may choose the witnessing function $f$ to be $C^\infty$ and so that $D^{n+1}f(x)>0$ for all $x\in (0,1)$,
$D^{n+1}f(0) = D^{n+1}f(1) = 1$ and $D^jf(0) = D^jf(1) = 0$ for $j>n+1$.
\begin{conj}
$(P_n)$ is true for all nonnegative integers $n$.
\end{conj}
In this paper, we investigate the nature of the sets $W_n$. Our main results are for the case $n\leq 3$.
We prove the following. The case $n=0$ is proven in \cite{Bu2019}, Proposition 6.2.

\begin{thmx}[Theorem \ref{t:conj}]\label{t:A}
The statements $(P_n)$, $n=0,1,2,3$, all hold. We have the following:
\begin{align*}
W_0 & = \{a\in \R:a>0\}, \\
W_1 & = \{(a,b)\in \R^2 :0<a<b\}, \\
W_2 & = \{(a,b,c)\in \R^3 :0<2a<b,\,b^2<2ac\}, \\
W_3 & = \{(a,b,c,d)\in \R^4 :0<c<d,\,2b^2<3ac,\,6ad + 4b^2 + c^2 < 6ac + 2bc + 2bd\}.
\end{align*}
\end{thmx}
Our approach to the proof is to construct
a function $g$ which will be the third derivative of $f$ and then get $f$ as the third integral of $g$, doing so in such a way that the successive integrals have prescribed values. Our main theorem is the following. The statement mentions functions $\s_\de\colon[0,\de]\to[0,\de]$ and $\tau_\de\colon [1-\de,1]\to [1-\de,1]$ which are introduced in Section \ref{s:spec.higher}. They satisfy $\s_\de(0)=0$, $D\s_\de>0$ on $(0,\de]$, $\tau_\de(1)=1$ and $D\tau_\de>0$ on $[1-\de,1)$. The derivatives $D^n\s_\de(0)$ and $D^n\tau_\de(1)$, $n\geq 1$, can be specified arbitrarily as long as they are consistent with the increasing nature of $\s_\de$ and $\tau_\de$ (Proposition \ref{p:endpt.der}).

\begin{thmx}[Theorem \ref{t:n=3.b}]\label{t:B}
Let $a,b,c$ be positive numbers satisfying
\begin{gather}\label{eq:3ineq}
0<a<1,\ \ \frac{a^2}2 < b < \frac{a}2,\ \ \frac{2b^2}{3a} < c < \frac{-a^2 + 2ab -4b^2 + 2b}{6(1-a)}.
\end{gather}
For each small enough $\de>0$, there is a $C^\infty$ function $f\colon [0,1]\to[0,1]$ such that $f=\s_\de$ on $[0,\de/2]$, $f=\tau_\de$ on $[1-\de/2,1]$, $Df>0$ on $(0,1)$, $I(f)(1)=a$, $I^2(f)(1)=b$ and $I^3(f)(1)=c$.
\end{thmx}

($I$ is the integral operator $I(f)(x)=\int_0^xf(t)\,dt$.)
The bounds on the values of $a,b,c$ are exact. This follows from Theorem \ref{t:a-e} part of which we state here as Theorem C.
For intervals $J$ and functions $f,g\colon J\to\R$, we write $f=_\fin g$ if $f$ and $g$ are equal modulo the ideal of finite sets, i.e., $\{x\in J:f(x)\not=g(x)\}$ is finite.

\begin{thmx}[Cf.\ Theorem \ref{t:a-e}]\label{t:C}
The possible values of $a=I(f)(1)$, $b=I^2(f)(1)$, and $c=I^3(f)(1)$ for increasing functions $f\colon[0,1]\to[0,1]$ are those which satisfy the following inequalities.
\begin{enumerate}[\rm(a)]
\item
$0\leq a\leq 1$

\item
$a^2/2\leq b\leq a/2$

\item
$2b^2\leq 3ac$

\item
$6(1-a)c\leq -a^2+2ab-4b^2+2b$

\item
$0\leq c\leq a/6$
\end{enumerate}
Either of $(a)$ or $(b)$ can be omitted, and $(c)$ can be omitted if $0<a<1$.
The inequalities are all strict unless $f=_\fin g$ for some $g\colon[0,1]\to[0,1]$ which is either constant or a $2$-step step function%
\footnote{$g$ is a $2$-step step function if its range consists of two elements both of whose preimage is a nontrivial subinterval of $[0,1]$.}
whose value on the first step is $0$ or whose value on the second step is $1$.
In particular, the inequalities in {\rm (\ref{eq:3ineq})} hold when $f$ has a positive derivative at some point of the interval $[0,1]$.
\end{thmx}

The paper is organized as follows. In Section \ref{s:prelim}, we introduce some notation and prove some technical facts needed later. We also review the properties of convex functions and the properties of a standard topology on the $C^\infty$ functions. In Section \ref{s:spec.higher}, we give our main device for modifying functions so that the derivatives at endpoints will have specified values. In Section \ref{s:endpoint.values}, we analyze the nature of the sets $W_n$ defined above. In Section \ref{s:p.w.l.}, we introduce a technique for rounding off the corners of a piecewise linear function to get a $C^\infty$ function. In subsequent sections we need to be able to do this so that the output function depends continuously on the input function in a suitable sense. In Section \ref{s:If.I2f}, for increasing functions $f\colon[0,1]\to[0,1]$ we develop necessary conditions for the successive integrals $I(f)(1)$, $I^2(f)(1)$ and $I^3(f)(1)$ to have specified values, and we prove Theorem \ref{t:C}.
In Section \ref{s:fuv} we prove our main theorem, Theorem \ref{t:B}.  In Section \ref{s:Thm9.1}, we use Theorem \ref{t:B} to prove Theorem \ref{t:A}.

\section{Preliminaries}\label{s:prelim}

Fix an interval $I$ of the real line $\R$. For $a\in\R$, let $a_+=\max(0,a)$. For $f\colon I\to\R$, we use $Df(x)$ for the derivative of $f$ at $x$, writing $D_-f(x)$ and $D_+f(x)$ for the one-sided derivatives on the left and right, respectively.
We write $C^n(I)$ for the functions on $I$ which are $n$ times continuously differentiable, and $C^\infty(I)$ for the functions on $I$ which have derivatives of all orders. At endpoints of $I$, if there are any, the derivatives are meant in the one-sided sense. We use the notation $f(x+)$ for $\lim_{t\to x+}f(t)$, and $f(x-)$ for $\lim_{t\to x-}f(t)$, when these limits exist.
We shall make use of Borel's theorem on the surjectivity of the derivative map $f\mapsto (D^kf(p):k=0,1,\dots)$, namely, given $p\in \R$ and a sequence of real numbers $t_k$, $k=0,1,\dots$, there is an $f\in C^\infty(\R)$ such that $D^kf(p)=t_k$, $k=0,1,\dots$ (see \cite{Nar}, Theorem 1.5.4).

For an integrable function $f$ on an interval containing a point $a$, we let $I_a(f)$ denote the function on the same interval given by $I_a(f)(x)=\int_a^xf(t)\,dt$. When $a=0$ we omit the subscript, writing $I(f)$, or just $If$, for $I_0(f)$, as we did in the introduction.

We state without proof the following easy but useful observation.

\begin{prop}\label{p:fin}
Suppose $J$ is an open interval in $\R$ and $f\colon J\to \R$ is constant. If $g\colon J\to \R$ is increasing and $g=_\fin f$ then $g(x)=f(x)$ for all $x\in J$.
\end{prop}

For a function on a product $f\colon X\times Y\to\R$, for $x\in X$ we write $f_x$ for the vertical section of $f$ at $x$ given by $f_x(y)=f(x,y)$, $y\in Y$. For any function $f$ and a subset $A$ of its domain, $f\!\mid_A$ denotes the restriction of $f$ to $A$.
$\N$ denotes the set of natural numbers $1,2,3,\dots$.

\m

\n \textbf{A. A system of inequalities.}\ \
We shall make use of the following properties of the system of inequalities $(\beta)$ defined in the next proposition, which is related to the statements of Theorems A, B, and C.

\begin{prop}\label{p:prelim.ineq}
We consider the following system $(\al)$ of linear equations
\begin{gather}
2y^2 = 3xz \label{eq.a} \\
6x + 4y^2 + z^2 = 6xz + 2yz + 2y \label{eq.b}
\end{gather}
and the following system $(\beta)$ of linear inequalities in real numbers $x,y,z$.
\begin{gather}
2y^2 \leq 3xz \label{eq.c} \\
6x + 4y^2 + z^2 \leq 6xz + 2yz + 2y \label{eq.d}
\end{gather}

\begin{enumerate}
\item
The solutions to $(\al)$ are $(x,y,z)=(z/6,z/2,z)$ and $(x,y,z)=(z^3/6,z^2/2,z)$ for $z\in\R$.

\item
Let $(x,y,z)$ be a solution to $(\beta)$.
\begin{enumerate}[\rm(i)]
\item
If $z=0$ then $(x,y,z)=(x,0,0)$ with $x\leq 0$. Conversely, all triples $(x,0,0)$ with $x\leq 0$ are solutions to $(\beta)$, and satisfy {\rm(\ref{eq.a})}.

\item
If $z=1$ then $(x,y,z)=(x,1/2,1)$ with $x\geq 1/6$. Conversely, all triples $(x,1/2,1)$ with $x\geq 1/6$ are solutions to $(\beta)$, and satisfy {\rm(\ref{eq.b})}.

\item
$0\leq z\leq 1$ if and only if $z^2/2\leq y\leq z/2$. When $y=z^2/2$ or $y=z/2$, $(x,y,z)$ is a solution to $(\al)$.

\item
$0 < z < 1$ if and only if $z^2/2 < y < z/2$, when $(x,y,z)$ is not a solution to $(\al)$.

\item
If $0<z<1$ then $(2/3)y^2<x\leq y/3$, with equality on the right if and only if $y=z/2$.

\item
If $0\leq z<1$ then $x\leq z/6$, with equality if and only if $y=z/2$.
\end{enumerate}
\end{enumerate}
\end{prop}

\begin{rem}
The inequalities in $(\beta)$, even if strict, imply neither $z>0$ nor $z<1$ as can be seen by considering the triples $(x,y,z)=(-2,1,-1)$ and $(x,y,z)=(1,2,3)$. The strict inequalities do however rule out the values $z=0,1$ by (2) (i, ii).
\end{rem}

\begin{rem}\label{r:l.r}
When $0<z<1$, the system $(\beta)$ can be written as $l\leq x\leq r$, where \begin{gather*}
l = \frac{2y^2}{3z}\ \ \text{and}\ \ r = \frac{-4y^2-z^2+2yz+2y}{6(1-z)}.
\end{gather*}
If $y=z^2/2$ or $y=z/2$ then it is readily checked that $l=r$ ($=z^3/6$ and $z/6$, respectively). When $z^2/2<y<z/2$ we shall see later (Proposition \ref{p:g1g2}) that $l<r$. This provides a method for generating the solutions to $(\beta)$ when $0<z<1$, namely, choose $z,y,x$ in that order, first choosing $z$ so that $0<z<1$, ensuring that $z^2/2<z/2$, then choosing $y$ so that $z^2/2\leq y\leq z/2$, ensuring that $l\leq r$, then choosing $x$ so that $l\leq x\leq r$. Since $l\leq x\leq r$ holds, we have a solution to $(\beta)$. Conversely, every solution $(x,y,z)$ to $(\beta)$ with $0<z<1$ satisfies $z^2/2\leq y\leq z/2$ by (iii), and satisfies $l\leq x\leq r$, and so is obtainable by such a sequence of selections.

For the system $(\beta)$ with strict inequalities, the solutions when $0<z<1$ are obtainable similarly, choosing $z,y,x$ so that $0<z<1$, then $z^2/2<y<z/2$, then $l<x<r$. (Use (iv) instead of (iii).)
\end{rem}

\begin{proof}
(1)
Clearly $(x,y,z)=(0,0,0)$ is a solution to $(\al)$, and if $z=0$ then (\ref{eq.a}) gives $y=0$ and then (\ref{eq.b}) gives $x=0$. Now suppose $z\not=0$. To solve the system, from (\ref{eq.a}) get $6x=4y^2/z$ and substitute into (\ref{eq.b}). This gives $4y^2/z+4y^2+z^2 = 4y^2 + 2yz + 2y$, or $4y^2-2(z^2+z)y+z^3=(2y-z)(2y-z^2)=0$. Hence, when $z\not=0$, the solutions to (\ref{eq.b}) in the presence of (\ref{eq.a}) are precisely $y=z/2$ and $y=z^2/2$. For (\ref{eq.a}) we need $x=(2/3)y^2/z$, so $x=z/6$ when $y=z/2$ and $x=z^3/6$ when $y=z^2/2$.

\smallskip

(2) (i) If $z=0$ then (\ref{eq.c}) implies $y=0$ and then (\ref{eq.d}) reduces to $x\leq 0$.

\smallskip

(ii) If $z=1$ then (\ref{eq.d}) reduces to $(2y-1)^2\leq 0$ which has the unique solution $y=1/2$. Then (\ref{eq.c}) becomes $x\geq 2y^2/(3z) = 1/6$.

\smallskip

(iii) If $z^2/2\leq y\leq z/2$ then $z^2\leq z$, so $0\leq z\leq 1$. For the converse, (i) and (ii) show that if $z=0$ or $z=1$ then $z^2/2 = y = z/2$. For the case $0<z<1$, write (\ref{eq.d}) as $6x-6xz \leq -4y^2 - z^2 + 2yz + 2y$. From (\ref{eq.c}) we get $4y^2/z\leq 6x$. Multiplying by $(1-z)$ leads by transitivity to  to $4y^2/z -4y^2 \leq -4y^2 - z^2 + 2yz + 2y$, or $4y^2-2(z^2+z)y+z^3=(2y-z)(2y-z^2)\leq 0$. The solutions are given by $z^2/2\leq y\leq z/2$.
If $y=z^2/2$ then from (\ref{eq.c}) we get $z^4/2\leq 3xz$, or $z^3\leq 6x$. From (\ref{eq.d}) we get $6x + z^4 + z^2 \leq 6xz + z^3 + z^2$ which yields $6x(1-z)\leq z^3(1-z)$, so $6x\leq z^3$. Thus, $x=z^3/6$ and $y=z^2/2$, and by (1), $(x,y,z)$ is a solution to $(\al)$.
Similarly, when $y=z/2$, (\ref{eq.c}) and (\ref{eq.d}) give $x=z/6$ and from (1) we get that $(x,y,z)$ is a solution to $(\al)$.

\smallskip

(iv) If $z^2/2< y< z/2$ then $z^2< z$, so $0< z< 1$. The other direction follows from (iii).

\smallskip

(v)
Assume that $0<z<1$. By (iii), $z^2/2\leq y\leq z/2$, so $y>0$. By (\ref{eq.c}), $x\geq (2/3)y^2/z > (2/3)y^2$. There remains to show that $x\leq y/3$ with equality when and only when $y=z/2$.
By (iii) and (1),
when $y=z/2$ we have $x=z/6$, so $x = z/6 = y/3$.
Writing (\ref{eq.d}) as $x\leq r$, with $r$ as in Remark \ref{r:l.r}, we see that $x < y/3$ when $z^2/2\leq y < z/2$ will follow if we show $r<y/3$, i.e.,
\begin{gather}\label{eq:v}
 \frac{-4y^2-z^2+2yz+2y}{6(1-z)} < \frac{y}3,
\end{gather}
which is equivalent, upon clearing the denominators and rearranging, to
$z < y+z^2/(4y)$. This is indeed true because the function $f(t) = t + z^2/(4t)$ decreases strictly on $(0,z/2]$ and hence has unique minimum value $f(z/2) = z$. Thus, $f(y)=y+z^2/(4y) > z$, and hence (\ref{eq:v}) holds, when $z^2/2\leq y<z/2$.

\smallskip

(vi)
Since $z<1$, we can write (\ref{eq.d}) as $x\leq r$, so $x\leq z/6$ will follow if we show that $r\leq z/6$. In the inequality $r\leq z/6$, cancelling the $6$'s and multiplying by $1-z$, then simplifying and factoring, gives the equivalent form $(2y-1)(2y-z)\geq 0$ which is true because by (iii), $y\leq z/2<1/2$.
Thus, $x\leq r\leq z/6$. If $x=z/6$ then $r=z/6$ and the same computation leads to $(2y-1)(2y-z)=0$ giving $y=z/2$ (since we must have $y\leq z/2$). Conversely, if $y=z/2$ then by (iii) and (1), $x=z/6$.
\end{proof}

\n \textbf{B. Convex functions.}\ \
We need some properties of convex functions defined on an interval $I$. We say that a function $f\colon I\to\R$ is \emph{convex} if
$f((1-\lb)x+\lb y)\leq (1-\lb)f(x)+\lb f(y)$ for all $x,y\in I$ and all $\lb\in[0,1]$.
The following proposition lists standard properties of convex functions.
The introductory chapters of textbooks on convexity contain the proofs, for example \cite{Hor1994}, \cite{NP2018}, \cite{RV}.

\begin{prop}\label{p:convex.properties}
Let $f\colon I\to\R$ be convex. Write $I^\circ$ for the interior of $I$ in $\R$. Write $\al=\inf I$, $\beta=\sup I$ taken in $[-\infty,\infty]$. $($So $I^\circ = I\sm\{\al,\beta\}$.$)$
\begin{enumerate}
\item
For any $x<y$ in $I$, $f(t)\leq f(x) + ((f(y)-f(x))/(y-x))(t-x)$, $x\leq t\leq y$ $($i.e., the graph of $f$ is below its secant on $[x,y])$.

\item
The slope of a secant of $f$ over an interval gets larger if the endpoints of the interval are moved to the right, more precisely, whenever $a<b$ and $a'<b'$ in $I$, with $a\leq a'$ and $b\leq b'$, we have $(f(b)-f(a))/(b-a)\leq (f(b')-f(a'))/(b'-a')$.

\item
If $a<b<c$ are points of $I$ and the point $(b,f(b))$ is on the secant of $f$ over the interval $[a,c]$, then the graph of $f$ coincides with its secant on $[a,c]$.

\item
$f$ is continuous on $I^\circ$.

\item
$D_-f(x)$ exists at each $x\in I\sm\{\al\}$. It is finite except possibly when $x=\beta\in I$. $D_-f$ is increasing on its domain $I\sm\{\al\}$.
$D_-f$ is left-continuous on $I^0$ and is left-continuous at $\beta$ if $\beta\in I$ and $f$ is continuous at $\beta$.

\item
$D_+f(x)$ exists at each $x\in I\sm\{\beta\}$. It is finite except possibly when $x=\al\in I$. $D_+f$ is increasing on its domain $I\sm\{\beta\}$.
$D_+f$ is right-continuous on $I^0$ and is right-continuous at $\al$ if $\al\in I$ and $f$ is continuous at $\al$.

\item
For each $x\in I^\circ$, $D_-f(x)\leq D_+f(x)$. We have $D_{-}f(x) = D_{+}f(x)$ if and only if $D_{-}f$ is continuous at $x$ if and only if $D_{+}f$ is continuous at $x$.

\item
$f$ is differentiable on $I^\circ$ except at the countably many points $x$ where $D_-f(x) < D_+f(x)$.

\item
For $x\in I$ and $m\in\R$, consider the line $L(s) = f(x) + m(s-x)$ which satisfies $L(x)=f(x)$.
\begin{enumerate}[\rm(a)]
\item
If $x\not=\beta$ and $m\leq D_+f(x)$ then $f(s)\geq L(s)$ for all $s\in I$, $s\geq x$.

\item
If $x\not=\al$ and $m\geq D_-f(x)$ then $f(s)\geq L(s)$ for all $s\in I$, $s\leq x$.

\item
If $x\in I^\circ$ and $D_-f(x)\leq m\leq D_+f(x)$ then $f(s)\geq L(s)$ for all $s\in I$.
\end{enumerate}
In all three settings, we say that the line $L(s)=m(s-x)+f(x)$ is a \emph{support line} for $f$ at $x$ on the interval $I\cap [x,\beta]$, $I\cap [\al,x]$ and $I$, respectively.

\item
If $f$ is continuous, then for any $p\in I$, $f(x)=f(p)+\int_p^x D_-f(t)\,dt$ for all $x\in I$. We can replace $D_-f$ by $D_+f$ since they agree except on a countable set.

\item
Either $f$ is monotone on $I^\circ$, or there is a point $\gamma\in I^\circ$ such that $f$ is decreasing on $(\al,\gamma]$ and increasing on $[\gamma,\beta)$. Consequently, $\lim_{x\to\al+}f(x)$ and $\lim_{x\to \beta-} f(x)$ exist in $[-\infty,\infty]$.

\item
{\rm(a)} If $\al\in I$ then $-\infty<\lim_{x\to\al+}f(x)\leq f(\al)$.
If $\lim_{x\to\al+}f(x) < f(\al)$ then $D_+f(\al)=-\infty$.
{\rm(b)} If $\beta\in I$ then $-\infty<\lim_{x\to\beta-}f(x)\leq f(\beta)$.
If $\lim_{x\to\beta-}f(x) < f(\beta)$ then $D_-f(\beta)=\infty$.

\item
Let $a<b$ be points of $I$. If $D_+f(a) = D_-f(b)$ then $f$ coincides with its secant on $[a,b]$. If $D_+f(a)<D_-f(b)$ then $D_+f(a) < (f(b)-f(a))/(b-a) < D_-f(b)$.
\end{enumerate}
\end{prop}

Several of these properties characterize convexity. It will be useful to have the following converses for (5) (or (6)) and (10).

\begin{prop}
Let $f\colon I\to\R$ and let $a\in I$.
\begin{enumerate}
\item
If $f$ is continuous and has a left derivative $D_-f$ or a right derivative $D_+f$ which is increasing on $I^\circ$ then $f$ is convex.

\item
If $f$ is increasing and we define $g\colon I\to\R$ by $g(x) = \int_a^xf(t)\,dt$, then $g$ is convex and $D_-g(x) = f(x-)$, $D_+g(x)=f(x+)$ whenever $x\in I$ is in the domain of $D_-g$, $D_+g$, respectively.
\end{enumerate}
\end{prop}

\begin{proof}
(1) \cite{Hor1994}, Theorem 1.1.9. (The assumption that $D_-f$ or $D_+f$ exists and is increasing can be weakened to say that one of the four Dini derivatives%
\footnote{These are the function $x\mapsto \liminf_{t\to x-}(f(t)-f(x))/(t-x)$ and the same with $\liminf$ replaced by $\limsup$ and/or $t\to x-$ replaced by $t\to x+$.  They are valued in $[-\infty,\infty]$. It isn't necessary to assume that the values are finite, though after the fact they are.}
is increasing on $I^\circ$.)

(2) The convexity of $g$ follows by \cite{RV}, Theorem A page 9. (That theorem assumes $I$ is open, but the proof of convexity on page 10 does not require that assumption.)
At any point $x\in I$ which is not the least element of $I$, if we redefine $f(x)$ to be $f(x-)$, then this revised $f$ is continuous from the left at $x$ and has the same integral as $f$ on intervals where $f$ is integrable. Hence $D_-g(x)=f(x-)$. Similarly $D_+g(x)=f(x+)$ if $x$ is not the largest element of $I$.
\end{proof}

\begin{prop}\label{p:translate.conv}
Let $f\colon \R\to\R$ be a convex function. Let $m\leq D_+f(a)$, and let $c>0$. Define $g\colon \R\to\R$ by
\[
g(x) = \begin{cases}
m(x-a) + f(a) & \text{when $x\leq a+c$} \\
f(x-c) + mc & \text{when $x\geq a+c$}.
\end{cases}
\]
Then: {\rm(a)} $g$ is convex; {\rm(b)} $D_+g(x)\leq D_+f(x)$ for $x\geq a$ and $f-g$ is increasing on $[a,\infty)$;
{\rm(c)} if $D_+f(x)>m$ for $x>a$, then $g(x)<f(x)$ for all $x > a$.
\end{prop}

\begin{rem}
Note that with $v=c(1,m)$ we have $g\!\mid_{[a+c,\infty)} = f\!\mid_{[a,\infty)}+v$ since for $x\geq a+c$,
\[
(x,g(x)) = (x-c,f(x-c)) + (c,mc).
\]
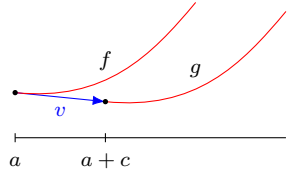
\[
\begin{tikzpicture}
\def\u{0.6cm}
    \draw (0,0) -- (6*\u,0);
    \coordinate (start) at (0,1*\u);
    \coordinate (p1) at (2*\u, 0.8*\u);
    \coordinate (end) at (4*\u,3*\u);
    \draw[blue,-Latex] (start) -- node [below] {\sz $v$} (p1);
    \coordinate (c1) at ($(start) + 1*(2*\u,-0.2*\u)$);
    \coordinate (c4) at ($(end) - 1*(1*\u,1.2*\u)$);
    \draw[red] (start) .. controls (c1) and (c4) .. (end);
    \fill (start) circle [radius=1pt];
    \node at (2*\u,1.7*\u) {\sz $f$};
    \begin{scope}[xshift = 2*\u,yshift=-0.2*\u]
        \coordinate (start) at (0,1*\u);
        \coordinate (p1) at (2*\u, 0.8*\u);
        \coordinate (end) at (4*\u,3*\u);
        \coordinate (c1) at ($(start) + 1*(2*\u,-0.2*\u)$);
        \coordinate (c4) at ($(end) - 1*(1*\u,1.2*\u)$);
        \draw[red] (start) .. controls (c1) and (c4) .. (end);
        \fill (start) circle [radius=1pt];
        \node at (2*\u,1.7*\u) {\sz $g$};
    \end{scope}
    \draw (0,-0.1*\u) -- ++(0,0.2*\u);
    \draw (2*\u,-0.1*\u) -- ++(0,0.2*\u);
    \node[below] at (0,-0.1*\u) {\sz \rule{0pt}{5pt}$a$};
    \node[below] at (2*\u,-0.1*\u) {\sz \rule{0pt}{5pt}$a+c$};
\end{tikzpicture}
\]
Moreover, the graph of $g$ follows the line through $(a,f(a))$ in the direction of the vector $v$ on the interval $[a,a+c]$: $g\!\mid_{[a,a+c]} = \{(a,f(a))+tc(1,m):0\leq t\leq 1\}$.
\end{rem}

\begin{proof}
(a) For $x < a+c$, $D_+g(x)=Dg(x)=m$ is constant. For $x\geq a+c$, $D_+g(x) = D_+f(x-c)$ is increasing. and $m \leq D_+f(a) = D_+g(a+c)$, so $D_+g$ is increasing.
Since, as follows readily from the formulas, $g$ is continuous, we conclude that $g$ is convex.

(b) To see that $D_+g\leq D_+f$ on $[a,\infty)$, first consider $a\leq x < a+c$. We have $D_+g(x) = m \leq D_+f(a)\leq D_+f(x)$. If $x\geq a+c$ then $D_+g(x) = D_+f(x-c)\leq D_+f(x)$. This gives $D_+(f-g)\geq 0$ on $[a,\infty)$ and hence $f-g$ is increasing on that interval (\cite{Hor1994}, Lemma 1.1.8).

(c) If $D_+f(x)>m$ for $x>a$, then in (b) we get $D_+g(x) = m < D_+f(x)$ when $a < x < a+c$. Together with $g(a)=f(a)$ and the fact from (b) that $D_+g\leq D_+f$ on $[a,\infty)$, this gives that for $x>a$,
\[
g(x) = g(a) + \int_a^x D_+g(t)\,dt < f(a) + \int_a^x D_+f(t)\,dt = f(x). \qedhere
\]
\end{proof}

We will make use of the proposition in the following forms which follow straightforwardly from the proposition.

\begin{cor}\label{c:translate.conv}
Let $f\colon\R\to\R$, $a\in\R$.
\begin{enumerate}
\item
Assume $f$ is concave. Let $m\geq D_+f(a)$, and let $c>0$. Define $g\colon\R\to\R$ by
\[
g(x) = \begin{cases}
m(x-a) + f(a) & \text{when $x\leq a+c$} \\
f(x-c)+mc & \text{when $x\geq a+c$}.
\end{cases}
\]
Then: {\rm (a)} $g$ is concave; {\rm (b)} $D_+f(x)\leq D_+g(x)$ for $x\geq a$ and $g-f$ is increasing on $[a,\infty)$; {\rm (c)} if $D_+f(x)<m$ for $x>a$, then $f(x)<g(x)$ for all $x > a$.

\item
Assume $f$ is convex. Let $m \geq D_-f(a)$, and let $c<0$. Define $g\colon\R\to\R$ by
\[
g(x) = \begin{cases}
f(x-c) + mc & \text{when $x\leq a+c$} \\
m(x-a) + f(a) & \text{when $x\geq a+c$}.
\end{cases}
\]
Then: {\rm (a)} $g$ is convex; {\rm (b)} $D_-f(x)\leq D_-g(x)$ for $x\leq a$ and $f-g$ is decreasing on $(-\infty,a]$; {\rm (c)} if $D_-f(x)<m$ for $x<a$, then $g(x)<f(x)$ for all $x < a$.
\end{enumerate}
\end{cor}

\n \textbf{C. A topology on the $C^\infty$ functions}.\ \
There is a natural and standard topological vector space structure on the family $C^\infty(I)$. (Cf.\ \cite{Ru3}, 1.46.) This is the strucure generated by the seminorms $f\mapsto \|D^jf\|_K := \sup\{|D^jf(x)|:x\in K\}$ ($=0$ if $K=\e$), where $j$ is a nonnegative integer and $K$ is a compact set in $I$.
The finite intersections of sets $V_I(K,j,\ep) = \{f\in C^\infty(I):\|D^jf\|_K<\ep\}$, for $\ep>0$, form a base of open neighborhoods at the origin (i.e., the zero function).

As in any topological vector space, the operations of addition and multiplication by a scalar are continuous.

\begin{exmp}
For each fixed $f_1,\dots,f_n\in C^\infty(\R)$, the map $\R^n\to C^\infty(\R)$ given by $(a_1,\dots,a_n)\mapsto a_1f_1+\dots+a_nf_n$ is continuous. In particular, if we associate to $a=(a_0,\dots,a_n)$ the polynomial $p_a(x)=a_0+\dots+a_nx^n$, then $a\mapsto p_a$ is continuous on $\R^{n+1}$.
\end{exmp}

We also need the following facts.

\begin{prop}\label{p:restr.cont}
For $a\in I$, the evaluation map $\phi_a\colon C^\infty(I)\to\R$ given by $\phi(f)=f(a)$ is continuous. For intervals $J\sq I$ of $\R$, the restriction map $C^\infty(I)\to C^\infty(J)$ is continuous.
\end{prop}

\begin{proof}
Evaluation and restriction are both linear, so it is enough to show that they are continuous at $0$. Given a neighborhood $(-\ep,\ep)$ of $0\in\R$, the image of $V_I(\{a\},0,\ep)$ under $\phi_a$ is contained in $(-\ep,\ep)$.
Given a neighborhood $V_J(K,j,\ep)$ of $0\in C^\infty(J)$, where $K$ is a compact subset of $J$. the image of $V_I(K,j,\ep)$ under the restriction map is contained in $V_J(K,j,\ep)$.
\end{proof}

\begin{prop}\label{p:mult.cont}
The operation of pointwise multiplication on $C^\infty(\R)$ is continuous.
\end{prop}

\begin{proof}
This is similar to the proof that multiplication on $\R$ is continuous. Fix $f_0,g_0\in C^\infty(\R)$, as well as a nonnegative integer $j$, a compact $K\sq\R$, and an $\ep>0$. Let $M$ be any upper bound on the numbers $\|D^if_0\|_K$, $\|D^ig_0\|_K$, $i=0,\dots,j$.
For any $f,g\in C^\infty(\R)$ and $\de>0$, if $f-f_0$ and $g-g_0$ belong to $\bigcap_{i=0}^j V(K,i,\de)$
then using the Leibniz formula we see that
$\|D^j(fg-f_0g_0)\|_K = \|D^j(fg)-D^j(f_0g_0)\|_K$ is bounded by
\[
\sum_{i=0}^j \binom{j}{i} (\|D^i(f-f_0)\|\|D^{j-i}(g-g_0)\| + \|D^i(f_0)\|\|D^{j-i}(g-g_0)\| + \|D^i(f-f_0)\|D^{j-i}(g_0)\|)
\]
which does not exceed $2^j(\de^2+2M\de)$ and hence is $<\ep$ if $\de$ is small enough.
\end{proof}

\begin{prop}\label{p:properties.of.tau}
The operation of composition on $C^\infty(\R)$
is continuous.
\end{prop}

\begin{proof}
For each nonnegative integer $m$, there is a polynomial
\[
P_m(x_0,\dots,x_m,y_1,\dots,y_m)
\]
with integer coefficients such that for any $f,g\in C^\infty(\R)$, writing $g\circ f$ as $g(f)$,
\[
D^m(g(f)) = P_m(g(f),(Dg)(f),\dots,(D^mg)(f),Df,\dots,D^mf).
\]
For example, we can take $P_0(x_0)=x_0$, $P_1(x_0,x_1,y_1) = x_1y_1$. Exact formulas for $D^m(g(f))$ are known (see for example the Fa\`{a} di Bruno formula \cite{KP1992}, Lemma 1.3.1.), but this property, which follows easily by induction on $m$, suffices for our purposes.

Given $f_0,g_0\in C^\infty(\R)$, a compact interval $K$, a nonnegative integer $m$ and $\ep>0$, let $L=[c,d]$ be a compact interval such that $D^jf_0(K)\sq L$, $j=0,\dots,m$. Let $L'=[c-1,d+1]$. By uniform continuity of the polynomials $P_j$ on $(L')^{2j+1}$, there is an $\eta>0$ such that for $j=0,\dots,m$,
\[
|P_j(x_0,\dots,x_j,y_1,\dots,y_j) - P_j(x'_0,\dots,x'_j,y'_1,\dots,y'_j)|<\ep
\]
whenever the arguments $x_i$, $x'_i$, $y_k$, $y'_k$ belong to $L'$ with $|x_i-x'_i|<\eta$, $|y_k-y'_k|<\eta$. We may take $\eta\leq 1$.

By uniform continuity of $D^jg_0$ on $L'$, there is a $\de>0$ such that
\[
|x-y|<\de \Rightarrow |D^jg_0(x)-D^jg_0(y)|<\eta/2
\]
for all $x,y\in L'$, $j=0,\dots,m$. We may take $\de\leq \eta\leq 1$. Now let $f,g\in C^\infty(\R)$ and suppose
\[
\|D^jf-D^jf_0\|_K<\de,\ \ \|D^jg-D^jg_0\|_{L'}<\de/2,\ \ j=0,\dots,m.
\]
Note that for $x,y\in L'$ and $j=0,\dots,m$, $|x-y|<\de$ implies
\begin{align*}
| & D^jg(x)-D^jg_0(y)| \leq |D^jg(x)-D^jg_0(x)| + |D^jg_0(x)-D^jg_0(y)| < \de/2 + \eta/2 \leq \eta.
\end{align*}
Thus, for any $j=0,\dots,m$ and $x\in K$, $D^jf_0(x)\in D^jf_0(K)\sq L$ and $|D^jf(x)-D^jf_0(x)|<\de\leq 1$, so $D^jf(x)$ and $D^jf_0(x)$ both belong to $L'$. In particular, $f(x)$ and $f_0(x)$ both belong to $L'$, so $|(D^jg)(f(x)) - D^jg_0(f_0(x))|<\eta$. By the choice of $\eta$ and the formula for $D^m(g(f))$, we get $\|D^m(g(f))-D^m(g_0(f_0))\|_K<\ep$.
\end{proof}

\begin{rem}
Given $n,k\in\N$, the operation of composition $C^\infty(\R^n)\times C^\infty(\R^k)^n\to C^\infty(\R^k)$ which to $g\in C^\infty(\R^n)$ and $f_i\in C^\infty(\R^k)$, $i=1,\dots,n$, associates $F\in C^\infty(\R^k)$ given by $F(x) = g(f_1(x),\dots,f_n(x))$ is continuous by a proof similar to the one above for the case $n=k=1$.%
\footnote{The topology on $C^\infty(\R^n)$ is generated by the seminorms $f\mapsto \|D^\al f\|_K := \sup\{|D^\al f(x)|:x\in K\}$ where $\al$ is an $n$-multi-index, $K$ is a compact set in $\R^n$, and for $\al=(\al_1,\dots,\al_n)$, $D^\al = (\partial/\partial x_1)^{\al_1}\dots (\partial/\partial x_n)^{\al_n}$.}
If we state continuity of composition in this form, then the continuity of multiplication from Proposition \ref{p:mult.cont} follows since $f_1f_2=g(f_1,f_2)$, where $g(x,y)=xy$ is multiplication on $\R$.
\end{rem}

\begin{prop}\label{p:properties.of.tau.2}
The derivative operator and the integral operators $I_a(f)(x)=\int_a^xf(t)\,dt$ on $C^\infty(\R)$ are continuous.
\end{prop}

\begin{proof}
Continuity of the derivative operator follows from $\|D^j(Df)\|_K = \|D^{j+1}f\|_K$ and linearity of the operator. For continuity of $I_a$, fix a compact interval $K$, and write $L$ for the convex hull of $K\cup \{a\}$, and $b$ for the diameter of $L$. Continuity of $I_a$ follows from its linearity along with the inequality $\|I_a(f)\|_K\leq b\|f\|_{L}$, which holds since for $x\in K$, $|I_a(f)(x)| = |\int_a^x f(t)\,dt|\leq \|f\|_{L}|x-a|$, as well as the equalities $\|D^jI_a(f)||_K = \|D^{j-1}f\|_K$ valid for $j\geq 1$.
\end{proof}

\begin{prop}\label{p:C.infty.by.cases}
Let $X$ be a topological space, $I$ a nontrivial interval of $\R$. Let $f,g\colon X\to C^\infty(I)$ be continuous, and let $a\colon X\to I^\circ$ be continuous. Write $f(x)=f_x$, $g(x)=g_x$, $a(x)=a_x$. Assume that for all nonnegative integers $j$ and all $x\in X$, $D^jf_x(a_x) = D^jg_x(a_x)$. Define $h\colon X\to C^\infty(I)$ by taking $h(x)(y) = h_x(y)$ for $y\in I$ to be
\[
h_x(y) = \begin{cases}
f_x(y) & \text{if $y\leq a_x$} \\
g_x(y) & \text{if $y\geq a_x$}.
\end{cases}
\]
Then $h$ is continuous.
\end{prop}

\begin{proof}
Since $D^jf_x(a_x) = D^jg_x(a_x)$ for all $j$, we have $h_x\in C^\infty(I)$. Let $x_0\in X$. Fix a subbasic open neighborhood $V=V_I(K,j,\ep)$ of the origin in $C^\infty(I)$.

Let $\de>0$ be such that for any $y\in I$, $|y-a_{x_0}|<\de$ implies both $|D^jf_{x_0}(y)-D^jf_{x_0}(a_{x_0})|<\ep/3$ and $|D^jg_{x_0}(y)-D^jg_{x_0}(a_{x_0})|<\ep/3$. Then let $U$ be an open neighborhood of $x_0$ in $X$ so that for all $x\in U$, $f_x-f_{x_0}\in V_I(K,j,\ep/3)$, $g_x-g_{x_0}\in V_I(K,j,\ep/3)$, and $|a_x - a_{x_0}|<\de$.

Now let $x\in U$ and $y\in K$.
We want to show that $|D^jh_x(y) - D^jh_{x_0}(y)| < \ep$.
Without loss of generality, $a_x\leq a_{x_0}$.

If $y\leq a_x$ then $|D^jh_x(y) - D^jh_{x_0}(y)| = |D^jf_x(y) - D^jf_{x_0}(y)| < \ep/3 < \ep$.

If $y\geq a_{x_0}$ then $|D^jh_x(y) - D^jh_{x_0}(y)| = |D^jg_x(y) - D^jg_{x_0}(y)| < \ep/3<\ep$.

If $a_x<y<a_{x_0}$ then $|y-a_{x_0}|<\de$, so since $D^jg_{x_0}(a_{x_0}) = D^jf_{x_0}(a_{x_0})$ we have
\begin{align*}
|D^jh_x(y) & - D^jh_{x_0}(y)| = |D^jg_x(y) - D^jf_{x_0}(y)| \\
& \leq |D^jg_x(y) - D^jg_{x_0}(y)| + |D^jg_{x_0}(y) - D^jg_{x_0}(a_{x_0})| + |D^jf_{x_0}(a_{x_0}) - D^jf_{x_0}(y)| \\
& < \ep/3 + \ep/3 + \ep/3 = \ep. \qedhere
\end{align*}
\end{proof}

\m

\n\textbf{D. Continuity of transversals.}\ \
We require for functions of two variables, for example $f\colon \R^2\to\R$, each of whose vertical sections $f_x = f(x,\,\cdot\,)$ take a particular value $a$ at a unique point $g(x)$, i.e., $f(x,g(x))=a$, to know that the transversal function $g$ is continuous under suitable assumptions on $f$. The following proposition gives sufficient conditions on $f$ for our purposes.

\begin{prop}\label{p:cont.of.sect}
Let $X$ be a topological space, and let $Y$ and $Z$ be linearly ordered spaces with their order topology. Let $E\sq X\times Y$ satisfy that each horizontal section $E^y$ is open in $X$ and each vertical section $E_x$ is an interval of $Y$. Let $f\colon E\to Z$ be continuous in the first variable and strictly increasing in the second. Let $a\in Z$. Let $X_0=\{x\in X:E_x\not=\e\}$. Suppose that for each $x\in X_0$, there is a $y$ such that $(x,y)\in E$, $y$ is not an endpoint of $E_x$, and $f(x,y)=a$. Then the function $X_0\to Y$ given by $x\mapsto f_x^{-1}(a)$ is continuous.
\end{prop}

\begin{proof}
It is enough to show that the preimage of each ``half-line'' of the form $\{y\in Y:y_0<y\}$ or $\{y\in Y:y<y_0\}$ under $x\mapsto f_x^{-1}(a)$ is open, so fix $y_0\in Y$. Let $x\in X_0$. Write $y=f_x^{-1}(a)$. Assume that $y_0 < y$. Since $E_x$ is an interval of $Y$ of which $y$ is not the least element, there is a $y_1\in E_x$ such that $y_0\leq y_1<y$. Then by assumption we have $f(x,y_1) < f(x,y)=a$. Let $U$ be an open neighborhood of $x$ such that $t\in U$ implies $(t,y_1)\in E$ and $f(t,y_1) < a$. Then for $t\in U$ we have that $E_t$ is nonempty (it contains $y_1$) and $f_t(y_1) < a$, so $y_0 \leq y_1 < f_t^{-1}(a)$. Similarly, if $y<y_0$ then there is an open neighborhood $U$ of $x$ such that $t\in U$ implies $f_t^{-1}(a)$ is defined and $< y_0$.
\end{proof}

The following example shows that the assumption that $y=f_x^{-1}(a)$ is not an endpoint of $E_x$ cannot be omitted.

\begin{exmp}
Consider $f(x,y)=xy$ on $E=[0,1]^2\sm (\{0\}\times[0,1))$ with $X=Y=[0,1]$, $a=0$. All assumptions are satisfied except that when $x=0$, $f_x^{-1}(0)=1$ is an endpoint of $E_x$. The function $x\mapsto f_x^{-1}(0)$ is not continuous at $0$ since it equals $0$ for $x>0$.
\end{exmp}

\n\textbf{E. A $C^\infty$ function.}\ \
In our constructions of $C^\infty$ functions, we shall make use of a $C^\infty$ function $h\colon [0,1]\to [0,1]$ having the properties listed in the following proposition.

\begin{prop}\label{p:h}
There is a $C^\infty$ function $h\colon [0,1]\to [0,1]$ having the following properties.
\noindent\begin{tabular}{@{}ll}
\parbox{7cm}{
\begin{enumerate}
\item
$D^nh(0)=0$ $(n\geq 0)$,

\item
$h(1)=1$, $D^nh(1)=0$ $(n\geq 1)$,

\item
$Dh(x)>0$ when $0<x<1$,

\item
$\int_0^1 h(t)\,dt = 1/2$.
\end{enumerate}}
&
\parbox{7cm}
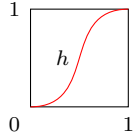
{
\begin{tikzpicture}
\def\u{1.3cm}
    \draw (0,0) -- ++(1*\u,0) -- ++(0,1*\u) -- ++(-1*\u,0) -- cycle;
    \coordinate (start) at (0,0);
    \coordinate (end) at (1*\u,1*\u);
    \coordinate (c1) at ($(start) + 0.7*(1*\u,0)$);
    \coordinate (c2) at ($(end) - 0.7*(1*\u,0)$);
    \draw[red] (start) .. controls (c1) and (c2) .. (end);
    \node[below left] at (0,0) {\sz $0$};
    \node[below] at (1*\u,0) {\sz $1$};
    \node[left] at (0,1*\u) {\sz $1$};
    \node[left] at (0.5*\u,0.5*\u) {\sz $h$};
 \end{tikzpicture}}
\end{tabular}
\end{prop}

\begin{proof}
Start with any $C^\infty$ function $f$ satisfying (1)--(3), for example, $f(0)=0$, $f(1)=1$, $f(x)=\exp(-(1/x)\exp(-1/(1-x)))$ for $0<x<1$.
Let $g$ be the function whose graph on $[0,1]$ is obtained by rotating the graph of $f$ by $180^\circ$ around the center of the unit square, i.e., when $f(x)=y$, we have $g(1-x) = 1-y$,
or $g(x) = 1-f(1-x)$.
Then $g$ is also a $C^\infty$ function satisfying (1)--(3).
The average of $f$ and $g$, namely the function $h=(f+g)/2$, is $C^\infty$ and satisfies (1)--(3). Moreover, $h$ is invariant under rotation of its graph by $180^\circ$ around the center of the unit square since
\begin{align*}
2(1-h(1-x)) & = 2 - 2h(1-x) = 2 - f(1-x) - g(1-x) \\
& = (1-f(1-x)) + (1-g(1-x)) = g(x) + f(x) = 2h(x)
\end{align*}
Thus, $\int_0^1 h(t)\,dt = 1/2$ since by the invariance under rotation, the area above the curve equals the area under the curve.
\end{proof}

\section{Specifying higher order derivatives at the endpoints}
\label{s:spec.higher}

For the remainder of the paper, for each $\de>0$, we fix continuous functions $\de\mapsto \s_\de$ and $\de\mapsto \tau_\de$ mapping positive numbers $\de$ into $C^\infty(\R)$. We denote the restriction of $\s_\de$ to $[0,\de]$ also by $\s_\de$, and we denote the restriction of $\tau_\de$ to $[1-\de,1]$ also by $\tau_\de$. These functions are required to satisfy the following conditions.
\begin{itemize}
\item
$\s_\de\colon [0,\de]\to [0,\de]$, $\s_\de(0)=0$, $D\s_\de>0$ on $(0,\de]$, and $\s_\de(x)=x$ when $\de/2\leq x\leq \de$.

\item
$\tau_\de\colon [1-\de,1]\to [1-\de,1]$, $\tau_\de(1)=1$, $D\tau_\de>0$ on $[1-\de,1)$, and $\tau_\de(x)=x$, $1-\de\leq x\leq 1-\de/2$.
\end{itemize}

Existence of such functions is trivial since we could take all $\s_\de$ and $\tau_\de$ to be the identity function, but we are interested in controlling also the values of the derivatives $D^j\s_\de(0)$ and $D^j\tau_\de(1)$. The following proposition shows that we could impose arbitrary values on these derivatives when $j\geq 1$, subject to the restriction that the first nonzero $D^j\s_\de(0)$ (if there is one) is positive, and the first nonzero $D^j\tau_\de(1)$ (if there is one) is positive if $j$ is odd, and negative if $j$ is even.

\begin{prop}\label{p:endpt.der}
Let $\al_0,\al_1,\dots$ be a sequence of real numbers such that $\al_0=0$ and
either $\al_j=0$ for all $j\geq 1$, or for the least $j\geq 1$ for which $\al_j\not=0$ we have $\al_j>0$.
Let $\beta_0,\beta_1,\dots$ be a sequence of real numbers such that $\beta_0=1$ and either $\beta_j=0$ for all $j\geq 1$, or for the least $j\geq 1$ for which $\beta_j\not=0$ we have $(-1)^{j+1}\beta_j>0$.

Then there are continuous functions $\de\mapsto \s_\de$ and $\de\mapsto \tau_\de$ from positive reals $\de>0$ into $C^\infty(\R)$ such that the following properties hold.
\begin{enumerate}
\item
$D^j\s_\de(0)=\al_j$ for all $j=0,1,2,\dots$, $D\s_\de(x)>0$ for $x>0$, and $\s_\de(x)=x$ when $x\geq \de/2$.

\item
$D^j\tau_\de(1)=\beta_j$ for all $j=0,1,2,\dots$, $D\tau_\de(x)>0$ for $x<1$, and $\tau_\de(x)=x$ when $x\leq 1-\de/2$.
\end{enumerate}
\end{prop}

\begin{proof}
The proof builds on the ideas used to prove \cite{Bu2019}, Proposition 6.2.
First we construct a continuous map $\de\mapsto \s_\de$ satisfying (1). By the theorem of Borel on the surjectivity of the derivative map, there is a $C^\infty$ function $u$ for which $D^ju(0)=\al_{j}$, $j=0,1,\dots$. If for all $j\geq 1$, $\al_j=0$, then take $u$ to be constant, $u(x)=\al_0=0$.
By our assumptions on the coefficients $\al_j$, $u(0)=\al_0=0$ and $Du(x)\geq 0$ for $x\geq 0$ close enough to $0$.%
\footnote{Proof. Case 1. $u=0$ or $Du(0)>0$. The claim
holds when $Du(0)>0$ as $Du$ is continuous. Case 2. $u\not=0$ and $Du(0)=0$. Then the least index $j\geq 0$ such that $\al_{j}\not=0$ is $>1$, and we have $D^ju(0)=\al_{j}>0$. Choose $\theta>0$ so that $D^ju(x)>0$, $0\leq x<\theta$. For $x\in(0,\theta)$, the Taylor formula gives $Du(x)=D^ju(\xi)x^{j-1}/(j-1)!$ for some $\xi\in (0,x)$. Hence, $Du(x)>0$.}
Fix $k\in\N$ satisfying $k\geq 2$ and $Du(0)<k$. Then $u(x)<kx$ for $x>0$ close enough to $0$. Since the statement of (1) continues to hold if we make $\de$ larger, it is enough to arrange it for $\de\mapsto \s_\de$ defined when $0<\de\leq \de_0$, where $\de_0$ is chosen so that $Du(x)\geq 0$ and $u(x)<kx$ when $0<x\leq \de_0$. (Then we can take $\s_\de=\s_{\de_0}$ for $\de\geq \de_0$.)

For $0<\de\leq \de_0$, set
\[
f(x)=f_\de(x)=u(x)+h\left(\frac{4kx}{\de}\right)((x+\de)/4 - u(x)).
\]
Note that $\de\mapsto f_\de$ is continuous. When $x\geq \de/(4k)$, we have $f(x)=(x+\de)/4$, so $Df(x)=1/4>0$. Let us verify that $Df(x)$ is positive when $0<x<\de/(4k)$.
\begin{align*}
Df(x) & = Du(x)+\frac{4k}{\de}Dh\left(\frac{4kx}{\de}\right)((x+\de)/4 - u(x))
+ h\left(\frac{4kx}{\de}\right)(1/4-Du(x)) \\
& =  Du(x)\left(1-h\left(\frac{4kx}{\de}\right)\right)+\frac{4k}{\de}Dh\left(\frac{4kx}{\de}\right)((x+\de)/4 - u(x))
+ \frac14 h\left(\frac{4kx}{\de}\right)
.\end{align*}
Using the fact that $u(x) < kx < \de/4$, we see that all terms are nonnegative and the last is positive, so $Df(x)>0$. Because the derivatives of all orders of $h$ are zero at the origin, we get $D^jf(0)=D^ju(0)=\al_{j}$ for all $j=0,1,2,\dots$. The graph of $y=f(x)$, for $x\geq \de/(4k)$, meets the diagonal when $(x+\de)/4 = x$, i.e., when $x=\de/3$.

We consider the $C^\infty$ function $G_{\gamma_\de}$, with $\gamma_\de = (1/4,1,\de/3,\de/3,\de/12)$, which on the interval $[\de/8,\de/2]\sq [\de/(4k),\de/2]$, equals $(x+\de)/4$ to the left of the interval $J=(\de/3-\de/12,\de/3+\de/12) = (\de/4,5\de/12)$, and equals $x$ to the right of $J$, and whose derivative on $J$ increases monotonically from $1/4$ to $1$ (Proposition \ref{roundoffcorners}). Take $\s_\de(x) = f_\de(x)$ when $x\leq \de/5$, $\s_\de(x)=G_{\gamma_\de}(x)$ when $x\geq \de/5$. Since both functions equal $(x+\de)/4$ in a neighborhood of $\de/5$, it follows from Proposition \ref{p:C.infty.by.cases} that $\de\mapsto \s_\de$ is continuous.

To construct $\de\mapsto \tau_\de$, apply (1) to the sequence $\al_0=0$, $\al_j=(-1)^{j+1}\beta_j$ for $j\geq 1$, to get a continuous map $\de\mapsto g_\de$ such that for $\de>0$, $g=g_\de$ satisfies $g(0)=0$, $D^jg(0)=\al_j=(-1)^{j+1}\beta_j$ for $j\geq 1$, $Dg(x)>0$ for $x>0$, and $g(x)=x$ for $x\geq\de/2$.
Take $\tau_\de(x) = -g_\de(1-x)+1$. The map $\de\mapsto \tau_\de$ is continuous, and for each $\de>0$, for  $\tau=\tau_\de$ we have $\tau(1)=1$. For any $j\geq 1$, $D^j\tau(x) = (-1)^{j+1}D^jg(1-x)$, so $D^j\tau(1) = (-1)^{j+1}D^jg(0) = \beta_j$. When $x<1$, we have $1-x>0$, so $D\tau(x) = Dg(1-x)>0$. And when $x\leq 1-\de/2$, we have $1-x\geq \de/2$, so $\tau(x) = -g(1-x)+1 = -(1-x) + 1 = x$.
\end{proof}

The simple observation in the following example will be useful later.

\begin{exmp}\label{n=0}
For each $\de$ with $0<\de<1/2$, there is a $C^\infty$ function $f\colon [0,1]\to[0,1]$ such that $f=\s_\de$ on $[0,\de]$, $f=\tau_\de$ on $[1-\de,1]$, and $Df>0$ on $(0,1)$.
\end{exmp}

\begin{proof}
Define $f=\s_\de$ on $[0,\de]$, $f=\tau_\de$ on $[1-\de,1]$, $f(x)=x$ on $[\de,1-\de]$.
\end{proof}

By a simple rescaling of the function in the example, we get the following.

\begin{prop}[\cite{Bu2019}, Proposition 6.2]
Let $\al=(\al_0,\al_1,\dots)$ and $\beta=(\beta_0,\beta_1,\dots)$ be sequences of real numbers such that $\al_0<\beta_0$ and either $\al_j=0$ for all $j\geq 1$, or for the least $j\geq 1$ for which $\al_j\not=0$ we have $\al_j>0$; and either $\beta_j=0$ for all $j\geq 1$, or for the least $j\geq 1$ for which $\beta_j\not=0$ we have $(-1)^{j+1}\beta_j>0$.

Then for any interval $[a,b]$, $a<b$, there is a $C^\infty$ function $f\colon [a,b]\to\R$ such that $Df(x)>0$ for all $x\in (a,b)$, and $D^jf(a)=\al_j$ and $D^jf(b)=\beta_j$ for $j=1,2,\dots$.
\end{prop}

\begin{proof}
Get $g\colon[0,1]\to[0,1]$ from Example \ref{n=0} using $\s_\de$ and $\tau_\de$ obtained from Proposition \ref{p:endpt.der} applied to the sequences $\al^*_j = (b-a)^j\al_j$ and $\beta^*_j = (b-a)^j\beta_j$.
Let $f(x) = g((x-a)/(b-a))$, $a\leq x\leq b$. For $j=0,1,2,\dots$, we have $D^jf(x)=(b-a)^{-j}D^jg((x-a)/(b-a))$, so $D^jf(a)=(b-a)^{-j}D^jg(0)=(b-a)^{-j}\al_j^*=\al_j$, $D^jf(b)=(b-a)^{-j}D^jg(1)=(b-a)^{-j}\beta_j^*=\beta_j$.
\end{proof}

\section{The families $V_n[c,d]$ and $W_n[c,d]$}
\label{s:endpoint.values}

In this section we examine in more detail the assumption $(P_n)$ mentioned in the introduction and its associated sets.
For $a=(a_0,\dots,a_n)$ and $b=(b_0,\dots,b_n)$ in $\R^{n+1}$, let $(a;b)$ denote the concatenation $(a;b) = (a_0,\dots,a_n,b_0,\dots,b_n)\in \R^{2(n+1)}$. We also write this tuple as $(a_j\,;\,b_j)$ when $n$ is clear from the context.
We write
\begin{align*}
\F_n[c,d] & = \{f\in C^n[c,d]: D^nf\ \text{is increasing but not constant}\}, \\
\F_n^\infty[c,d] & = \{f\in C^\infty[c,d]: D^{n+1}f(x)>0\ \text{for all}\ x\in (c,d)\},
\end{align*}
and set $\F_n=\F_n[0,1]$, $\F_n^\infty=\F^\infty_n[0,1]$.

\begin{rem}\label{r:signs}
If $f\in\F_n^\infty[c,d]$ has $D^jf(c)=\al_j$ and $D^jf(d)=\beta_j$ for all nonnegative integers $j$, then for $j>n$, the first nonzero $\al_j$, if there is one, must be positive, and the first nonzero $\beta_j$, if there is one, must be positive if $n+j$ is odd and negative if $n+j$ is even.

\begin{proof}
We can see this from the Taylor theorem. Suppose $j>n$ is least with $D^jf(c)\not=0$, where $f$ is $C^\infty$. Then for $x\in [c,d]$ we have
\[
D^nf(x) = \sum_{i=n}^{j-1} \frac{D^{i}f(c)}{(i-n)!}(x-c)^{i-n} + \frac{D^jf(\xi_x)}{j!}(x-c)^{j-n} = D^{n}f(c) + \frac{D^jf(\xi_x)}{j!}(x-c)^{j-n}
\]
for some $\xi_x$ with $c<\xi_x<x$. Since $D^jf$ is continuous, if $D^jf(c)$ is negative, then $D^nf(x)<D^nf(c)$ for $x$ close enough to $c$.
At $d$ we can argue similarly, or apply the result at the left endpoint to the function $g(x) = (-1)^{n+1}f(-x)$ defined on $[-d,-c]$. We have the formula $D^jg(x) = (-1)^{j+n+1}D^jf(-x)$. It follows that $D^ng(x) = -D^nf(-x)$ is increasing if and only if $D^nf(x)$ is, and $D^jf(d)=\beta_j$ if and only if $D^jg(-d)=(-1)^{j+n+1}\beta_j$. If $j>n$ is least with $\beta_j\not=0$ then, the increasing nature of $D^nf$ forces that of $D^ng$ and hence forces $(-1)^{j+n+1}\beta_j>0$.
\end{proof}
\end{rem}

Let $\mathcal{S}$ denote the set of infinite sequences $\al=(\al_0,\al_1,\al_2,\dots)$ of real numbers, and for $x=(x_0,\dots,x_n)\in\R^{n+1}$, set
\begin{align*}
\A_n(x) & = \{\al\in\mathcal{S}: \al_j=x_j,\, j=0,\dots,n,\ \text{and either}\ \al_j=0\ \text{for all}\ j>n, \\
& \rule{2.02cm}{0cm}\text{or for the least $j>n$ such that $\al_j\not=0$, we have $\al_j>0$}\}, \\
\B_n(x) & = \{\beta\in\mathcal{S}: \beta_j=x_j,\, j=0,\dots,n,\ \text{and either}\ \beta_j=0\ \text{for all}\ j>n, \\
& \rule{2.02cm}{0cm}\text{or for the least $j>n$ such that $\beta_j\not=0$, we have $(-1)^{n+j+1}\beta_j>0$}\}.
\end{align*}
We also make use of sequences indexed starting at $1$, so we let
$\wt{\mathcal{S}}$ denote the set of infinite sequences $\al=(\al_1,\al_2,\al_3,\dots)$ of real numbers, and we set
\begin{align*}
\wt{\A} & = \{\al\in\wt{\mathcal{S}}: \text{either}\ \al_j=0\ \text{for all}\ j,\ \text{or for the least $j$ such that $\al_j\not=0$, we have $\al_j>0$}\}, \\
\wt{\B} & = \{\beta\in\wt{\mathcal{S}}: \text{either}\ \beta_j=0\ \text{for all}\ j,\ \text{or for the least $j$ such that $\beta_j\not=0$, we have $(-1)^{j+1}\beta_j>0$}\}.
\end{align*}

We define the following subsets of $\R^{2(n+1)}$.
\begin{align*}
V_n[c,d] & = \{(a;b): a,b\in\R^{n+1}\ \text{and there is an $f\in \F_n[c,d]$} \\
& \rule{1.9cm}{0cm}\text{such that $D^jf(c)=a_j$ and $D^jf(d)=b_j$ for all $j=0,\dots,n$}\} \\
V_n^\infty[c,d] & = \{(a;b): a,b\in\R^{n+1}\ \text{and for all $\al\in\A_n(a)$, $\beta\in\B_n(b)$ there is an $f\in \F_n^\infty[c,d]$} \\
& \rule{1.9cm}{0cm}\text{such that $D^jf(c)=\al_j$ and $D^jf(d)=\beta_j$ for all $j=0,1,2,\dots$}\}
\end{align*}
We set $V_n=V_n[0,1]$, $V^\infty_n=V_n^\infty[0,1]$.
We shall see (Proposition \ref{p:equiv.P_n}) that $(P_n)$ is equivalent to the statement that $V_n[c,d] = V_n^{\infty}[c,d]$ and $V_n[c,d]$ is open in $\R^{2(n+1)}$ for all $c<d$.
Working toward a proof of this equivalence, we establish some properties of the families $V_n[c,d]$.

\begin{rem}
Using witnessing functions with $D^nf$ decreasing instead of increasing in the definitions simply negates the family $V_n[c,d]$. More precisely we have the following.

(a) If $\F_n^*[c,d]$ is obtained from $\F_n[c,d]$ by saying that $D^nf$ is decreasing instead of increasing, then $\F_n^*[c,d]=-\F_n[c,d] = \{-f:f\in \F_n[c,d]\}$, and replacing $\F_n[c,d]$ by $\F_n^*[c,d]$ in the definition of $V_n[c,d]$ gives the family $V^*_n[c,d] = -V_n[c,d]$.

(b) Similarly, if $\A_n^*(a)$, $\B_n^*(b)$, $(\F_n^\infty)^*[c,d]$ are obtained from $\A_n(a)$, $\B_n(b)$, $\F_n^\infty[c,d]$ by saying that $\al_j<0$, $(-1)^{n+j+1}\beta_j<0$, and $D^{n+1}f(x)<0$ instead of $\al_j>0$, $(-1)^{n+j+1}\beta_j>0$ and $D^{n+1}f(x)>0$, respectively, then $(\F_n^\infty)^*[c,d] = -\F_n^\infty[c,d]$, and replacing $\A_n(a)$, $\B_n(b)$, $\F_n^\infty[c,d]$ by $\A_n^*(a)$, $\B_n^*(b)$, $(\F_n^\infty)^*[c,d]$ in the definition of $V_n^\infty[c,d]$ gives the family $(V_n^\infty)^*[c,d] = -V_n^\infty[c,d]$.
\end{rem}

\begin{rem}
If we define a new family of $2(n+1)$-tuples by deleting ``but not constant'' from the definition of $V_n[c,d]$, then the new family has the form $V_n[c,d]\cup\{(v,\psi(v)):v\in \R^{n+1}\}$, where $\psi\colon \R^{n+1}\to \R^{n+1}$ is a linear isomorphism.
To see this, note that if $D^nf$ is constant, then $f$ is a polynomial of degree at most $n$. In that case, the values $a_i=D^if(c)$, $0\leq i\leq n$, are arbitrary, and they uniquely determine $f$, namely $f(x)=\sum_{i=0}^n(a_i/i!)(x-c)^i$, and hence uniquely determine the values $b_j=D^jf(d) = \sum_{i=j}^n(a_i/i!)(\prod_{0\leq k<j}(i-k))(d-c)^{i-j}$, $0\leq j\leq n$. Since the same statement is true with the roles of `$c$, $(a_i)$' and `$d$, $(b_j)$' interchanged, the map $\psi\colon \R^{n+1}\to \R^{n+1}$ given by $\psi((a_i)_{0\leq i\leq n}) = (b_j)_{0\leq j\leq n}$ is a bijection. Linearity of $\psi$ is clear from the formula for $b_j$ above.
\end{rem}

We now reduce our examination of $V_n[c,d]$, $V_n^\infty[c,d]$ to the examination of the families $W_n$, $W_n^\infty$ defined below.
\begin{align*}
W_n[c,d] & = \{b\in\R^{n+1}: \text{there is an $f\in \F_n[c,d]$ such that} \\
& \rule{2.5cm}{0cm}\text{$D^jf(c)=0$ and $D^jf(d)=b_j$ for all $j=0,\dots,n$}\} ,\\
W_n^\infty[c,d] & = \{b\in\R^{n+1}: \text{for all $\al\in\A_n(0)$, $\beta\in\B_n(b)$ there is an $f\in \F_n^\infty[c,d]$ such that} \\
& \rule{2.5cm}{0cm}\text{$D^jf(c)=\al_j$ and $D^jf(d)=\beta_j$ for all $j=0,1,2,\dots$}\}.
\end{align*}
We set $W_n=W_n[0,1]$, $W^\infty_n=W_n^\infty[0,1]$. In later sections we work with integration rather than differentiation, so it will be useful to rephrase the definitions of $W_n$ and $W^\infty_n$ in terms of integration.

\begin{prop}\label{p:W.equiv}
We have
{\rm
\begin{align*}
W_n & = \{b\in\R^{n+1}: \text{there is a $g\in \F_0$ such that $g(0)=0$ and $I^jg(1)=b_{n-j}$ for all $j=0,\dots,n$}\} ,\\
W_n^\infty & = \{b\in\R^{n+1}: \text{for all $\al\in\wt{\A}$, $\beta\in\wt{\B}$ there is a $g\in \F_0^\infty$ such that $g(0)=0$,} \\
& \rule{0.7cm}{0cm}\text{$I^jg(1)=b_{n-j}$ for all $j=0,\dots,n$, and $D^jg(0)=\al_j$, $D^jg(1)=\beta_j$ for $j=1,2,3,\dots$}\}.
\end{align*}
}
\end{prop}

\begin{proof}
Use the correspondence $g=D^nf$, $f=I^ng$ between the witnessing functions.
\end{proof}

For $\lb=(\lb_j)=(\lb_0,\dots,\lb_n)$, write $D_\lb\colon\R^{n+1}\to\R^{n+1}$ for the diagonal operator
\[
D_\lb(x_0,\dots,x_n)=(\lb_0x_0,\dots,\lb_nx_n).
\]
For $a=(a_0,\dots,a_n)$, write $p_a(x)=\sum_{k=0}^na_kx^k/k!$. Let $T_n\colon \R^{2(n+1)}\to\R^{n+1}$ be the linear transformation given by
\[
T_n(a_j\,;\,b_j) = (b_j-D^jp_a(1)).
\]
Let $\pi_n\colon \R^{2(n+1)}\to \R^{n+1}$ be the projection onto the second half of the coordinates, $\pi_n(u;v)=v$ where $u,v\in\R^{n+1}$. Note that for vectors $\lb_1,\lb_2\in\R^{n+1}$, setting $\lb=(\lb_1;\lb_2)\in \R^{2(n+1)}$, we have the relation $\pi_n D_\lb = D_{\lb_2}\pi_n$. Finally, let $H_n$ denote the subspace of $\R^{2(n+1)}$ consisting of vectors whose first $n+1$ coordinates are zero: $H_n = \{(0;v):v\in\R^{n+1}\}$ where $0$ is the zero vector of $\R^{n+1}$.

\begin{prop}\label{p:scaling hat W_n}
The following relations hold, as well as the same relations with a superscript $\infty$ added to all $V$'s and $W$'s.
\begin{enumerate}
\item
$W_n[c,d]=\pi_n(V_n[c,d]\cap H_n)$.

\item
$V_n[c,d]=D_\lb(V_n)$, where $\lb=((d-c)^{-j}\,;\,(d-c)^{-j})$.

\item
$W_n[c,d]=D_\lb(W_n)$,
where $\lb=(\,(d-c)^{-j}:j=0,\dots,n)$.

\item
$V_n=T_n^{-1}(W_n)$.
\end{enumerate}
\end{prop}

\begin{proof}
(1) Clear from the definitions.

(2) If the $C^n$ function $f\colon[c,d]\to\R$ witnesses $v=(a_j;b_j)\in V_n[c,d]$, then $g(x)=f(c+x(d-c))$ is a $C^n$ function on $[0,1]$. For $j=0,\dots,n$, we have
\[
D^jg(x) = (d-c)^jD^jf(c+x(d-c))
\]
so $D^jg(0)=(b-a)^ja_j$, $D^jg(1)=(b-a)^jb_j$, $j=0,\dots,n$, and $D^ng$ is increasing but not constant.
Hence the vector $w = ((b-a)^ja_j;(b-a)^jb_j)$ belongs to $V_n$, and we have $D_\lb(w)=v$.

In the case $n=\infty$, we want that $w\in V^\infty_n$, so we let $\al_j,\beta_j$ be as in the definition. Then the sequences $(d-c)^{-j}\al_j,(d-c)^{-j}\beta_j$ satisfy the conditions with respect to $v$, so there is a $C^\infty$ function $f\colon[c,d]\to\R$ with $D^jf(c)=(d-c)^{-j}\al_j$, $D^jf(d)=(d-c)^{-j}\beta_j$, $j=0,1,\dots$, and
$D^{n+1}f(x)>0$, $c<x<d$. The function $g(x)=f(c+x(d-c))$ is $C^\infty$ on $[0,1]$. For $j=0,\dots,n$, we have
\[
D^jg(x)=(d-c)^jD^jf(c+x(d-c))
\]
so $D^jg(0)=\al_j$, $D^jg(1)=\beta_j$, $D^{n+1}g(x)>0$, $0<x<1$, and hence $w\in W_n^\infty$.

Conversely given a $C^n$ function $f\colon [0,1]\to\R$ witnessing that $w=(a_j;b_j)\in V_n$, the reader can verify that the function $g(x)=f((x-c)/(d-c))$ on $[c,d]$ witnesses that
the vector $v = (a_j(d-c)^{-j};b_j(d-c)^{-j})$ belongs to $V_n[a,b]$ and we have $D_\lb(w)=v$. Similarly, if we add the superscripts $\infty$.

(3) With $\lb = (\lb_1;\lb_2)$ as in (3), where $\lb_1=\lb_2=(\,(d-c)^{-j}:j=0,\dots,n)$, we have $W_n[c,d] = \pi_n(V_n[c,d]\cap H_n) = \pi_n(D_\lb(V_n)\cap H_n) = \pi_n(D_\lb(V_n\cap H_n)) = D_{\lb_2}\pi_n(V_n\cap H_n) = D_{\lb_2}W_n$.

(4) Suppose $f\colon[0,1]\to \R$ witnesses that $(a_j\,;\,b_j)\in V_n$. Let $g(x)=f(x)-p_a(x)$. Then $g$ is a $C^n$ function with $D^ng(x)=D^nf(x)-a_n$ increasing but not constant since $D^nf$ is increasing but not constant. Also, for $j=0,\dots,n$, $D^jg(0)=D^jf(0)-a_j=a_j-a_j=0$ and $D^jg(1)=b_j-D^jp_a(1)$. Thus, $T_n(a_j\,;\,b_j)=(b_j-D^jp_a(1))\in W_n$. Conversely, suppose $T_n(a_j\,;\,b_j)=(b_j-D^jp_a(1))\in W_n$ with witnessing function $g$. Set $f(x)=g(x)+p_a(x)$. This is a $C^n$ function with $D^nf(x)=D^ng(x)+a_n$ increasing but not constant. For $j=0,\dots,n$, $D^jf(0)=D^jg(0)+a_j=a_j$ and $D^jf(1)=D^jg(1)+D^jp_a(1)=b_j$. Hence, $(a_j\,;\,b_j)\in V_n$. Similarly, we get $V_n^\infty=T_n^{-1}(W_n^\infty)$. (Note that $D^jf(x)=D^jg(x)$ when $j>n$.)
\end{proof}

\begin{prop}\label{p:equiv.P_n}
The following statements are equivalent.
\begin{enumerate}
\item
$(P_n)$

\item
$W_n = W_n^{\infty}$ and $W_n$ is open in $\R^{n+1}$.

\item
$V_n[c,d] = V_n^{\infty}[c,d]$ and $V_n[c,d]$ is open in $\R^{2(n+1)}$ for all $c<d$.
\end{enumerate}
\end{prop}

\begin{proof}
The equivalence of (2) and (3) follows easily from Proposition \ref{p:scaling hat W_n}. Also, (2) clearly implies (1), so there remains to show that (1) implies (2). First we establish the following claim.

\begin{claim}\label{claim:w.u}
Let $\de$ satisfy $0<\de<1/2$. Suppose $w\colon [0,1]\to [0,1]$ is increasing and satisfies $w(x)=x$ for $x\in\{0,\de,1-\de,1\}$. Let $\s$ and $\tau$ denote the restrictions of $w$ to $[0,\de]$ and $[1-\de,1]$, respectively.
If the values of $w$ on $[\de,1-\de]$ are obtained from an increasing function $u\colon [0,1]\to[0,1]$ satisfying $u(0)=0$ and $u(1)=1$, using the formula
\[
w(x) = (1-2\de)u\left(\frac{x-\de}{1-2\de}\right) + \de,\ \ \de\leq x\leq 1-\de,
\]
then the numbers $I^n w(1)$ and $I^nu(1)$ satisfy $I^0w(1)=w(1)=1$, $I^0u(1)=u(1)=1$, and, for $n\geq 1$, satisfy linear equations
\[
I^nw(1) = r_n + (1-2\de)^{n+1}I^nu(1) + \sum_{k=1}^{n-1}s_kI^ku(1),
\]
where $r_n$ and the coefficients $s_k$ depend only on $\s$ and $\tau$, and satisfy $0\leq s_k\leq \de$ and $0\leq r_n\leq d_n\de$, where $d_1=3$, $d_{n+1}=d_n+n+2$.
\end{claim}

\begin{proof}
We first verify that we have for each nonnegative integer $n$,
\[
I^nw(x) = \begin{cases}
T^n_1(x), & \text{$0\leq x\leq \de$} \\
T^n_2(x) + I^n_\de w(x), & \text{$\de\leq x\leq 1-\de$} \\
T^n_3(x) + \sum_{k=0}^{n-1} I^{n-k}_\de w(1-\de)\bigl(x - (1-\de)\bigr)^k/k!, & \text{$1-\de\leq x\leq 1$}
\end{cases}
\]
where $T^n_1,T^n_2,T^n_3$ are nonnegative functions on $[0,\de]$, $[\de,1-\de]$, $[1-\de,1]$, respectively, which
depend only on $\s$ and $\tau$, and
satisfy $T^n_1(x)\leq \de$, $T^n_2(x)\leq n\de$, and, for $n\geq 1$, $T^n_3(x)\leq c_n\de$, where $c_1=2$, and $c_{n+1}=c_n+(n+1)$.

We proceed by induction on $n$. When $n=0$, the formula in the middle clause is just $w(x)=I_\de^0w(x)$ with $T^0_2(x)=0$, and the first and third clause formulas are $T^0_1(x) = \s(x)$ and $T^0_3(x) = \tau(x)$, respectively. For the inductive step, given the formula for $n$,
when $0\leq x\leq \de$ we have $I^{n+1}w(x) = T^{n+1}_1(x) = IT^n_1(x)$. For $\de\leq x\leq 1-\de$ we have
\begin{align*}
I^{n+1}w(x) & = T^{n+1}_1(\de) + I_\de T^n_2(x) + I^{n+1}_\de w(x) \\
& = T^{n+1}_2(x) + I^{n+1}_\de w(x),
\end{align*}
where $T_2^{n+1}(x) = T_1^{n+1}(\de) + I_\de T^n_2(x)\leq \de + n\de = (n+1)\de$.
And for $1-\de\leq x\leq 1$, we have
\begin{align*}
I^{n+1}w(x) & = T^{n+1}_2(1-\de) + I^{n+1}_\de w(1-\de) + I_{1-\de}T^n_3(x) + \sum_{k=0}^{n-1} I^{n-k}_\de w(1-\de)\frac{(x - (1-\de))^{k+1}}{(k+1)!} \\
& = T^{n+1}_3(x) + \sum_{k=0}^{n} I^{(n+1)-k}_\de w(1-\de)\frac{(x - (1-\de))^k}{k!}
\end{align*}
When $n=0$, $T^{n+1}_3(x) = T^1_3(x) = T^{1}_2(1-\de) + I_{1-\de}T^0_3(x) = T^{1}_2(1-\de) + I_{1-\de}\tau(x)\leq 2\de = c_{n+1}\de$, and for $n\geq 1$,
\[
T^{n+1}_3(x) = T^{n+1}_2(1-\de) + I_{1-\de}T^n_3(x) \leq (n+1)\de + c_n\de = c_{n+1}\de.
\]
Next, by induction on nonnegative integers $n$, we have that for $\de\leq x\leq 1-\de$,
\[
I_\de^n w(x) = (1-2\de)^{n+1} I^n u\left(\frac{x-\de}{1-2\de}\right) + \de\frac{(x-\de)^n}{n!}.
\]
For $n=0$, this is just the given relationship between $w$ and $u$, and if the formula holds for $n$, then
\begin{align*}
I_\de^{n+1} w(x) & = (1-2\de)^{n+1} \int_\de^x I^n u\left(\frac{t-\de}{1-2\de}\right)\,dt + \de\frac{(x-\de)^{n+1}}{(n+1)!} \\
& = (1-2\de)^{n+2} \int_0^{(x-\de)/(1-2\de)} I^n u(s)\,ds + \de\frac{(x-\de)^{n+1}}{(n+1)!} \\
& = (1-2\de)^{n+2} I^{n+1} u\left(\frac{t-\de}{1-2\de}\right) + \de\frac{(x-\de)^{n+1}}{(n+1)!}
\end{align*}
Taking $x=1$ in the formula for $I^nw(x)$, for $n\geq 1$, gives
\[
I^nw(1) = T^n_3(1) + I^n_\de w(1-\de) + \sum_{k=1}^{n-1} I^{n-k}_\de w(1-\de)\frac{\de^k}{k!}.
\]
Taking $x=1-\de$ in the formula for $I_\de^n w(x)$ gives
\[
I_\de^n w(1-\de) = (1-2\de)^{n+1} I^n u(1) + \de\frac{(1-\de)^n}{n!}.
\]
Substituting the latter into the former, we get
\begin{align*}
I^nw(1) & = T^n_3(1) + (1-2\de)^{n+1} I^n u(1) + \de\frac{(1-\de)^n}{n!} + \sum_{k=1}^{n-1} (1-2\de)^{n-k+1}I^{n-k} u(1)\frac{\de^k}{k!} + \de\frac{(1-\de)^{n-k}}{(n-k)!}\frac{\de^k}{k!} \\
& = r_n + (1-2\de)^{n+1} I^n u(1) + \sum_{k=1}^{n-1} s_kI^k u(1),
\end{align*}
where each $s_k$ depends only on $\de$ and satisfies $0\leq s_k\leq \de$, and $r_n$ depends only on $\s$ and $\tau$ and satisfies
\[
0\leq r_n = T^n_3(1) + \de\frac{(1-\de)^n}{n!} + \sum_{k=1}^{n-1} \de\frac{(1-\de)^{n-k}}{(n-k)!}\frac{\de^k}{k!} \leq c_n\de + \de + (n-1)\de = d_n\de,
\]
where $d_n = c_n + n$, so $d_1=c_1+1=3$ and $d_{n+1} = c_{n+1}+(n+1) = c_n+(n+1)+(n+1) = d_n+(n+2)$.
\end{proof}

Now returning to the proof that (1) implies (2), assume (1). In (2), that $W_n$ is open in $\R^{n+1}$ is part of $(P_n)$, and $W_n^\infty\sq W_n$ is clear, so we must show that $W_n\sq W_n^\infty$.

Let $b=(b_0,\dots,b_n)\in W_n$.
We use Proposition \ref{p:W.equiv}.
Let $\al\in\wt{\A}$, $\beta\in\wt{\B}$. We must find an $f\in \F_0^\infty$ such that $f(0)=0$, $I^jf(1)=b_{n-j}$, $j=0,\dots,n$, and $D^jf(0)=\al_j$, $D^jf(1)=\beta_j$ for all $j=1,2,\dots$.

Fix maps $\de\mapsto\s_\de$ and $\de\mapsto\tau_\de$ obtained from Proposition \ref{p:endpt.der} using $\al_0=0$ with $b_n^{-1}\al_{j}$ in the place of $\al_j$ for $j\geq 1$, and $\beta_0=1$ with $b_n^{-1}\beta_{j}$ in the place of $\beta_j$ for $j\geq 1$.

From the claim, taking $\s=\s_\de$, $\tau=\tau_\de$ (and any $u$, for example $u(x)=x$), we get the coefficients $r_n$ and $s_k$ depending only on $\de$ (and not on $u$), and $r_n,s_k\to 0$ as $\de\to 0$.

Since $W_n$ is closed under scaling by positive constants,%
\footnote{In fact, $W_n$ is a convex cone in the sense of \cite[ \S 27]{Be1974}, i.e., is closed under taking linear combinations with positive coefficients.}
we have $b_n^{-1}b = (b_n^{-1}b_0,\dots,b_n^{-1}b_n)\in W_n$. Solve the system of linear equations for $q_0,\dots,q_{n-1}$,
\[
b_n^{-1}b_{n-j} = r_j + (1-2\de)^{j+1}q_{n-j} + \sum_{k=1}^{j-1}s_kq_{n-k},\ \ j=1,\dots,n,
\]
and take $q_n=1$. Since $W_n$ is open by $(P_n)$ and, from the equations above, we see that $q\to b_n^{-1}b$ as $\de\to 0$, we get $q=(q_0,\dots,q_n)\in W_n$ if $\de$ is small enough. Fix such a $\de$.

By $(P_n)$, there is a $g\in \F^\infty_n$ such that $D^jg(0)=0$, $D^jg(1)=q_j$, $j=0,\dots,n$, $D^{n+1}g(0)=D^{n+1}g(1)=1$, $D^{j}g(0)=D^{j}g(1)=0$ for $j>n+1$.

Define $u=D^ng$, and let $w$ be obtained from $u$ as in Claim \ref{claim:w.u}, with $\s=\s_\de$ and $\tau=\tau_\de$. Let $f=b_nw$.

We see that $w$ is $C^\infty$ by checking that at $\de$ and $1-\de$, the derivatives from the left and right all agree. At $\de$, on the left $w=\s_\de$ and we have $\s_\de(\de)=\de$, $D\s_\de(\de)=1$ and for $j\geq 2$, $D^j\s_\de(\de)=0$. On the right, using the formula in Claim \ref{claim:w.u}, $w(\de)=\de$, $Dw(\de)=Du(0)=D^{n+1}g(0)=1$ and for $j\geq 2$, $D^jw(\de)=(1-2\de)^{1-j}D^ju(0)=(1-2\de)^{1-j}D^{n+j}g(0)=0$. Similarly for the values at $1-\de$.

Also $w\in \F^\infty_0$. For this we check that $Dw(x)>0$ for $0<x<1$. When $0<x\leq \de$, $Dw(x) = D\s_\de(x)>0$. When $\de\leq x\leq 1-\de$, $Dw(x) = Du((x-\de)/(1-2\de)) = D^{n+1}g((x-\de)/(1-2\de))>0$. And when $1-\de\leq x < 1$, $Dw(x) = D\tau_\de(x)>0$.

Hence, we have that $f\in \F^\infty_0$.

For each $j=0,\dots,n$, $I^ju(1) = D^{n-j}g(1) = q_{n-j}$.
It follows that for $j=0,\dots,n$ we have $I^jw(1) = b_n^{-1}b_{n-j}$. This is clear if $j=0$, and for $j=1,\dots,n$ we have by Claim \ref{claim:w.u} and the choice of $q$,
\begin{align*}
I^jw(1) & = \ts r_j + (1-2\de)^{j+1}I^ju(1) + \sum_{k=1}^{j-1}s_kI^ku(1) \\
& = \ts r_j + (1-2\de)^{j+1}q_{n-j} + \sum_{k=1}^{j-1}s_kq_{n-k} \\
& = b_n^{-1}b_{n-j}
\end{align*}
Then for each $j=0,\dots,n$, we have
$I^jf(0) = b_n I^jw(0) = b_n I^{j}\s_\de(0)=0$, $I^jf(1) = b_n I^jw(1) = b_nb_n^{-1}b_{n-j} = b_{n-j}$.
For $j\geq 1$, we have $D^{j}f(0) = b_n D^{j}w(0) = b_n D^{j}\s_\de(0)=b_nb_n^{-1}\al_{j} = \al_{j}$. Similarly $D^{j} f(1) = b_n D^{j}w(1) = b_nD^j\tau_\delta(1) = b_nb_n^{-1}\beta_{j} = \beta_{j}$.

This completes the proof of the proposition.
\end{proof}

As an exercise in applying the properties of convex functions, we work out $V_1[c,d]$. Once we know from Theorem \ref{t:conj} that $W_1=\{(a_0,a_1)\in\R^2:0<a_0<a_1\}$, we could also get this via the clauses of Proposition \ref{p:scaling hat W_n}. In the following example, we sketch a direct verification.

\begin{exmp}
$V_1[c,d] = \{(a_0,a_1,b_0,b_1)\in\R^4: a_1 < (b_0-a_0)/(d-c) < b_1\}$.

\begin{proof}
If $(a_0,a_1,b_0,b_1)\in V_1[c,d]$ and $f$ is a witnessing function satisfying in particular that $D^i(f)(c) = a_i$ and $D^i(f)(d) = b_i$ for $i=0,1$, then we want to show that
\[
Df(c) < \frac{f(d)-f(c)}{d-c} < Df(d).
\]
But $Df$ is increasing and not constant, so $f$ is convex and by the first part of Proposition \ref{p:convex.properties} (13), we have $Df(c) < Df(d)$, and then the second part gives the inequalities above. Conversely, given $(a_0,a_1,b_0,b_1)\in\R^4$ satisfying $a_1 < (b_0-a_0)/(d-c) < b_1$, to build a witnessing function $f$, start with the graph of the two-segment piecewise linear function having the correct values of $f$ and $Df$ at $c$ and $d$, and then ``round off'' the corner (with an arc of circle for example) to get a $C^1$ function.
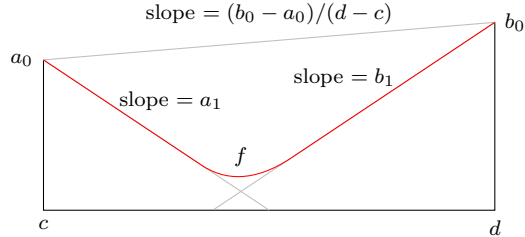
\[
\begin{tikzpicture}
\def\u{3cm}
\def\d{1cm}
\draw (0,2*\d) -- (0,0) -- (2*\u,0) -- ++(0,2.5*\d);
\node[left] at (0,2*\d) {\sz $a_0$};
\node[right] at (2*\u,2.5*\d) {\sz $b_0$};
\node[below] at (0,0) {\sz $c$};
\node[below] at (2*\u,0) {\sz $d$};
\coordinate (start) at (0,2*\d);
\coordinate (end) at (2*\u,2.5*\d);
\coordinate (sstart) at ($(0,2*\d)+(3*0.7,-2*0.7)$);
\coordinate (eend) at ($(2*\u,2.5*\d)+(- 0.9*3,- 0.9*2)$);
\draw[gray!50] (start) -- node[black,above,yshift=0.5ex] {\sz slope $=(b_0-a_0)/(d-c)$} (end);
\draw[gray!50] (start) -- node[black,above,xshift=3ex,pos=0.4] {\sz slope $=a_1$} ++(3,-2);
\draw[gray!50,shorten >= -5.3ex] (end) -- node[black,above,xshift=-3ex] {\sz slope $=b_1$} ++(-3,-2);
    \coordinate (c1) at ($(sstart) + 0.1*(3,-2)$);
    \coordinate (c2) at ($(eend) - 0.17*(3,2)$);
    \draw[red] (sstart) .. controls (c1) and (c2) .. (eend) node[above,black,pos=0.5] {\sz $f$};
\draw[red] (start) -- (sstart);
\draw[red] (end) -- (eend);
\end{tikzpicture}
\]
We leave it for the reader to check the details.
\end{proof}
\end{exmp}

\section{Converting piecewise linear functions into $C^\infty$ functions}
\label{s:p.w.l.}

We develop a tool for converting a piecewise linear function into a $C^\infty$ function in a way that is continuous in the parameters of the piecewise linear function. This will be useful in subsequent sections.

\begin{prop}\label{roundoffcorners}
Let $\Gamma=\R^4\times\R^{>0}$. For $\g=(m_1,m_2,a,c,\de)\in \Gamma$, let $F_\g\in C(\R)$ be the function whose graph on $(-\infty,a]$ is a straight line through the point $(a,c)$ having slope $m_1$, and whose graph on $[a,\infty)$ is a straight line through the point $(a,c)$ having slope $m_2$.
There is a continuous map $\Gamma\to C^\infty(\R)$, $\g\mapsto G_\g$, such that  the following hold.

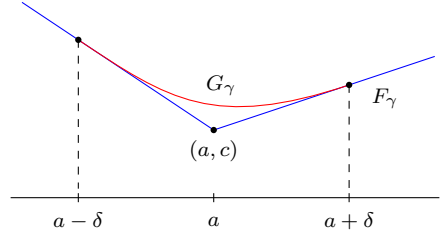
\[
\begin{tikzpicture}
\def\u{0.6cm}
    \coordinate (start) at (0,4*\u);
    \coordinate (p1) at (3*\u, 2*\u);
    \coordinate (end) at (6*\u,3*\u);
    \draw (-1.5*\u,0.5*\u) -- (8*\u,0.5*\u);
    \draw (0,0.4*\u) -- ++(0,0.2*\u);
    \draw (3*\u,0.4*\u) -- ++(0,0.2*\u);
    \draw (6*\u,0.4*\u) -- ++(0,0.2*\u);
    \node[below] at (0,0.4*\u) {\sz \rule{0pt}{5pt}$a-\de$};
    \node[below] at (3*\u,0.4*\u) {\sz \rule{0pt}{5pt}$a$};
    \node[below] at (6*\u,0.4*\u) {\sz \rule{0pt}{5pt}$a+\de$};

    \draw[blue,shorten <= -1.5*\u] (start) -- (p1);
    \draw[blue,shorten >= -2*\u] (p1) -- (end);
    \coordinate (c1) at ($(start) + 0.4*(3,-2)$);
    \coordinate (c2) at ($(end) - 0.6*(3,1)$);
    \draw[red] (start) .. controls (c1) and (c2) .. (end) node[above,black,pos=0.6] {\sz $G_\g$};
    \draw[dashed] (start) -- (0,0.7*\u);
    \draw[dashed] (end) -- (6*\u,0.7*\u);
    \fill (start) circle [radius=0.07*\u];
    \fill (p1)    circle [radius=0.07*\u];
    \fill (end)   circle [radius=0.07*\u];
    \node[below] at (p1)    {\sz $(a,c)$};
    \node at (6.8*\u,2.7*\u) {\sz $F_\g$};
\end{tikzpicture}
\]
\begin{enumerate}
\item
If $m_1=m_2$ then $G_\g=F_\g$.

\item
If $m_1\not=m_2$ then $G_\g(x)=F_\g(x)$ when $x\notin (a-\de,a+\de)$ and $D^2G_\g(x)\not= 0$ for all $x\in (a-\de,a+\de)$.

\item
The map $\g=(m_1,m_2,a,c,\de)\mapsto G_\g$ commutes with translation of the point $(a,c)$ in the sense that if $\g'=(m_1,m_2,a',c',\de)\in \Gamma$ and we let $v=(a',c')-(a,c)$, we have $G_{\g'} = G_\g + v$, i.e., $G_{\g'}(x) = G_\g(x-a'+a) + c'-c$ for $x\in\R$.

\item
The graph of $G_\g$ on $[a-\de,a+\de]$ lies inside the convex hull of the points $(x,F_\g(x))$, for $x=a-\de,\,a,\,a+\de$.
\end{enumerate}
\end{prop}

As a result of (2), when $m_1<m_2$, $DG_\g$ is strictly increasing from $m_1$ to $m_2$ on $[a-\de,a+\de]$, and similarly is strictly decreasing when $m_1>m_2$.

\begin{proof}
(1), (2), (3) Fix $\g=(m_1,m_2,a,c,\de)\in \Gamma$. Define $b_1,b_2$ by the equations
$m_1a + b_1 = c = m_2a + b_2$. Thus, the function $F=F_\g$ is given by $F(x)=m_1x+b_1$ when $x\leq a$, and $F(x)=m_2x+b_2$ when $x\geq a$. We will first produce the $C^\infty$ function $G_\g$ for the case $a=0$, $c=0$, $\de=1$ (so $b_1=b_2=0$ as well), getting a map $\Gamma'=\R^2\to C^\infty(\R)$ sending $\g=(m_1,m_2)$ to $G=G_{\g}$. Let $h$ be as in Proposition \ref{p:h}.
The derivative of $G$ will be
\[
\varphi(x) = m_1 + (m_2-m_1)h\left(\frac{x+1}{2}\right).
\]
Note that if $m_1\not=m_2$, we will have $D^2G = D\varphi \not=0$ when $-1<x<1$. Let
\begin{align}
G(x) & = F(-1) + \int_{-1}^x \varphi(t)\,dt \nonumber \\
& = -m_1 + m_1(x+1) + (m_2-m_1)\int_{-1}^x h\left(\frac{t+1}{2}\right)\,dt \nonumber \\
& = m_1x + 2(m_2-m_1)\int_{0}^{(x+1)/2} h(t)\,dt\label{eq3}.
\end{align}
We have $G(x) = F(x)$ if $m_1=m_2$. Note that $\g\mapsto G_{\g}$ is continuous on $\Gamma'$ since each of the maps into $C^\infty(\R)$ given by $\g\mapsto m_1x$, $\g\mapsto 2(m_2-m_1)r$, where $r$ is the $C^\infty$ functions $r(x)=I(h)((x+1)/2)$, is continuous, and $\g\mapsto G_{\g}$ is obtained by adding, these.

We have $G(x)=m_1x=F(x)$ when $x\leq -1$. We require $G(1) = F(1)$. Using (\ref{eq3}) and $\int_0^1 h(t)\,dt = 1/2$ we have
\begin{gather*}
G(1) = m_1 + 2(m_2-m_1)\int_{0}^{1} h(t)\,dt = m_2 = F(1).
\end{gather*}
Then for $x\geq 1$, $\varphi(x) = m_2$, so
\begin{align*}
G(x) & = F(-1) + \int_{-1}^x \varphi(t)\,dt = F(-1) + \int_{-1}^{1} \varphi(t)\,dt + \int_{1}^x \varphi(t)\,dt \\
& = G(1) + m_2(x-1) = m_2 + m_2(x-1) = m_2x = F(x).
\end{align*}

For a general $\de>0$, with $\g=(m_1,m_2,a,c,\de)$, letting $\bar{\g}=(m_1,m_2)$, we can take
\[
G_\g(x) = c + \de\, G_{\bar{\g}}\left(\frac{x-a}{\de}\right).
\]
It is easy to see that the map $\g\mapsto G_\g$ is continuous. (Write $w_{a,b}(x) = b(x-a)$, then $(a,b)\mapsto w_{a,b}$ is continuous, hence so is $\g\mapsto (\bar{\g},a,c,\de)\mapsto (\bar{\g},a,c,1/\de)\mapsto (G_{\bar{\g}},w_{a,1/\de})\mapsto G_{\bar{\g}}(w_{a,1/\de})$.)

When $x\leq a-\de$, $(x-a)/\de\leq -1$, so
\begin{align*}
G_\g(x) & = c + \de m_1(x-a)/\de \\
& = m_1a+b_1 + m_1(x-a) = m_1x+b_1 = F(x).
\end{align*}
(This argument also shows that $G_\g = F$ if $m_1=m_2$.)
Similarly, when $x\geq a+\de$, $G_\g(x) = m_2x+b_2 = F(x)$. When $a-\de<x<a+\de$, $DG_\g(x) = DG_{\bar{\g}}((x-a)/\de)$, so $D^2G_\g(x) = \de^{-1}D^2G_{\bar{\g}}((x-a)/\de)\not=0$ as long as $m_1\not=m_2$.

For (3), we have
\begin{align*}
G_\g(x - a' + a) + c'-c & = c + \de\,G_{\bar{\g}}\left(\frac{(x-a'+a)-a}{\de}\right) + c'-c \\
& = c' + \de\,G_{\bar{\g}}\left(\frac{x-a'}{\de}\right) = G_{\g'}(x)
.\end{align*}

(4) The claim is clear from (1) if $m_1=m_2$, so assume $m_1\not=m_2$. On $[a-\de,a+\de]$, $G_\g$ agrees with $F_\g$ at the endpoints. If $m_1<m_2$, then $D^2G_\g\geq 0$ on $[a-\de,a+\de]$, so $G_\g$ is convex on $[a-\de,a+\de]$ and hence its graph is below the secant line $L$ through the points $(x,F_\g(x))$ for $x=a-\de$ and $x=a+\de$.
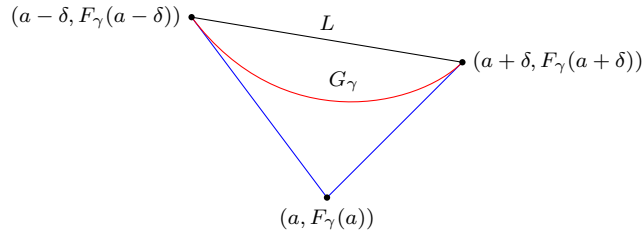
\[
\begin{tikzpicture}
\def\u{0.6cm}
    \coordinate (start) at (0,4*\u);
    \coordinate (p1) at (3*\u, 0*\u);
    \coordinate (end) at (6*\u,3*\u);
    \draw[blue] (start) -- (p1) -- (end);
    \coordinate (c1) at ($(start) + 0.4*(3,-4)$);
    \coordinate (c2) at ($(end) - 0.6*(1,1)$);
    \draw[red] (start) .. controls (c1) and (c2) .. (end) node[above,black,pos=0.5] {\sz $G_\g$};
    \draw (start) -- node[above] {\sz $L$} (end);
    \fill (start) circle [radius=0.07*\u];
    \fill (p1)    circle [radius=0.07*\u];
    \fill (end)   circle [radius=0.07*\u];
    \node[left]  at (start) {\sz $(a-\de,F_\g(a-\de))$};
    \node[below] at (p1)    {\sz $(a,F_\g(a))$};
    \node[right] at (end)   {\sz $(a+\de,F_\g(a+\de))$};
\end{tikzpicture}
\]
The graph is also above the support lines at $a-\de$ and $a+\de$ which are given by the graph of $F_\g$. Thus, the graph of $G_\g$ lies between that of $F_\g$ and the secant line $L$. Similarly if $m_1>m_2$ using concavity of $G_\g$.
\end{proof}

Given points $a_1<\dots<a_n$ with $a_i+2\de<a_{i+1}$ for $1\leq i<n$, and a continuous function $f$ which has constant derivative on each component of each set $(a_i-\de,a_i+\de)\sm\{a_i\}$, we have $f=F_{\g_i}$ on $[a-\de,a+\de]$ for some (unique) $\g_i\in \Gamma$ having last coordinate $\de$. The function $g$ obtained by taking $g(x)=G_{\g_i}(x)$ for $x\in [a_i-\de,a_i+\de]$, $i=1,\dots,n$, $g(x)=f(x)$ otherwise, will be called the \emph{$\de$-modification of $f$ at $a_1,\dots,a_n$}.

\begin{cor}\label{c:roundoffcorners}
Let $n\in\N$ and let $\Gamma_n$ be the convex subset of $\R^{2n+3}$ defined by
\[
\Gamma_n=\{(m_1,m_{n+1},a_1,c_1,a_2,c_2,\dots,a_n,c_n,\de):\de>0,\,a_i+2\de < a_{i+1}\ (1\leq i<n)\}.
\]
For $\g = (m_1,m_{n+1},a_1,c_1,\dots,a_n,c_n,\de)\in \Gamma_n$, let $F^n_\g$ be the continuous function with $F^n_\g(a_i)=c_i$, $i=1,\dots,n$, $F^n_\g$ having constant derivative on each component of $\R\sm\{a_1,\dots,a_n\}$, with $DF^n_\g=m_1$ on $(-\infty,a_1)$, $DF^n_\g=m_{n+1}$ on $(a_n,\infty)$. Let $G^n_\g$ be the $\de$-modification of $F^n_\g$ at $a_1,\dots,a_n$.

The function $G^n_\g$ is $C^\infty$ with $G^n_\g = F_\g$ on $\R\sm\bigcup_{i=1}^n(a_i-\de,a_i+\de)$. The map $\Gamma_n\to C^\infty(\R)$ given by $\g\mapsto G^n_\g$ is continuous.
\end{cor}

\begin{proof}
This follows by induction on $n$.
We have $\Gamma_1 = \Gamma = \R^4\times \R^{>0}$, so the case $n=1$ follows from Proposition \ref{roundoffcorners}. For the continuity of $\g\mapsto G^n_\g$ when $n>1$, let
\[
\g = (m_1,m_{n+1},a_1,c_1,\dots,a_n,c_n,\de) = (m_1^\g,m_{n+1}^\g,a_1^\g,c_1^\g,\dots,a_n^\g,c_n^\g,\de^\g)\in \Gamma_n.
\]
Define $m_n=(c_{n}-c_{n-1})/(a_{n}-a_{n-1})$ and let
\[
\g' = (m_1,m_{n},a_1,c_1,\dots,a_{n-1},c_{n-1},\de)\in \Gamma_{n-1},\ \g''=(m_n,m_{n+1},a_{n},c_{n},\de)\in \Gamma.
\]
Note that $m_n$ is a continuous function of $\g$ and therefore so are $\g'$ and $\g''$. By the induction hypothesis, the function $\g\mapsto \g'\mapsto G^{n-1}_{\g'}$ is continuous, and by the case $n=1$, so is $\g\mapsto \g''\mapsto G_{\g''}$. The function $\g\mapsto G^n_\g$ is obtained from these two by definition by cases as in Proposition \ref{p:C.infty.by.cases} taking $\g\mapsto a_\g$ to be given by $a_\g = (a_{n-1}+a_{n})/2$, and therefore is continuous.
\end{proof}

\section{Restrictions on $I(f)(1)$, $I^2(f)(1)$ and $I^3(f)(1)$}
\label{s:If.I2f}

Let $f\colon[0,1]\to[0,1]$ be an increasing function. Recall that we write $I(f)(x)=\int_0^xf(t)\,dt$. Let $a=I(f)(1)$, $b=I^2(f)(1)$, and $c=I^3(f)(1)$.
In this section, we develop inequalities necessarily satisfied by $a$, $b$ and $c$, and show that these are sufficient for the existence of an increasing function $f\colon[0,1]\to[0,1]$ such that $a=I(f)(1)$, $b=I^2(f)(1)$, and $c=I^3(f)(1)$.

The restrictions on $a$ are straightforward to determine. We record them for reference and leave their verification to the reader.

\begin{prop}\label{n=1.a}
We have $0\leq a\leq 1$. If $a=0$ then $f(x)=0$ for $0<x<1$. If $a=1$ then $f(x)=1$ for $0<x<1$.
\end{prop}

Note that since $f$ is increasing, if $f(x)=0$ for $0<x<1$ then necessarily $f(0)=0$, so we could have written that $f(x)=0$ for $0\leq x<1$. In such statements, when $f$ is a step function, we prefer formulations which ignore the values at the endpoints of the intervals on which $f$ is constant since in general they are not uniquely determined and they are irrelevant to the computation of the integrals $I^n(f)$.

Next, we examine the restrictions on $b$, showing in particular that $a^2/2\leq b\leq a/2$.

\begin{prop}\label{n=2.a}
Let $f\colon[0,1]\to[0,1]$ be increasing with $I(f)(1)=a$, $I^2(f)(1)=b$.
\begin{enumerate}
\item
$(x-(1-a))_+\leq I(f)(x)\leq ax,\ \ 0\leq x\leq 1$.

\item
$I(f)(x) < ax$ for $x\in (0,1)$, unless $f(x)=a$ for all $x\in (0,1)$.

\item
For $0<x\leq 1-a$, $(x-(1-a))_+ < I(f)(x)$ unless  $f=0$ on $(0,x)$.

\n For $1-a\leq x<1$, $(x-(1-a))_+ < I(f)(x)$ unless $f=1$ on $(x,1)$.

\item
We have
\[
((x-(1-a))_+)^2/2\leq I^2(f)(x)\leq ax^2/2,\ \ 0\leq x\leq 1.
\]
Taking $x=1$, we get in particular,
\[
a^2/2\leq b\leq a/2,
\]
with the first inequality being strict unless $f=0$ on $(0,1-a)$ and $f=1$ on $(1-a,1)$, and the second inequality being strict unless $f$ is constant on $(0,1)$.
\end{enumerate}
\end{prop}

\begin{proof}
(1) $g=I(f)$ is convex since $f$ is increasing, and we have $D_+g(0)=f(0+)\geq 0$, $D_-g(1)=f(1-)\leq 1$, so the graph of $g$ is above the line through $(0,0)$ of slope $0$, i.e., the $x$-axis, and above the line of slope $1$ through $(0,a)$. This gives the first inequality. The graph of $g$ also lies below its secant through the points at $x=0$ and $x=1$, giving the second inequality.

(2) If $g(x)=I(f)(x)=ax$ for some $x\in (0,1)$, then the convex function $g$ on $[0,1]$ has three points on the line $y=ax$ and therefore coincides with that line. It follows that at each of the (all but countably many) points $x$ where $f$ is continuous, we have $a = Dg(x)=f(x)$. Since $f$ is increasing, we get $f=a$ everywhere on $(0,1)$.

(3) The first part is clear, so suppose that for some $x\in [1-a,1)$, $I(f)(x) = x-(1-a)$. Then the graph of $g=I(f)$ has a secant on $[x,1]$ which follows the line $y=x-(1-a)$. It therefore lies below that line. But we know from (1) that it lies above that line, so it must coincide with the secant. It follows that at each of the (all but countably many) points $t\in [x,1]$ where $f$ is continuous, we have $1 = Dg(t)=f(t)$. Since $f$ is increasing, we get $f=1$ everywhere on $(x,1)$.

(4) Follows from the previous items.
\end{proof}

To study the restrictions on $c$, we introduce a family of functions $f_{uva}$ and their integrals $g_{uva}=I(f_{uva})$, defined as follows. Let $a\in [0,1]$. For $0<u<1$ and $v\in\R$, let $f_{uva}\colon[0,1]\to\R$ be given by
\[
f_{uva}(x) = \begin{cases}
v/u & \text{for $0\leq x<u$} \\
(a-v)/(1-u) & \text{for $u\leq x\leq 1$}.
\end{cases}
\]
When the value of $a$ is fixed, we also write $f_{uv}$ and $g_{uv}$ instead of $f_{uva}$ and $g_{uva}$.

\begin{prop}
The functions $f_{uva}$ have the following properties.
\begin{enumerate}
\item
$I(f_{uva})(1) = a$.

\item
$f_{uva}$ is constant if and only if $v=au$.

\item
$f_{uva}$ is either constant or a $2$-step step function.

\item
$f_{uva} = f_{u'v'a'}$ implies $a=a'$ and $v/u=v'/u'$. We have $u=u'$ $($and hence also $v=v')$ unless the functions are constant.

\item
If $f\colon[0,1]\to\R$ is either constant or a $2$-step step function, then $f=_\fin f_{uva}$, where $a=I(f)(1)$, $u$ is chosen so that $f$ is constant on $(0,u)$ and on $(u,1)$, and $v$ is chosen so that the value of $f$ on $(0,u)$ is $v/u$.
\end{enumerate}
\end{prop}

\begin{proof}
(1) and (3) are obvious.

(2) Solving $v/u = (a-v)/(1-u)$ gives $v=au$.

(4), (5) Suppose $f$ is either constant or a $2$-step step function.
Choose $u\in (0,1)$ so that $f$ is constant on $(0,u)$ and on $(u,1)$. (This value is unique unless $f$ is constant on $(0,1)$ in which case it is arbitrary.) Then choose $v$, as we must if we want $f=_\fin f_{uva}$, so that the value of $f$ on $(0,u)$ is $v/u$. By (1), we must take $a=I(f)(1)$. The value $m$ of $f$ on $(u,1)$ can be computed from the equation $a=I(f)(1)$ which gives $a=u(v/u) + (1-u)m$ and hence $m=(a-v)/(1-u)$. Thus, $f=_\fin f_{uva}$.
\end{proof}

Now fix $a$ with $0\leq a\leq 1$.
Note that if $f\colon[0,1]\to\R$ is increasing and $f=_\fin f_{uv}$ then in fact $f(x)=f_{uv}(x)$ for all $x\in (0,u)\cup (u,1)$ (by Proposition \ref{p:fin}). All such functions $f$ have the same integral function $g_{uv}=I(f)=I(f_{uv})$ given by
\[
g_{uv}(x) = \begin{cases}
(v/u)x & \text{for $0\leq x \leq u$} \\
v + ((a-v)/(1-u))(x-u) & \text{for $u\leq x \leq 1$}.
\end{cases}
\]
Note that $g_{uv}(u)=v$ and $g_{uv}(1)=a$.
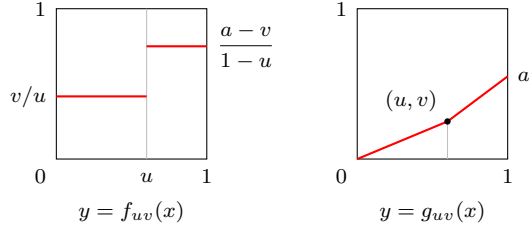
\[
\begin{tikzpicture}
\def\u{2cm}
\def\a{0.55}
\def\b{0.2}
\begin{scope}
    \draw (0,1*\u) -- (0,0) -- (1*\u,0) -- (1*\u,1*\u) -- cycle;
    \coordinate (p1) at (0*\u, 0*\u);
    \coordinate (p2) at (0.6*\u, 0.25*\u);
    \coordinate (p3) at (1*\u,\a*\u);
    \draw[gray!50] (0.6*\u, 0*\u) -- ++(0,1*\u);
    \draw[red,thick] (0,{(0.25/0.6)*\u}) -- (0.6*\u,{(0.25/0.6)*\u});
    \draw[red,thick] (0.6*\u,{((\a-0.25)/(1-0.6))*\u}) -- (1*\u,{((\a-0.25)/(1-0.6))*\u});
    \node[left]  at (0,{(0.25/0.6)*\u}) {\sz $v/u$};
    \node[left]  at (0,1*\u) {\sz $1$};
    \node[below left]  at (0,0) {\sz $0$};
    \node[below]  at (1*\u,0) {\sz $1$};
    \node[below]  at (0.6*\u,0) {\sz \rule{0pt}{5pt}$u$};
    \node[right]  at (1*\u,{((\a-0.25)/(1-0.6))*\u}) {\sz $\ds\frac{a-v}{1-u}$};
    \node[below,yshift=-2.6ex] at (0.5*\u,0) {\sz\upshape $y=f_{uv}(x)$};
\end{scope}
\begin{scope}[xshift=2*\u]
    \draw (0,1*\u) -- (0,0) -- (1*\u,0) -- (1*\u,1*\u) -- cycle;
    \coordinate (p1) at (  0*\u, 0*\u);
    \coordinate (p2) at (  0.6*\u, 0.25*\u);
    \coordinate (p3) at (  1*\u,\a*\u);
    \draw[gray!50] (  0.6*\u, 0*\u) -- (p2);
    \draw[red,thick] (p1) -- (p2) -- (p3);
    \fill (p2) circle [radius = 1.2pt];
    \node[left]  at (0,1*\u) {\sz $1$};
    \node[below left]  at (0,0) {\sz $0$};
    \node[below]  at (1*\u,0) {\sz $1$};
    \node[above left]  at (p2) {\sz $(u,v)$};
    \node[right]  at (1*\u,\a*\u) {\sz $a$};
    \node[below,yshift=-2.6ex] at (0.5*\u,0) {\sz\upshape $y=g_{uv}(x)$};
\end{scope}
\end{tikzpicture}
\]

\begin{prop}\label{p:T}
Let $u,v\in\R$, $0<u<1$.
\begin{enumerate}
\item
The following properties are equivalent.
\begin{enumerate}[\rm(a)]
\item
$f_{uv}$ is increasing with values in $[0,1]$, i.e.,
$0 \leq v/u \leq (a-v)/(1-u) \leq 1$.

\item
$(u,v)$ belongs to the region in $[0,1]^2$ bounded by the lines $y=ax$, $y=0$, $y=x-(1-a)$, equivalently, $(u-(1-a))_+ \leq v \leq au$. This region is a triangle except when $a=0$, when it equals the horizontal axis, or $a=1$, when it equals the diagonal.
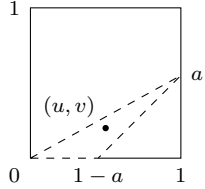
\[
\begin{tikzpicture}
\def\u{2cm}
\def\a{0.55}
    \draw ({(1-\a)*\u},0) -- (1*\u,0) -- (1*\u,1*\u) -- (0,1*\u) -- (0,0);
    \draw[dashed ] (0,0) -- (1*\u,\a*\u)
        (0,0) -- ({(1-\a)*\u},0) -- (1*\u,\a*\u);
    \coordinate (p1) at (  0*\u, 0*\u);
    \coordinate (p2) at (0.5*\u,0.2*\u);
    \coordinate (p3) at (  1*\u,\a*\u);
    \fill (p2) circle [radius=1.2pt];
    \node[above left]  at (p2) {\sz $(u,v)$};
    \node[left]  at (0,1*\u) {\sz $1$};
    \node[below left]  at (0,0) {\sz $0$};
    \node[below]  at (1*\u,0) {\sz $1$};
    \node[below]  at ({(1-\a)*\u},0) {\sz $1-a$};
    \node[right]  at (1*\u,\a*\u) {\sz $a$};
\end{tikzpicture}
\]
\end{enumerate}

\item
When $0<a<1$, the following properties are equivalent.
\begin{enumerate}[\rm(a)]
\item
$f_{uv}$ is increasing and nonconstant with values in $(0,1)$, i.e., $0 < v/u < (a-v)/(1-u) < 1$.

\item
$(u,v)$ belongs to the interior $T$ of the triangle bounded by the lines $y=ax$, $y=0$, $y=x-(1-a)$, equivalently, $(u-(1-a))_+ < v < au$.
\end{enumerate}
\end{enumerate}
\end{prop}

\begin{proof}
(1) Using $0<u<1$ we have that $u$ and $1-u$ are both positive.
Writing each of the inequalities $0 \leq v/u \leq (a-v)/(1-u) \leq 1$ as a restriction on $v$, we have $v \geq 0$, $v \leq au$, $u-(1-a) \leq v$. Thus, the point $(u,v)$ must lie in the region bounded by the lines $y=ax$, $y=0$, $y=x-(1-a)$. The equivalence of the formulation $(u-(1-a))_+ \leq v \leq au$ is clear, as is the claim about the shape of the region.

(2) Replace the inequalities in the proof of (1) by strict ones.
\end{proof}

\begin{defn}
Fix $a,b$ satisfying $0<a<1$ and $a^2/2\leq b < a/2$.
Set
\[
u_1=1-\frac{2b}a,\ \ v_1=0,\ \ u_2=\frac{1-2a+2b}{1-a},\ \ v_2=\frac{2b-a^2}{1-a}.
\]
Also set $g_1 = g_{u_1,v_1}$, $g_2 = g_{u_2,v_2}$. The graphs of these functions are drawn in Proposition \ref{p:g1g2} below.
\end{defn}

\begin{rem}\label{r:u1u2}
(a) The condition $a^2/2\leq b < a/2$ is equivalent separately to each of the conditions $0<u_1\leq 1-a$ and $1-a\leq u_2<1$. In particular it ensures $0<u_i<1$ for $i=1,2$, so that $f_{u_iv_ia}$ is defined. If $b=a^2/2$ then $u_1=u_2=1-a$ and $v_1=v_2=0$, so $g_1=g_2$.

(b) The function $f_{u_1,v_1,a}$ takes the value $0$ on $(0,u_1)$, and the function $f_{u_2,v_2,a}$ takes the value $1$ on $(u_2,1)$. (Its value on $(u_2,1)$ is $(a-v_2)/(1-u_2)$ and we have $a-v_2=1-u_2=(a-2b)/(1-a)$.)
\end{rem}

The next proposition is a converse to Remark \ref{r:u1u2} (b).

\begin{prop}\label{p:two.step.0.1}
Let $0<u<1$. Let $f\colon [0,1]\to [0,1]$ be a nonconstant increasing function constant on each of the intervals $(0,u)$ and $(u,1)$. Let $a=I(f)(1)$ and $b=I^2(f)(1)$.
\begin{enumerate}
\item
If $f$ has value $0$ on $(0,u)$, then $f=_\fin f_{u_1,v_1,a}$.

\item
If $f$ has value $1$ on $(u,1)$, then $f=_\fin f_{u_2,v_2,a}$.
\end{enumerate}
\end{prop}

\begin{rem}\label{r:p:two.step.0.1}
That $f$ is not constant on $(0,1)$ ensures that it is not identically $0$ or identically $1$ on $(0,1)$, so $0<a<1$.
By Proposition \ref{n=2.a} (4), we then have $a^2/2\leq b<a/2$ so that $f_{u_i,v_i,a}$ is defined for $i=1,2$.
\end{rem}

\begin{proof}
(1) We have $f(x)=m$ for $u<x<1$, for some $m$ which is positive since $a>0$ (Remark \ref{r:p:two.step.0.1}). By direct calculation, $a=I(f)(1)=m(1-u)$ and $b=I^2(f)(1)=m(1-u)^2/2$.
Plugging these values of $a$ and $b$ into $u_1 = 1-2b/a$, we get $u_1=1-(m(1-u)^2)/(m(1-u)) = 1-(1-u) = u$, and therefore $f=_\fin
f_{u_1,v_1,a}$ since both functions are $0$ on $(0,u)$ and they have the same integral $a$ on $[0,1]$, so are necessarily equal on $(u,1)$ as well.

(2) We have $f(x)=m$ for $0<x<u$, $f(x)=1$ for $u<x<1$, for some $m$ which is $<1$ since $a<1$.
Then
\begin{align*}
I(f)(x) & =
\begin{cases}
mx & \text{if $0\leq x\leq u$} \\
mu + (x-u) & \text{if $u\leq x\leq 1$},
\end{cases} \\
I^2(f)(x) & =
\begin{cases}
mx^2/2 & \text{if $0\leq x\leq u$} \\
mu^2/2 + mu(x-u) + (x-u)^2/2 & \text{if $u\leq x\leq 1$}.
\end{cases}
\end{align*}
so $a=I(f)(1)=mu+(1-u)$ and $b=I^2(f)(1)=mu^2/2+mu(1-u)+(1-u)^2/2$, which we can write as $2b
= u^2(1-m) - 2u(1-m) + 1$.
Substituting the values of $a$ and $b$ into the formula for $u_2$, we get
\[
u_2=\frac{1-2a+2b}{1-a}
= \frac{1 - 2mu - 2(1-u) + u^2(1-m) - 2u(1-m) + 1}{1-mu-(1-u)}
= \frac{u^2(1-m)}{u(1-m)} = u.
\]
It follows that $f =_\fin
f_{u_2,v_2,a}$ since both functions have value $1$ on $(u,1)=(u_2,1)$ and they have the same integral $a$ on $[0,1]$, so are necessarily equal on $(0,u)$ as well.
\end{proof}

\begin{prop}\label{p:I2.of.f}
Let $0<a<1$, $0<u<1$.
\begin{enumerate}
\item
For any $b\in\R$, we have $I^2(f_{uv})(1) = b$ if and only if $(u,v)$ is on the straight line $L$ of equation $y = ax - a + 2b$.
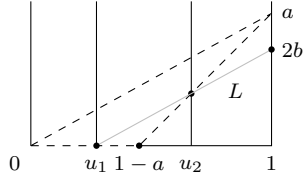
\[
\begin{tikzpicture}
\def\u{3.2cm}
\def\a{0.55}
\def\b{0.2}
\def\xx{((1-2*\a+2*\b)/(1-\a))}
    \draw (0,0.6*\u) -- (0,0) ({(1-\a)*\u},0) -- (1*\u,0) -- (1*\u,0.6*\u);
    \draw[dashed ] (0,0) -- (1*\u,\a*\u)
        (0,0) -- ({(1-\a)*\u},0) -- (1*\u,\a*\u);
    \coordinate (p1) at (  0*\u, 0*\u);
    \coordinate (p2) at (  0.4*\u, {(-0.6*\a + 2*\b)*\u});
    \coordinate (p3) at (  1*\u,\a*\u);
    \coordinate (p4) at ( {(1-2*\b/\a)*\u}, 0*\u);
    \coordinate (p5) at (1*\u,2*\b*\u);
    \coordinate (p6) at ( {(1-\a)*\u}, 0*\u);
    \coordinate (p2t) at (  {\xx*\u}, {((\xx-1)*\a + 2*\b)*\u});
    \fill (p2t) circle [radius = 1.2pt];
    \draw (p4) -- ++(0,0.6*\u);
    \draw ({\xx*\u},0) -- ++(0,0.6*\u);
    \draw[gray!50] (p4) -- (p5);
    \fill (p4) circle [radius = 1.2pt];
    \fill (p5) circle [radius = 1.2pt];
    \fill (p6) circle [radius = 1.2pt];
    \node[below left]  at (0,0) {\sz $0$};
    \node[below]  at (1*\u,0) {\sz $1$};
    \node[below]  at (p4) {\sz \rule{0pt}{5pt}$u_1$};
    \node[below]  at ({\xx*\u},0) {\sz \rule{0pt}{5pt}$u_2$};
    \node[right]  at (1*\u,\a*\u) {\sz $a$};
    \node[below]  at ({(1-\a)*\u},0) {\sz $1-a$};
    \node[right]  at (p5) {\sz $2b$};
    \node at (0.85*\u,0.23*\u) {\sz $L$};
\end{tikzpicture}
\]

\item
$L\cap T\not=\e$ if and only if $a^2/2<b<a/2$.

\item
When $a^2/2<b<a/2$, $(u,v)\in L\cap T$ if and only if $(u,v)\in L$ and $u_1<u<u_2$. The endpoints of this segment of $L$ from $u_1$ to $u_2$ are $(u_1,v_1)$ and $(u_2,v_2)$.
\end{enumerate}
\end{prop}

\begin{proof}
(1) Requiring that $I^2(f_{uv})(1) = b$, or equivalently $I(g_{uv})(1) = b$, translates into
\[
(v/u)\frac{u^2}{2} + v(1-u) + \frac{a-v}{1-u}\cdot \frac{(1-u)^2}{2} = b,
\]
which simplifies to $v = au - a + 2b$, which is equivalent to saying $(u,v)\in L$.

(2), (3) The conditions defining $(x,y)\in T$ are $y<ax$, $y>0$, and $y>x-(1-a)$. Plugging $y = ax - a + 2b$ into each of these, we see that the first is satisfied if and only if $b<a/2$ and the other two are satisfied if and only if $u_1<x<u_2$. Thus, all three can be satisfied if and only if $b<a/2$ and $u_1<u_2$. The latter simplifies to $a^2/2<b$. The computation of $v_2=au_2-a+2b$ is straightforward.

(The criterion $a^2/2<b<a/2$ in (2) can also be seen geometrically. Note that $L$ has slope $a$ and hence is parallel to the top edge of $T$. The point where $L$ crosses the $x$-axis is $u_1$. In order that $L$ intersect $T$, we therefore need $0<u_1<1-a$, or equivalently $a^2/2<b<a/2$.)
\end{proof}

\begin{prop}\label{p:g1g2}
Assume that $0<a<1$ and $a^2/2<b<a/2$. The integral $I^2(g_{uv})(1)$, for $(u,v)\in L$, $0<u<1$, is a continuous strictly increasing function of $u$, specifically
\[
I^2(g_{uv})(1) = \frac{(a - 2b)u - a + 4b}6.
\]
Taking $u=u_1$ and $u=u_2$ we get
\[
I^2(g_{1})(1) = \frac{2b^2}{3a},\ \
I^2(g_{2})(1) = \frac{-a^2 + 2ab -4b^2 + 2b}{6(1-a)}.
\]
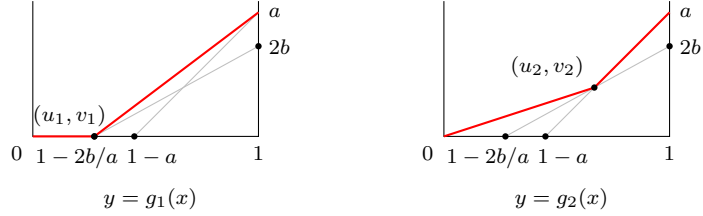
\[
\begin{tikzpicture}
\def\u{3cm}
\def\a{0.55}
\def\b{0.2}
\def\xx{(1-2*\b/\a)}
    \draw (0,0.6*\u) -- (0,0) -- (1*\u,0) -- (1*\u,0.6*\u);
    \coordinate (p1) at (  0*\u, 0*\u);
    \coordinate (p2) at (  {\xx*\u}, {((\xx-1)*\a + 2*\b)*\u});
    \coordinate (p3) at (  1*\u,\a*\u);
    \coordinate (p4) at ( {(1-2*\b/\a)*\u}, 0*\u);
    \coordinate (p5) at (1*\u,2*\b*\u);
    \coordinate (p6) at ( {(1-\a)*\u}, 0*\u);
    \draw[gray!50] (p4) -- (p5);
    \draw[gray!50] (p3) -- (p6);
    \draw[red,thick] (p1) -- (p2) -- (p3);
    \fill (p2) circle [radius = 1.2pt];
    \fill (p4) circle [radius = 1.2pt];
    \fill (p5) circle [radius = 1.2pt];
    \fill (p6) circle [radius = 1.2pt];
    \node[below left]  at (0,0) {\sz $0$};
    \node[below]  at (1*\u,0) {\sz $1$};
    \node[below, xshift=-1.4ex]  at (p4) {\sz \rule{0pt}{5pt}$1-2b/a$};
    \node[below, xshift=1.4ex]  at (p6) {\sz \rule{0pt}{6pt}$1-a$};
    \node[above left, xshift=1.8ex]  at (p2) {\sz $(u_1,v_1)$};
    \node[right]  at (1*\u,\a*\u) {\sz $a$};
    \node[right]  at (p5) {\sz $2b$};
    \node[below] at (current bounding box.south) {\sz\upshape $y=g_{1}(x)$};
\end{tikzpicture}
\rule{1.5cm}{0cm}
\begin{tikzpicture}
\def\u{3cm}
\def\a{0.55}
\def\b{0.2}
\def\xx{((1-2*\a+2*\b)/(1-\a))}
    \draw (0,0.6*\u) -- (0,0) -- (1*\u,0) -- (1*\u,0.6*\u);
    \coordinate (p1) at (  0*\u, 0*\u);
    \coordinate (p2) at (  {\xx*\u}, {((\xx-1)*\a + 2*\b)*\u});
    \coordinate (p3) at (  1*\u,\a*\u);
    \coordinate (p4) at ( {(1-2*\b/\a)*\u}, 0*\u);
    \coordinate (p5) at (1*\u,2*\b*\u);
    \coordinate (p6) at ( {(1-\a)*\u}, 0*\u);
    \draw[gray!50] (p4) -- (p5);
    \draw[gray!50] (p3) -- (p6);
    \draw[red,thick] (p1) -- (p2) -- (p3);
    \fill (p2) circle [radius = 1.2pt];
    \fill (p4) circle [radius = 1.2pt];
    \fill (p5) circle [radius = 1.2pt];
    \fill (p6) circle [radius = 1.2pt];
    \node[below left]  at (0,0) {\sz $0$};
    \node[below]  at (1*\u,0) {\sz $1$};
    \node[below, xshift=-1.4ex]  at (p4) {\sz \rule{0pt}{5pt}$1-2b/a$};
    \node[below, xshift=1.4ex]  at (p6) {\sz \rule{0pt}{6pt}$1-a$};
    \node[above left]  at (p2) {\sz $(u_2,v_2)$};
    \node[right]  at (1*\u,\a*\u) {\sz $a$};
    \node[right]  at (p5) {\sz $2b$};
    \node[below] at (current bounding box.south) {\sz\upshape $y=g_{2}(x)$};
\end{tikzpicture}
\]
\end{prop}

\begin{proof}
Fix $(u,v)$ and let $g=g_{uv}$. Write $m=m_{uv}=(a-v)/(1-u)$.
We have
\[
g(x) = \begin{cases}
(v/u)x & \text{for $0\leq x \leq u$} \\
v + m(x-u) & \text{for $u\leq x \leq 1$},
\end{cases}
\]
which gives
\[
I(g)(x) = \begin{cases}
\ds\frac{v}{2u}x^2 & \text{for $0\leq x \leq u$} \\[6pt]
\ds \frac{uv}{2} + v(x-u) + m\frac{(x-u)^2}{2} & \text{for $u\leq x \leq 1$},
\end{cases}
\]
\[
I^2(g)(x) = \begin{cases}
\ds\frac{v}{6u}x^3 & \text{for $0\leq x \leq u$} \\[8pt]
\ds \frac{u^2v}{6} + \frac{uv}{2}(x-u) + v\frac{(x-u)^2}{2} + m\frac{(x - u)^3}{6} & \text{for $u\leq x \leq 1$}. \\
\end{cases}
\]
Taking $x=1$ gives
\[
I^2(g)(1) = \frac{u^2v}{6} + \frac{uv}{2}(1-u) + v\frac{(1-u)^2}{2} + m\frac{(1 - u)^3}{6}.
\]
Multiplying by $6$ and expanding, we get
\begin{align*}
6I^2(g)(1) & = u^2v + 3uv - 3u^2v + 3v - 6uv + 3u^2v + a - 2au + au^2 - v + 2uv - u^2v \\
& = -uv + 2v + a - 2au + au^2 \\
& = -u(au-a+2b) + 2(au-a+2b) + a - 2au + au^2 \\
& = (a - 2b)u - a + 4b,
\end{align*}
as desired.
When $u=u_1=1-2b/a$, we get
\[
6I^2(g_1)(1) = (a - 2b)\left(1-\frac{2b}{a}\right) - a + 4b = \frac{4b^2}{a},
\]
which gives the claimed value for $I^2(g_1)(1)$, and when $u=u_2=(1-2a+2b)/(1-a)$, we get
\begin{align*}
6(1-a)I^2(g_2)(1) & = (a - 2b)(1-2a+2b) - a(1-a) + 4b(1-a) \\
& = a-2a^2+2ab - 2b +4ab -4b^2 - a + a^2 + 4b - 4ab \\
& = -a^2 + 2ab -4b^2 + 2b
\end{align*}
and that gives the claimed value for $I^2(g_2)(1)$.
\end{proof}

We are now ready to establish the necessary restriction on $c$.

\begin{prop}\label{p:n.equals.3}
Let $f\colon [0,1]\to[0,1]$ be increasing with $I(f)(1)=a$, $I^2(f)(1)=b$, $I^3(f)(1)=c$. Assume that $a^2/2<b<a/2$. Let $g=I(f)$. Then $I(g_1)\leq I(g)\leq I(g_2)$. More precisely we have the following.

\begin{enumerate}
\item
Either $g=g_1$, or there are $x_0\in [0,u_1)$ and $x_1\in (u_1,1)$ such that $g=g_1=0$ on $[0,x_0]$, $g_1<g$ on $(x_0,x_1)$, $g<g_1$ on $(x_1,1)$, and $I(g_1)(x)<I(g)(x)$ for all $x\in (x_0,1)$.

\item
Either $g=g_2$, or there are $x_0\in (0,u_2)$ and $x_1\in (u_2,1]$ such that $g<g_2$ on $(0,x_0)$, $g_2<g$ on $(x_0,x_1)$, and $g=g_2$ on $[x_1,1]$, with $I(g)(x)<I(g_2)(x)$ for all $x\in (0,x_1)$, and $I(g)(x)=I(g_2)(x)$ for all $x\in [x_1,1]$.
\end{enumerate}
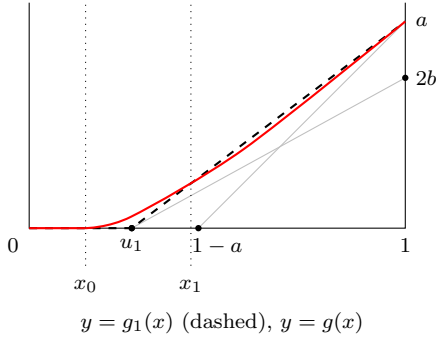
\[
\begin{tikzpicture}
\def\u{5cm}
\def\a{0.55}
\def\b{0.2}
\def\xx{(1-2*\b/\a)}
\draw (0,0.6*\u) -- (0,0) -- (1*\u,0) -- (1*\u,0.6*\u);
\coordinate (p1) at (  0*\u, 0*\u);
\coordinate (p2) at (  {\xx*\u}, {((\xx-1)*\a + 2*\b)*\u});
\coordinate (p3) at (  1*\u,\a*\u);
\coordinate (p4) at ( {(1-2*\b/\a)*\u}, 0*\u);
\coordinate (p5) at (1*\u,2*\b*\u);
\coordinate (p6) at ( {(1-\a)*\u}, 0*\u);
\draw[gray!50] (p3) -- (p6);
\draw[gray!50] (p4) -- (p5);
\draw[dotted] (0.15*\u,-0.1*\u) -- ++(0,0.7*\u);
\draw[dotted] (0.43*\u,-0.1*\u) -- ++(0,0.7*\u);
\draw[dashed,thick] (p1) -- (p2) -- (p3);
\draw[red,thick] (p1) -- (0.15*\u,0);
\draw[red,thick] plot[smooth] coordinates {(0.15*\u,0) (0.27*\u,0.03*\u) (0.57*\u,0.21*\u) (p3)};
\fill (p2) circle [radius = 1.2pt];
\fill (p4) circle [radius = 1.2pt];
\fill (p5) circle [radius = 1.2pt];
\fill (p6) circle [radius = 1.2pt];
\node[below left]  at (0,0) {\sz $0$};
\node[below]  at (1*\u,0) {\sz $1$};
\node[below, xshift=0ex]  at (p4) {\sz $u_1$};
\node[below, xshift=0ex]  at (0.15*\u,-0.1*\u) {\sz \rule{0pt}{5pt}$x_0$};
\node[below, xshift=0ex]  at (0.43*\u,-0.1*\u) {\sz \rule{0pt}{5pt}$x_1$};
\node[below, xshift=1.5ex]  at (p6) {\sz $1-a$};
\node[right]  at (1*\u,\a*\u) {\sz $a$};
\node[right]  at (p5) {\sz $2b$};
\node[below] at (current bounding box.south) {\sz\upshape $y=g_{1}(x)$ (dashed), $y=g(x)$};
\end{tikzpicture}
\rule{1cm}{0cm}
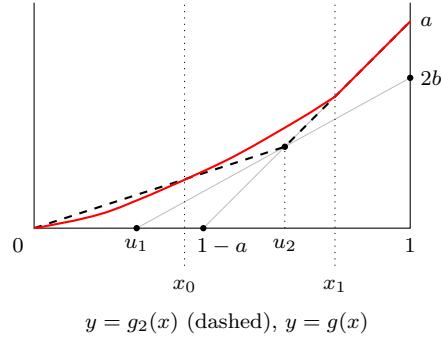
\begin{tikzpicture}
\def\u{5cm}
\def\a{0.55}
\def\b{0.2}
\def\xx{((1-2*\a+2*\b)/(1-\a))}
\draw (0,0.6*\u) -- (0,0) -- (1*\u,0) -- (1*\u,0.6*\u);
\coordinate (p1) at (  0*\u, 0*\u);
\coordinate (p2) at (  {\xx*\u}, {((\xx-1)*\a + 2*\b)*\u});
\coordinate (p3) at (  1*\u,\a*\u);
\coordinate (p4) at ( {(1-2*\b/\a)*\u}, 0*\u);
\coordinate (p5) at (1*\u,2*\b*\u);
\coordinate (p6) at ( {(1-\a)*\u}, 0*\u);
\draw[gray!50] (p4) -- (p5);
\draw[gray!50] (p3) -- (p6);
\draw[dotted] (0.40*\u,-0.1*\u) -- ++(0,0.7*\u);
\draw[dotted] (0.80*\u,-0.1*\u) -- ++(0,0.7*\u);
\draw[dotted] ({\xx*\u},0) -- (p2);
\draw[dashed,thick] (p1) -- (p2) -- (p3);
\draw[red,thick] (0.8*\u,0.35*\u) -- (p3);
\draw[red,thick] plot[smooth] coordinates {(p1) (0.20*\u,0.045*\u) (0.5*\u,0.175*\u) (0.7*\u,0.287*\u) (0.8*\u,0.35*\u)};
\fill (p2) circle [radius = 1.2pt];
\fill (p4) circle [radius = 1.2pt];
\fill (p5) circle [radius = 1.2pt];
\fill (p6) circle [radius = 1.2pt];
\node[below left]  at (0,0) {\sz $0$};
\node[below]  at (1*\u,0) {\sz $1$};
\node[below]  at (p4) {\sz $u_1$};
\node[below]  at ({\xx*\u},0) {\sz $u_2$};
\node[below, xshift=1.5ex]  at (p6) {\sz $1-a$};
\node[below, xshift=0ex]  at (0.40*\u,-0.1*\u) {\sz \rule{0pt}{5pt}$x_0$};
\node[below, xshift=0ex]  at (0.80*\u,-0.1*\u) {\sz \rule{0pt}{5pt}$x_1$};
\node[right]  at (1*\u,\a*\u) {\sz $a$};
\node[right]  at (p5) {\sz $2b$};
\node[below] at (current bounding box.south) {\sz\upshape $y=g_{2}(x)$ (dashed), $y=g(x)$};
\end{tikzpicture}
\]
Integrating $I(g_1)\leq I(g)\leq I(g_2)$ over $[0,1]$ gives
\[
\frac{2b^2}{3a} \leq c
\leq \frac{-a^2 + 2ab -4b^2 + 2b}{6(1-a)}.
\]
with
\begin{itemize}
\item
strict inequality on the left unless $g=g_1$, equivalently $f=_
{\fin}f_{u_1,v_1}$,

\item
strict inequality on the right unless $g=g_2$, equivalently $f=_{\fin}f_{u_2,v_2}$.
\end{itemize}
\end{prop}

\begin{proof}
(1) If $g(x)=0$ for $0\leq x\leq u_1$ then by convexity of $g$, the graph of $g$ must be below its secant on $[u_1,1]$, which coincides with the graph of $g_1$ on that interval. But then $g\leq g_1$, so the equality of the integrals of these continuous functions (both integrals $=b$) implies $g=g_1$.

Otherwise, since $g$ is continuous and increasing, there is a largest element $x_0\in [0,u_1)$ such that $g(x_0)=0$. We have $I(g_1)(x) = I(g)(x) = 0$ for $0\leq x\leq x_0$ and $0 = I(g_1)(x) < I(g)(x)$ for $x_0 < x\leq u_1$.
Since
\[
I(g_1)(u_1) + \int_{u_1}^1g_1(x)\,dx = I(g_1)(1) = b = I(g)(1) = I(g)(u_1) + \int_{u_1}^1g(x)\,dx,
\]
we have $\int_{u_1}^1g < \int_{u_1}^1g_1$ and therefore there exist points $x$ with $u_1 < x < 1$ where $g(x)<g_1(x)$.
By continuity of $g$ and $g_1$, there exists a point $x_1\in (u_1,1)$ where $g(x_1) = g_1(x_1)$.

\begin{claim}
$x_1$ is unique.
\end{claim}

If the set $A=\{x\in (u_1,1):g(x)=g_1(x)\}$ contains two distinct points,
say $x_1<x'_1$ both belong to $A$, then there are at least three points on $[x_1,1]$ where $g$ and $g_1$ agree (since $g(1)=a=g_1(1)$). Therefore $g=g_1$ on $[x_1,1]$ by convexity of $g$. Taking $x_1$ to be the least element of $A$, which exists since $g_1(u_1)<g(u_1)$, we have $g_1\leq g$ everywhere on $[0,1]$ and $g_1<g$ on $(x_0,x_1)$, so $I(g_1)(1)<I(g)(1)$, contradiction.

\m

It follows that $g_1<g$ on $(x_0,x_1)$ and $g<g_1$ on $(x_1,1)$ (since there do exist points where $g<g_1$). It is clear then that $I(g_1)(x)<I(g)(x)$ for all $x\in (x_0,x_1]$. Suppose that for some $x\in (x_1,1)$, we had $I(g)(x)\leq I(g_1)(x)$. Since $g<g_1$ on $(x,1)$, that would give $I(g)(1) < I(g_1)(1)$, a contradiction.

(2) Write $m$ for the slope of $g_2$ over the interval $[0,u]$, namely $m=(2b-a^2)/(1-2a+2b)$. If $D_+(g)(0)\geq m$ then since $Dg_-(1)\leq 1$ and the graph of $g$ lies above its support lines at $0$ and $1$, the graph of $g$ lies above the graph of $g_2$ and hence $g=g_2$ since the integrals are equal.

Thus, when $g\not=g_2$, we have $D_+(g)(0)<m$, which we now assume. It follows that $g(x)<g_2(x)$ for $x$ close enough to $0$.

\begin{claim}
There exists a unique point $x_0\in (0,u_2)$ such that $g(x_0)=g_2(x_0)$.
\end{claim}

Suppose we had $g(x)<g_2(x)$ for all $x\in (0,u_2)$. Then $g(u_2)\leq g_2(u_2)=v_2$ and hence $g(u_2)=v_2$ since the graph of $g$ lies above or on the line of slope $1$ through $(1,a)$. But then the graph of $g_2$ on $[u_2,1]$ is a secant through the graph of $g$ on that interval, and hence $g\leq g_2$ on $[u_2,1]$ by convexity, giving $I(g)(1)<I(g_2)(1)$, contradiction.  For uniqueness, if there were two such points $x_0<x'_0$, then together with $0$ there would be three points in $[0,x'_0]$ where $g$ agrees with $g_2$, and hence $g=g_2$ on $[0,x'_0]$. But then $D_+(g)(0)=m$, contradicting our assumption.

\m

It follows that $g<g_2$ on $(0,x_0)$ and $g_2<g$ on $(x_0,u_2)$. (If $g(x)<g_2(x)$ then $g$ being below its secant on $[0,x]$ would force $g(x_0)<g_2(x_0)$.) Then we also have $g_2(u_2)<g(u_2)$ because $g_2(u_2)=g(u_2)$ would give three points $0,x_0,u_2$ where $g=g_2$ and hence $g=g_2$ on $[0,u_2]$ contradicting that $g<g_2$ on $(0,x_0)$. Let $x_1$ be the least element of $(u_2,1]$ where $g(x_1)=g_2(x_1)$.  The graph of $g_2$ on $[x_1,1]$ is a secant through the graph of $g$ on that interval, and hence $g\leq g_2$ on $[x_1,1]$ by convexity. But the graph of $g$ does not go below that secant since $D_-g(1)\leq 1$, so $g=g_2$ on $[x_1,1]$. We thus have the stated relationship that $g<g_2$ on $(0,x_0)$, $g_2<g$ on $(x_0,x_1)$, and $g=g_2$ on $[x_1,1]$.

It is clear then that $I(g)(x)<I(g_2)(x)$ for all $x\in (0,x_0)$. Also, for $x\in [x_1,1]$, subtracting the equal integrals of $g$ and $g_2$ on $[x,1]$ from $I(g)(1)=I(g_2)(1)$ we see that $I(g)(x)=I(g_2)(x)$. In particular, $I(g)(x_1)=I(g_2)(x_1)$.  Suppose that for some $x\in (x_0,x_1)$, we had $I(g)(x)\geq I(g_2)(x)$. Since $g>g_2$ on $(x,x_1)$, that would give $I(g)(x_1) > I(g)(x_1)$, a contradiction.
\end{proof}

\begin{thm}\label{t:a-e}
Consider the following inequalities for real numbers $a,b,c$.
\begin{enumerate}[\rm(a)]
\item
$0\leq a\leq 1$

\item
$a^2/2\leq b\leq a/2$

\item
$2b^2\leq 3ac$

\item
$6(1-a)c\leq -a^2+2ab-4b^2+2b$

\item
$0\leq c\leq a/6$
\end{enumerate}
We have the following.
\begin{enumerate}
\item
Let $f\colon[0,1]\to[0,1]$ be increasing, and let $a=I(f)(1)$, $b=I^2(f)(1)$, $c=I^3(f)(1)$. Then {\rm(a)--(e)} hold.

\item
If $a,b,c$ are real numbers satisfying {\rm(a)--(e)}, then
there is an increasing function $f\colon[0,1]\to[0,1]$ such that $a=I(f)(1)$, $b=I^2(f)(1)$, $c=I^3(f)(1)$.

\item
The following table describes the conditions on $f$ in order for the inequalities in {\rm(a)--(e)} to hold with equality. When $0<a<1$, we let
\[
l=\frac{2b^2}{3a},\ \ r=\frac{-a^2+2ab-4b^2+2b}{6(1-a)}.
\]
\begin{enumerate}[\rm(i)]
\item
The first column of the table lists $7$ exhaustive and mutually exclusive conditions on $(a,b,c)$ that can hold under {\rm(a)--(e)}.

\item
In each numbered row, the second column lists the increasing functions $f\colon[0,1]\to[0,1]$ which satisfy the condition.
Except in row $7$, the listed function is either constant or a $2$-step step function and is uniquely determined up to equality modulo a finite set.
The \raisebox{1pt}{\sz $\bigstar$} in row $7$ indicates that all increasing functions $f\colon[0,1]\to[0,1]$ which are not $=_\fin$ to the ones in rows $1$--$6$ satisfy that condition.

\item
The third column gives the corresponding value of the triple $(a,b,c)$, where $a=I(f)(1)$, $b=I^2(f)(1)$, $c=I^3(f)(1)$. The rest of the columns indicate which of the eight inequalities in {\rm(a)--(e)} hold with equality.
\end{enumerate}
\end{enumerate}
\end{thm}

{\sz
\[
\rule{1cm}{0cm}\begin{array}{|r|l|c|l|c|c|c|c|c|c|c|c|}
\cline{5-12}
\multicolumn{3}{c}{} & & \multicolumn{8}{c|}{\rule{0pt}{10pt}\text{Inequalities which hold with equality}} \\[2pt]
\hline
& \rule{0pt}{10pt}\text{Condition} & f & (a,b,c) & 0\leq a & a\leq 1 & a^2/2\leq b & b\leq a/2 & {\rm(c)} & {\rm(d)} & 0\leq c & c\leq a/6 \\[2pt]
\hline
1. & \rule{0pt}{10pt}a=0      & 0 & (0,0,0)     & \cm &     & \cm & \cm & \cm & \cm & \cm & \cm \\
2. & \rule{0pt}{10pt}a=1      & 1 & (1,1/2,1/6) &     & \cm & \cm & \cm & \cm & \cm &     & \cm \\[2pt]
\hline
& \rule{0pt}{10pt}\underline{0<a<1} &&&&&&&&&&\\
3. & \rule{0cm}{10pt}\ \ \ \bullet\ b=a^2/2 & f_{1-a,\,0} & (a,a^2/2,a^3/6) &     &     & \cm &    & \cm & \cm &    & \\
4. & \rule{0pt}{10pt}\ \ \ \bullet\ b=a/2 & a & (a,a/2,a/6) &     &     &     & \cm & \cm & \cm &  & \cm \\[2pt]
\hline
& \rule{0pt}{10pt}\underline{0<a<1} &&&&&&&&&&\\
& \rule{0pt}{10pt}\underline{a^2/2<b<a/2} &&&&&&&&&&\\
5. & \rule{0pt}{10pt}\ \ \ \bullet\ c=l & f_{u_1,v_1} & (a,b,l) &     &     & &  & \cm &  &     &  \\
6. & \rule{0pt}{10pt}\ \ \ \bullet\ c=r & f_{u_2,v_2} & (a,b,r) &     &     & &  &  & \cm &  &  \\[2pt]
\hline
& \rule{0pt}{10pt}\underline{0<a<1} &&&&&&&&&&\\
7. & \rule{0pt}{10pt}\underline{a^2/2<b<a/2} & \bigstar & (a,b,c) &&&&&&&&\\
& \rule{0pt}{10pt}\underline{l<c<r} &&&&&&&&&&\\[2pt]
\hline
\end{array}
\]}

\m

\begin{rem}
It follows, as claimed in Theorem C in the introduction, that all inequalities are strict unless $f$ is equal modulo finite to a constant function (rows 1, 2, 4) or a $2$-step step function having value $0$ on the first step or value $1$ on the second step (rows 3, 5, 6). (See Remark \ref{r:u1u2} (b) and Proposition \ref{p:two.step.0.1}.)
\end{rem}

\begin{rem}\label{r:converse}
We make two observations concerning conclusion (2) of the theorem.

(i) In the presence of (c) and (d), clauses (a) and (b) are equivalent by Proposition \ref{p:prelim.ineq} (2)(iii), so one of them could be dropped.

(ii) Clause (e) follows from (c) and (d) if $0<a<1$ because then (c) clearly implies $c\geq 0$ and by Proposition \ref{p:prelim.ineq} (2)(vi), (d) implies $c\leq a/6$. However, (e) is needed for (2), to cover the cases $a=0$ and $a=1$. When $a=0$, by Proposition \ref{p:prelim.ineq} (2)(i), any triple $(a,b,c)=(0,0,c)$ satisfies (c) and (d) as long as $c\leq 0$. However, any increasing $f\colon[0,1]\to[0,1]$ with $a=I(f)(1)=0$ satisfies $f(x)=0$ for $0<x<1$, so $c=I^3(f)(1) = 0$. And when $a=1$, by Proposition \ref{p:prelim.ineq} (2)(ii), any triple $(a,b,c)=(1,1/2,c)$ satisfies (c) and (d) as long as $c\geq 1/6$. However, any increasing $f\colon[0,1]\to[0,1]$ with $a=I(f)(1)=1$ satisfies $f(x)=1$ for $0<x<1$, so $I(f)(x)=x$, $I^2(f)(x)=x^2/2$, $I^3(f)(x)=x^3/6$, giving $c=I^3(f)(1)=1/6$.
\end{rem}

\begin{proof}
(1) We have clauses (a) and (b) by Propositions \ref{n=1.a} and \ref{n=2.a}, respectively. When $a=0$, necessarily $f(x)=0$ for $0<x<1$, so $b=c=0$ as well and hence (c), (d), (e) all hold. When $a=1$, necessarily $f(x)=1$ for $0<x<1$, giving $b=I^2(f)(1)=1/2$, $c=I^3(f)(1)=1/6$ and hence (c), (d), (e) all hold.

When $0<a<1$, Proposition \ref{n=2.a}(4) states that if $b=I^2(f)(1)=a^2/2$ then $f(x)=0$ for $0<x<1-a$ and $f(x)=1$ for $1-a<x<1$. Thus, $I(f)(x)=(x-(1-a))_+$, $I^2(f)(x)=(x-(1-a))_+^2/2$, $I^3(f)(x)=(x-(1-a))_+^3/6$ and hence $c=I^3(f)(1)=a^3/6$. And if $b=I^2(f)(1)=a/2$ then $f(x)=a$ for $0<x<a$. Thus, $I(f)(x)=ax$, $I^2(f)(x)=ax^2/2$, $I^3(f)(x)=ax^3/6$ and hence $c=I^3(f)(1)=a/6$. In both cases, (c) and (d) hold by Proposition \ref{p:prelim.ineq} (1) and (e) clearly holds.

When $0<a<1$, and $a^2/2<b<a/2$, (c) and (d) hold by Proposition \ref{p:n.equals.3} (2), and then (e) holds by Proposition \ref{p:prelim.ineq} (2)(vi).

(2) When $a=0$, from (b) we get $b=0$ and from (d) and (e) we get $c=0$. We can take $f(x)=0$ for all $x$.
When $a=1$, from (b) we get $b=1/2$, and from (c) and (e) we get $c=1/6$. We can take $f(x)=1$ for all $x$.

Now suppose $0<a<1$. If $b=a^2/2$ then by Proposition \ref{p:prelim.ineq} (2)(iii) and (1), we have $c=a^3/6$. We can take $f=f_{1-a,\,0}$.
This function has the required values of $I^n(f)(1)$, $n=1,2,3$.
If $b=a/2$ then by Proposition \ref{p:prelim.ineq} (2)(vi) and (1), we have $c=a/6$. We can take
$f(x)=a$ for all $x$. This function clearly has the required values of $I^n(f)(1)$, $n=1,2,3$.

Finally suppose that $0<a<1$ and $a^2/2<b<a/2$. By (c) and (d), we have $I(g_1)(1)\leq c\leq I(g_2)(1)$. By Proposition \ref{p:g1g2}, there is a point $(u,v)$ on the line denoted $L$ there, with $u_1\leq u\leq u_2$, such that $I^3(f_{uv})(1)=c$. Since $(u,v)\in L$, we have $I^2(f)(1)=b$, and since $u_1\leq u\leq u_2$, $(u,v)$ belongs to the closure of the triangle $T$ and therefore $f_{uv}$ is increasing (Proposition \ref{p:T}) and $I(f_{uv})(1)=a$.

(3) The information in the rows $a=0$ and $a=1$ is provided by Proposition \ref{n=1.a} with the values of $b,c$ being easy computations.
When $0<a<1$, for the rows $b=a/2$ and $b=a^2/2$, we use Proposition \ref{n=2.a} (4) with the values of $c$ being easy computations.

Then in the remaining rows, we have $0<a<1$ and $a^2/2<b<a/2$. For the rows $c=l$ and $c=r$, we use Proposition \ref{p:n.equals.3}. By Proposition \ref{p:g1g2}, $l<r$, so each of $c=l$, $c=r$ implies the negation of the other. That $c<a/6$ follows from Proposition \ref{p:prelim.ineq} (vi) from which we see that $c=a/6$ only holds when $b=a/2$.
\end{proof}

\section{The main theorem}
\label{s:fuv}

In this section, we establish our main theorem.

\begin{thm}
\label{t:n=3.b}
Let $a,b,c$ be positive numbers satisfying
\[
0<a<1,\ \ \frac{a^2}2 < b < \frac{a}2,\ \ \frac{2b^2}{3a} < c < \frac{-a^2 + 2ab -4b^2 + 2b}{6(1-a)}.
\]
For each small enough $\de>0$, there is a $C^\infty$ function $f\colon [0,1]\to[0,1]$ such that $f=\s_\de$ on $[0,\de/2]$, $f=\tau_\de$ on $[1-\de/2,1]$, $Df>0$ on $(0,1)$, $I(f)(1)=a$, $I^2(f)(1)=b$ and $I^3(f)(1)=c$.
\end{thm}

\begin{rem}
Either of the first two displayed conditions could be omitted since the first condition is an easy consequence of the second, and, by Proposition \ref{p:prelim.ineq} (2)(iv), the second follows from the first and third..
\end{rem}

Fix $a,b,c$ as in the statement of the theorem. Recall that for $0<u<1$, $v,t\in\R$, we have
\[
f_{uvt}(x) = \begin{cases}
v/u & \text{for $0\leq x<u$} \\
(t-v)/(1-u) & \text{for $u\leq x\leq 1$},
\end{cases}
\]
and we have $I(f_{uvt})(u)=v$ and $I(f_{uvt})(1) = t$.

Most clauses in the following definition correspond to the inequalities in Proposition \ref{p:T} (2)(a) which describe when $(u,v)$ belongs to the open triangle $T$ defined in (2)(b) of that proposition.

\begin{defn}
Let $S$ be the open subset of $\R^4$ defined by letting $(\de,u,v,t)\in S$ when the following conditions hold:
\begin{itemize}
\item
$0<\de<1$

\item
$3\de < u < 1-3\de$

\item
$2\de<v/u$

\item
$v/u + 2\de < (t-v)/(1-u) < 1-3\de$
\end{itemize}
\end{defn}

\begin{prop}\label{p:S*}
There are continuous functions $\de,t_1,t_2\colon T\to\R$  such that for each $(u,v)\in T$, $\de(u,v)>0$, $t_1(u,v) + 6\de(u,v) < a < t_2(u,v) - 6\de(u,v)$
and when $0<\de<\de(u,v)$ and $t_1(u,v) \leq t \leq t_2(u,v)$ we have $(\de,u,v,t)\in S$.
\end{prop}

\begin{proof}
For each $(u,v)\in T$ we have, by Proposition \ref{p:T}, $v/u < (a-v)/(1-u)<1$. Using these inequalities, it is easily verified that the solutions $f(u,v)$ and $g(u,v)$ to the equations
\[
\frac{v}u + 2f(u,v) = \frac{a-6f(u,v)-v}{1-u}\ \text{and}\ \frac{a+6g(u,v)-v}{1-u} = 1 - 3g(u,v)
\]
are positive continuous functions on $T$. Four of the inequalities defining $(\de,u,v,t)\in S$ state that $\de<\varphi_i(u,v)$, $i=1,2,3,4$, where $\varphi_1(u,v)=1$, $\varphi_2(u,v)=u/3$, $\varphi_3(u,v)=(1-u)/3$, $\varphi_4(u,v)=v/(2u)$. Take $\de(u,v) = \min((1/2)f(u,v),\,(1/2)g(u,v),\,\varphi_i(u,v)\,(i=1,2,3,4))$. Then
\[
\frac{v}u + 2\de(u,v) < \frac{a-6\de(u,v)-v}{1-u} < \frac{a+6\de(u,v)-v}{1-u} < 1 - 3\de(u,v).
\]
Define $t_1(u,v)$ and $t_2(u,v)$ so that $(t_1(u,v)-v)/(1-u)$ is the midpoint between $v/u+2\de(u,v)$ and $(a-6\de(u,v)-v)/(1-u)$ and $(t_2(u,v)-v)/(1-u)$ is the midpoint between $(a+6\de(u,v)-v)/(1-u)$ and $1-3\de(u,v)$. Then $t_1$ and $t_2$ are clearly continuous.

Suppose $(u,v)\in T$, $0<\de<\de(u,v)$ and $t_1(u,v) \leq t \leq t_2(u,v)$. We want to verify that $(\de,u,v,t)\in S$. From the definitions of $t_1$ and $t_2$, we have
\[
\frac{v}u + 2\de(u,v) < \frac{t_1(u,v)-v}{1-u}<\frac{a-6\de(u,v)-v}{1-u} < \frac{a+6\de(u,v)-v}{1-u} < \frac{t_2(u,v)-v}{1-u} < 1 - 3\de(u,v).
\]
The second and fourth inequalities yield $t_1(u,v) + 6\de(u,v) < a < t_2(u,v) - 6\de(u,v)$.
Since $\de<\de(u,v)$ and $t_1(u,v) \leq t \leq t_2(u,v)$, we have
\[
\frac{v}u + 2\de < \frac{t-v}{1-u}< 1 - 3\de.
\]
Other than $\de>0$, the remaining inequalities in the definition of $(\de,u,v,t)\in S$ follow from $\de<\de(u,v)\leq \varphi_i(u,v)$, $i=1,2,3,4$.
\end{proof}

We state the following simple fact as a lemma for ease of reference.

\begin{lem}\label{l:delta.room}
Let $0<\de<1$. Let $A=(a_1,a_2)$, $B=(b_1,b_2)$ be points on a line $M$ of slope $1/\de$, with $a_1<b_1$.
Let $L$ be the line of slope $\de$ through $A$. Let $d$ be the distance from $B$ to $L$ measured along the vertical line through $B$. For any real number $r$, we have the following.
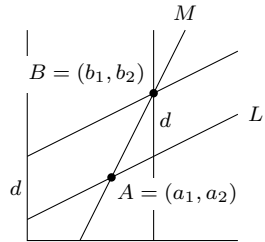
\[
\begin{tikzpicture}
\def\u{2.8cm}
\draw (1*\u,0) -- (0,0) -- (0,1*\u);
\draw (0.25*\u,0) -- (0.75*\u,1*\u) node[above] {\sz $M$};
\draw (0,0.1*\u) -- (1*\u,0.6*\u) node[right] {\sz $L$};
\draw (0,0.4*\u) -- (1*\u,0.9*\u) node[right] {};
\draw ({(3/5)*\u},0) -- ++(0,1*\u);
\coordinate (A) at ({(2/5)*\u},{(3/10)*\u});
\coordinate (B) at ({(3/5)*\u},{(7/10)*\u});
\coordinate (C) at ({(3/5)*\u},0.4*\u);
\coordinate (D) at ($(B)!0.5!(C)$);
\node[below right,fill=white,inner sep=2] at (A) {\sz $A=(a_1,a_2)$};
\node[above left,fill=white,inner sep=3] at (B) {\sz $B=(b_1,b_2)$};
\node[right,inner sep=2,yshift=2pt] at (D) {\sz $d$};
\node[left,inner sep=2] at (0,0.25*\u) {\sz $d$};
\fill (A) circle [radius=0.02*\u];
\fill (B) circle [radius=0.02*\u];
\end{tikzpicture}
\]
\begin{enumerate}
\item
$b_2-a_2>r$ if and only if $b_1-a_1>\de r$, and the same holds with $>$ replaced by $<$ or $=$.

\item
$b_2-a_2>r$ if and only if $d>(1-\de^2)r$. In particular, $b_2-a_2>d$.
\end{enumerate}
\end{lem}

\begin{proof}
(1) We have $\de(b_2-a_2) = (b_1-a_1)$ and hence $b_2-a_2>r$ if and only if $b_1-a_1>\de r$ where we can replace $>$ by $<$ or $=$.

(2) The lines of slope $\de$ through $A$ and $B$ have equations $y=\de (x-a_1) + a_2$ and $y=\de (x-b_1) + b_2$, respectively.  The distance from $B$ to $L$ measured along the vertical line through $B$ is the difference of the $y$-intercepts of these lines, $d = (b_2-\de b_1) - (a_2 - \de a_1) = (b_2 - a_2) - \de(b_1 - a_1) = (b_2-a_2)(1 - \de^2)$. It follows that $b_2-a_2>r$ is equivalent to $d>(1-\de^2)r$. When $r=d$, the inequality $d>(1-\de^2)r$ holds, so $b_2-a_2>d$.
\end{proof}

Fix $(\de,u,v,t)\in S$. Note that this implies that $f_{uvt}$ maps into $(0,1)$ with $v/u<(t-v)/(1-u)$.
To simplify the description that follows, write $m_1=v/u$, $m_2 = (t-v)/(1-u)$. Our assumption on $(\de,u,v,t)$ is thus that $0<\de<1$, $3\de < u < 1-3\de$ and
\[
2\de<m_1 < m_1 + 2\de < m_2 < 1-3\de.
\]
For each $(\de,u,v,t)\in S$, define a continuous function $\psi_{\de u v t}$ to be $\s_\de$ on $[0,\de]$, to be $\tau_\de$ on $[1-\de,1]$, and to be the piecewise linear function on $[\de,1-\de]$ given by joining with line segments the points $A\to B\to C\to D\to E\to F$, where $A=(\de,\de)$, $F = (1-\de,1-\de)$, $B$ and $C$ are chosen on the line $y=m_1+\de x$ where it meets the lines of slope $1/\de$ through $A$ and $(u,0)$, respectively, and $D$ and $E$ are chosen on the line $y=m_2+\de x$ where it meets the lines of slope $1/\de$ through $(u,0)$ and $F$, respectively.
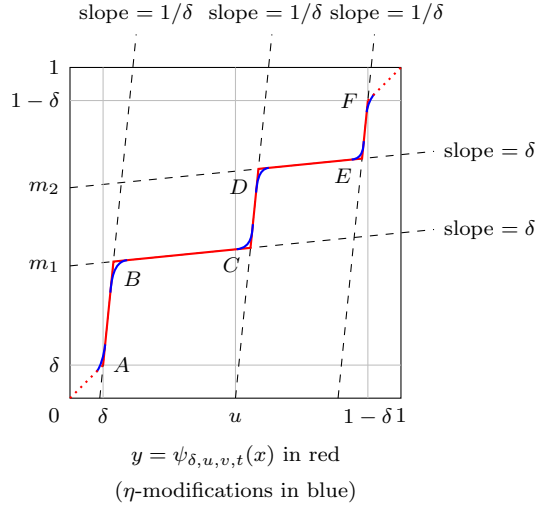
\[
\begin{tikzpicture}
\def\u{4.4cm}
\def\delt{0.1}
\def\a{0.55}
\def\bb{0.2} 
\def\mm{0.4} 
\def\myu{0.5} 
\pgfmathsetmacro\mmm{1 - \mm*\myu/\a}
\pgfmathsetmacro\us{1 - 2*\bb/\a}          
\pgfmathsetmacro\ue{(1-2*\a+2*\bb)/(1-\a)} 
\pgfmathsetmacro\bx{(\delt/(1-\delt^2))*(\mm + 1 - \delt)}
\pgfmathsetmacro\by{\mm + \delt*\bx}
\pgfmathsetmacro\cx{(\delt/(1-\delt^2))*(\mm + \myu/\delt)}
\pgfmathsetmacro\cy{\mm + \delt*\cx}
\pgfmathsetmacro\dx{(\delt/(1-\delt^2))*(\mmm - \cy + \cx/\delt)}
\pgfmathsetmacro\dy{\mmm + \delt*\dx}
\pgfmathsetmacro\ex{(\delt/(1-\delt^2))*(\mmm - 2 + \delt + 1/\delt)}
\pgfmathsetmacro\ey{\mmm + \delt*\ex}
\pgfmathsetmacro\ix{\delt - \delt^2}
\pgfmathsetmacro\jx{\myu}
\pgfmathsetmacro\kx{1 - 2*\delt + \delt^2}
\coordinate (A) at (\delt*\u,\delt*\u);
\coordinate (B) at (\bx*\u,\by*\u);
\coordinate (C) at (\cx*\u,\cy*\u);
\coordinate (D) at (\dx*\u,\dy*\u);
\coordinate (E) at (\ex*\u,\ey*\u);
\coordinate (F) at ({(1-\delt)*\u},{(1-\delt)*\u});
\draw (0,0) -- (1*\u,0) -- (1*\u,1*\u) -- (0,1*\u) -- cycle;
\draw[gray!50] (0,\delt*\u) -- ++(1*\u,0)
	(0,{(1-\delt)*\u}) -- ++(1*\u,0);
\draw[dashed]
	(0,\mm*\u) -- ++(1.1*\u,\delt*1.1*\u) node[right] {\sz slope $=\de$}
	(0,\mmm*\u) -- ++(1.1*\u,\delt*1.1*\u) node[right] {\sz slope $=\de$};
\draw[gray!50] (\delt*\u,0) -- ++(0,1*\u)
	(\myu*\u,0) -- ++(0,1*\u)
	({(1-\delt)*\u},0) -- ++(0,1*\u);
\draw[dashed]
	(\ix*\u,0) -- ++(\delt*1.1*\u,1.1*\u) node[above] {\sz slope $=1/\de$}
	(\jx*\u,0) -- ++(\delt*1.1*\u,1.1*\u) node[above,xshift=-0.5ex] {\sz slope $=1/\de$}
	(\kx*\u,0) -- ++(\delt*1.1*\u,1.1*\u) node[above,xshift=1ex] {\sz slope $=1/\de$};
    \coordinate (start0) at ($(0,0)!0.8!(A)$);
    \coordinate (end0)   at ($(A)!0.2!(B)$);
    \coordinate (start1) at ($(A)!0.7!(B)$);
    \coordinate (end1)   at ($(B)!0.1!(C)$);
    \coordinate (start2) at ($(B)!0.9!(C)$);
    \coordinate (end2)   at ($(C)!0.3!(D)$);
    \coordinate (start3) at ($(C)!0.7!(D)$);
    \coordinate (end3)   at ($(D)!0.1!(E)$);
    \coordinate (start4) at ($(D)!0.9!(E)$);
    \coordinate (end4)   at ($(E)!0.3!(F)$);
    \coordinate (start5) at ($(E)!0.8!(F)$);
    \coordinate (end5)   at ($(F)!0.2!(1*\u,1*\u)$);
    \coordinate (c1) at ($(start0) + 0.5*(\delt,\delt)$);
    \coordinate (c2) at ($(end0) - 0.45*(\bx-\delt,\by-\delt)$);
    \coordinate (c3) at ($(start1) + 0.5*(\bx-\delt,\by-\delt)$);
    \coordinate (c4) at ($(end1) - 0.45*(\cx-\bx,\cy-\by)$);
    \coordinate (c5) at ($(start2) + 0.5*(\cx-\bx,\cy-\by)$);
    \coordinate (c6) at ($(end2) - 0.45*(\dx-\cx,\dy-\cy)$);
    \coordinate (c7) at ($(start3) + 0.5*(\dx-\cx,\dy-\cy)$);
    \coordinate (c8) at ($(end3) - 0.45*(\ex-\dx,\ey-\dy)$);
    \coordinate (c9) at ($(start4) + 0.5*(\ex-\dx,\ey-\dy)$);
    \coordinate (c10) at ($(end4) - 0.45*(1-\delt-\ex,1-\delt-\ey)$);
    \coordinate (c11) at ($(start5) + 0.5*(1-\delt-\ex,1-\delt-\ey)$);
    \coordinate (c12) at ($(end5) - 0.45*(\delt,\delt)$);
\draw[red,thick] (A) -- (B) -- (C) -- (D) -- (E) -- (F);
    \draw[blue,thick] (start0) .. controls (c1) and (c2) .. (end0);
    \draw[blue,thick] (start1) .. controls (c3) and (c4) .. (end1);
    \draw[blue,thick] (start2) .. controls (c5) and (c6) .. (end2);
    \draw[blue,thick] (start3) .. controls (c7) and (c8) .. (end3);
    \draw[blue,thick] (start4) .. controls (c9) and (c10) .. (end4);
    \draw[blue,thick] (start5) .. controls (c11) and (c12) .. (end5);
\draw[red,thick,dotted] (0,0) -- (1*\delt*\u,1*\delt*\u);
\draw[red,thick,dotted] ({(1-\delt)*\u},{(1-\delt)*\u}) -- (1*\u,1*\u);
\node[left]  at (0,\delt*\u) {\sz $\de$};
\node[left]  at (0,\mm*\u) {\sz $m_1$};
\node[left]  at (0,\mmm*\u) {\sz $m_2$};
\node[left]  at (0,{(1-\delt)*\u}) {\sz $1-\de$};
\node[left]  at (0,1*\u) {\sz $1$};
\node[below left]  at (0,0) {\sz $0$};
\node[below]  at (1*\u,0) {\sz $1$};
\node[below]  at (\delt*\u,0) {\sz $\de$};
\node[below]  at (\myu*\u,0) {\sz \rule{0pt}{5pt}$u$};
\node[below]  at ({(1-\delt)*\u},0) {\sz $1-\de$};
\node[right] at (A) {\sz $A$};
\node[below right] at (B) {\sz $B$};
\node[below left] at (C) {\sz $C$};
\node[below left] at (D) {\sz $D$};
\node[below left] at (E) {\sz $E$};
\node[left] at (F) {\sz $F$};
\node[below,xshift=-3ex] at (current bounding box.south) {\sz\upshape $y=\psi_{\de,u,v,t}(x)$ in red};
\node[below,xshift=-3ex,yshift=0.5ex] at (current bounding box.south) {\sz\upshape ($\eta$-modifications in blue)};
\end{tikzpicture}
\]
Write $A=(a_x,a_y)=(\de,\de)$, $B=(b_x,b_y)$, $C=(c_x,c_y)$, $D=(d_x,d_y)$, $E=(e_x,e_y)$, $F=(f_x,f_y)=(1-\de,1-\de)$. The following inequalities are useful for locating the $x$ coordinates of these points.

\begin{lem}\label{l:ineq}
We have the following inequalities, where $\eta=\de^2/2$.
\begin{enumerate}
\item\label{lem.1}
$a_x + \eta < b_x - \eta < b_x + \de^2 < 2\de$,

\item\label{lem.2}
$u<c_x - \de^2 < c_x + \eta < d_x - \eta < d_x+\de^2 < u + \de$,

\item\label{lem.4}
$1-2\de < e_x-\de^2 < e_x+\eta < f_x-\eta$.
\end{enumerate}
\end{lem}

Note that we could have written this as one long sequence of inequalities since $2\de<3\de<u$ and $u+\de < (1-3\de)+\de = 1-2\de$. The graph below illustrates the inequalities.
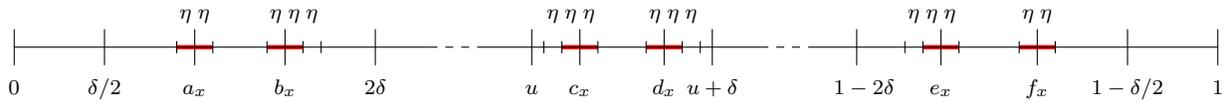
\[
\begin{tikzpicture}
\def\u{16cm}
\def\t{0.015}
\def\ta{0.005}
\draw[ultra thick,red] (0.135*\u,0) -- (0.165*\u,0)
(0.21*\u,0) -- (0.24*\u,0)
(0.455*\u,0) -- (0.485*\u,0)
(0.525*\u,0) -- (0.555*\u,0)
(0.755*\u,0) -- (0.785*\u,0)
(0.835*\u,0) -- (0.865*\u,0);
\draw (0,0) -- (1*\u,0);
\foreach \i/\j in {0/{0},0.075/{\de/2},0.15/{a_x},0.225/{b_x},0.30/{2\de},0.43/{u},0.47/{c_x},
0.54/{d_x},0.58/{u+\de},0.7/{\ \ 1-2\de},
0.77/{e_x},0.85/{f_x},0.925/{1-\de/2},1/{1}}
{
    \draw (\i*\u,-\t*\u) -- ++(0,2*\t*\u);
    \node[below] at (\i*\u,-\t*\u) {\sz \rule{0pt}{7pt}$\j$};
}
\draw[ultra thick,white] (0.345*\u,0) -- (0.385*\u,0);
\draw[dashed] (0.345*\u,0) -- (0.385*\u,0);
\draw[ultra thick,white] (0.62*\u,0) -- (0.66*\u,0);
\draw[dashed] (0.62*\u,0) -- (0.66*\u,0);
\foreach \i in {0.135,0.165,0.21,0.24,0.255,0.44,0.455,0.485,0.525,
0.555,0.57,0.74,0.755,0.785,0.835,0.865}
{
    \draw (\i*\u,-\ta*\u) -- ++(0,2*\ta*\u);
}
\foreach \i in {0.1425,0.1575,0.2175,0.2325,0.2475,0.4475,0.4625,0.4775,0.5325,0.5475,0.5625,
0.7475,0.7625,0.7775,0.8425,0.8575}
{
    \node[above] at (\i*\u,0.8*\t*\u) {\sz $\eta$};
}
\end{tikzpicture}
\]
(The inequalities $\de/2<a_x-\eta$ and $f_x+\eta < 1-\de/2$ hold because they state that $\de/2 < \de - \de^2/2$ and $1-\de + \de^2/2 < 1-\de/2$, respectively, both of which reduce to $\de<1$.)

\begin{proof}
We make repeated use of Lemma \ref{l:delta.room} without mention. Also note that since $B$ is on the line $y=m_1 + \de x$, $b_y=m_1 + \de b_x < m_1+\de$. Similarly $c_y<m_1+\de$ and $d_y,e_y<m_2+\de$.

\ref{lem.1} The slope of the line from $A$ to $B$ is $1/\de$. For the first inequality we have $b_y-a_y > m_1 - \de > \de$, so $b_x - a_x > \de^2$.
The middle one is clear.
The third can be written $b_x-a_x < \de - \de^2 = \de(1-\de)$ which
follows if we show $b_y-a_y<1-\de$, and this holds since $a_y=\de$ and $b_y<m_1+\de<1$.

\ref{lem.2} The slope of the line from $P=(p_x,p_y)=(u,0)$ to $C$ and to $D$ is $1/\de$. The second and fourth inequalities are clear. For the first, since $c_y - p_y = c_y > m_1 > \de$, we get $c_x - p_x = c_x - u > \de^2$.
For the middle inequality, we have $d_y-c_y > m_2-m_1>\de$, so $d_x-c_x > \de^2 = 2\eta$.
For the last inequality,
we have $d_y - p_y = d_y < m_2+\de$, giving $d_x - p_x = d_x - u < \de(m_2 + \de) < \de((1-3\de)+\de) = \de - 2\de^2 < \de - \de^2$.

\ref{lem.4} The slope of the line from $E$ to $F$ is $1/\de$. The first inequality can be written $f_x-e_x < \de - \de^2 = \de(1-\de)$ which follows from $f_y-e_y<1-\de$,
which in turn holds because $f_y=1-\de$ and $e_y>m_2>0$. The second inequality is clear, and for the third we have $f_y - e_y > (1-\de) - (m_2 + \de) > (1-\de) - (1-2\de) = \de$, so $f_x - e_x > \de^2 = 2\eta$.
\end{proof}

\begin{defn}
Let $\rho_{\de u v t}$ be obtained from $\psi_{\de u v t}$ by replacing the restriction to $[\de/2,1-\de/2]$ by its $\eta$-modification at the points $A,B,C,D,E,F$, where $\eta=\de^2/2$.
\end{defn}

As the inequalities in Lemma \ref{l:ineq} show, the inequalities $a_i+2\eta < a_{i+1}$ in Corollary \ref{c:roundoffcorners} (ensuring that the intervals over which the function is being modified do not overlap) are satisfied. (In the graph following the statement of Lemma \ref{l:ineq} above, these intervals are highlighted.)

\begin{lem}\label{l:cont.rho}
The map $(\de,u,v,t)\mapsto \rho_{\de u v t}$ from $S$ into $C^\infty([0,1])$ is continuous.
\end{lem}

\begin{proof}
The function $\rho_{\de u v t}$ agrees with $\s_\de$ on $[0,\de/2]$, with $G^6_\gamma$ on $[\de/2,1-\de/2]$, and with $\tau_\de$ on $[1-\de/2,1]$, where
\[
\gamma = \gamma_{\de u v t} = (1,1,A,B,C,D,E,F,\de).
\]
It is easy to solve explicitly for the coordinates of $A,B,C,D,E,F$, for example $A=(\de,\de)$, $B=(\de(m+1-\de)/(1-\de^2),\,m + \de^2(m+1-\de)/(1-\de^2))$ (where $m=v/u$), to see that they depend continuously on $(\de,u,v,t)$. It follows from Corollary \ref{c:roundoffcorners} that $(\de,u,v,t)\mapsto \gamma_{\de u v t}\mapsto G^6_{\gamma_{\de u v t}}$ is continuous. Since $(\de,u,v,t)\mapsto \de\mapsto \s_\de$ and $(\de,u,v,t)\mapsto \de\mapsto \tau_\de$ are continuous, Proposition \ref{p:C.infty.by.cases} and Proposition \ref{p:restr.cont} imply that $(\de,u,v,t)\mapsto \rho_{\de u v t}$ is continuous.
\end{proof}

\begin{lem}\label{l:g.incre}
Let $(\de,u,v,t)$ and $(\de,u,v',t')$ be two elements of $S$ sharing the same $\de$ and  $u$. Let
\[
m_1 = v/u,\ \ m_2 = (t-v)/(1-u),\ \ m'_1 = v'/u,\ \ m'_2 = (t'-v')/(1-u).
\]
Assume that $m_1\leq m'_1$ and $m'_1+\de<\min(m_2,m'_2)$.
Let $A,B,C,D,E,F$ be as in the definition of $\rho_{\de u v t}$, and let
$A=A',B',C',D',E',F'=F$ be the corresponding points in the definition of $\rho_{\de u v' t'}$. Write
$A=(a_x,a_y)$, $B=(b_x,b_y)$, and so on, and $A'=(a'_x,a'_y)$, $B'=(b'_x,b'_y)$, and so on.
Let $I_1 = (b_x-\eta,c'_x+\eta)$. Let $I_2 = (d_x-\eta,e'_x+\eta)$ if $m_2\leq m'_2$, and $I_2=(d'_x-\eta,e_x+\eta)$ otherwise. Then the functions $\rho_{\de u v t}$ and $\rho_{\de u v' t'}$ agree outside $I_1\cup I_2$.
\begin{itemize}
\item
On $I_1$: If $m_1=m'_1$ then $\rho_{\de u v t} = \rho_{\de u v' t'}$. If $m_1<m'_1$ then $\rho_{\de u v t} < \rho_{\de u v' t'}$. In both cases, $\rho_{\de u v' t'} - \rho_{\de u v t}$ has maximum value $m'_1 - m_1$.

\item
On $I_2$: If $m_2=m'_2$ then $\rho_{\de u v t} = \rho_{\de u v' t'}$. If
$m_2 < m'_2$ then $\rho_{\de u v t} < \rho_{\de u v' t'}$, with the inequality reversed if $m'_2<m_2$. In all cases, $|\rho_{\de u v' t'} - \rho_{\de u v t}|$ has maximum value $|m'_2 - m_2|$.
\end{itemize}
On $[0,1]$, we therefore have
$|\rho_{\de u v' t'} - \rho_{\de u v t}|\leq \max(|m_1-m'_1|,|m_2-m'_2|)$.
\[
\begin{tikzpicture}
\def\u{4cm}
\def\delt{0.1}
\def\a{0.55}
\def\bb{0.2} 
\def\mm{0.4} 
\def\mmq{0.25} 
\def\myu{0.5} 
\pgfmathsetmacro\mmm{1 - \mm*\myu/\a - 0.02}
\pgfmathsetmacro\mmmm{\mmm + 0.09}
\pgfmathsetmacro\us{1 - 2*\bb/\a}          
\pgfmathsetmacro\ue{(1-2*\a+2*\bb)/(1-\a)} 
\pgfmathsetmacro\bx{(\delt/(1-\delt^2))*(\mm + 1 - \delt)}
\pgfmathsetmacro\by{\mm + \delt*\bx}
\pgfmathsetmacro\bxq{(\delt/(1-\delt^2))*(\mmq + 1 - \delt)}
\pgfmathsetmacro\byq{\mmq + \delt*\bxq}
\pgfmathsetmacro\cx{(\delt/(1-\delt^2))*(\mm + \myu/\delt)}
\pgfmathsetmacro\cy{\mm + \delt*\cx}
\pgfmathsetmacro\cxq{(\delt/(1-\delt^2))*(\mmq + \myu/\delt)}
\pgfmathsetmacro\cyq{\mmq + \delt*\cxq}
\pgfmathsetmacro\dx{(\delt/(1-\delt^2))*(\mmm - \cy + \cx/\delt)}
\pgfmathsetmacro\ddx{(\delt/(1-\delt^2))*(\mmmm - \cy + \cx/\delt)}
\pgfmathsetmacro\dy{\mmm + \delt*\dx}
\pgfmathsetmacro\ddy{\mmmm + \delt*\dx}
\pgfmathsetmacro\ex{(\delt/(1-\delt^2))*(\mmm - 2 + \delt + 1/\delt)}
\pgfmathsetmacro\eex{(\delt/(1-\delt^2))*(\mmmm - 2 + \delt + 1/\delt)}
\pgfmathsetmacro\ey{\mmm + \delt*\ex}
\pgfmathsetmacro\eey{\mmmm + \delt*\ex}
\pgfmathsetmacro\ix{\delt - \delt^2}
\pgfmathsetmacro\jx{\myu}
\pgfmathsetmacro\kx{1 - 2*\delt + \delt^2}
\coordinate (A) at (\delt*\u,\delt*\u);
\coordinate (B) at (\bx*\u,\by*\u);
\coordinate (Bq) at (\bxq*\u,\byq*\u);
\coordinate (C) at (\cx*\u,\cy*\u);
\coordinate (Cq) at (\cxq*\u,\cyq*\u);
\coordinate (D) at (\dx*\u,\dy*\u);
\coordinate (DD) at (\ddx*\u,\ddy*\u);
\coordinate (E) at (\ex*\u,\ey*\u);
\coordinate (EE) at (\eex*\u,\eey*\u);
\coordinate (F) at ({(1-\delt)*\u},{(1-\delt)*\u});
\draw (0,0) -- (1*\u,0) -- (1*\u,1*\u) -- (0,1*\u) -- cycle;
\draw[gray!50] (0,\delt*\u) -- ++(1*\u,0)
	(0,{(1-\delt)*\u}) -- ++(1*\u,0);
\draw[dashed]
	(0,\mm*\u) -- ++(1.1*\u,\delt*1.1*\u)
	(0,\mmq*\u) -- ++(1.1*\u,\delt*1.1*\u)
	(0,\mmm*\u) -- ++(1.1*\u,\delt*1.1*\u)
	(0,\mmmm*\u) -- ++(1.1*\u,\delt*1.1*\u);
\draw[gray!50] (\delt*\u,0) -- ++(0,1*\u)
	(\myu*\u,0) -- ++(0,1*\u)
	({(1-\delt)*\u},0) -- ++(0,1*\u);
\draw[dashed]
	(\ix*\u,0) -- ++(\delt*1.1*\u,1.1*\u)
	(\jx*\u,0) -- ++(\delt*1.1*\u,1.1*\u)
	(\kx*\u,0) -- ++(\delt*1.1*\u,1.1*\u);
    \coordinate (start0) at ($(0,0)!0.8!(A)$);
    \coordinate (end0)   at ($(A)!0.2!(B)$);
    \coordinate (start1) at ($(A)!0.7!(B)$);
    \coordinate (end1)   at ($(B)!0.1!(C)$);
    \coordinate (sstart1) at ($(A)!0.6!(Bq)$);
    \coordinate (eend1)   at ($(Bq)!0.1!(Cq)$);
    \coordinate (start2) at ($(B)!0.9!(C)$);
    \coordinate (end2)   at ($(C)!0.3!(D)$);
    \coordinate (sstart2) at ($(Bq)!0.9!(Cq)$);
    \coordinate (eend2)   at ($(Cq)!0.3!(C)$);
    \coordinate (start3) at ($(C)!0.7!(D)$);
    \coordinate (end3)   at ($(D)!0.1!(E)$);
    \coordinate (sstart3) at ($(C)!0.82!(DD)$);
    \coordinate (eend3)   at ($(DD)!0.1!(EE)$);
    \coordinate (start4) at ($(D)!0.9!(E)$);
    \coordinate (end4)   at ($(E)!0.3!(F)$);
    \coordinate (sstart4) at ($(DD)!0.9!(EE)$);
    \coordinate (eend4)   at ($(EE)!0.5!(F)$);
    \coordinate (start5) at ($(E)!0.8!(F)$);
    \coordinate (end5)   at ($(F)!0.2!(1*\u,1*\u)$);
    \coordinate (c1) at ($(start0) + 0.5*(\delt,\delt)$);
    \coordinate (c2) at ($(end0) - 0.45*(\bx-\delt,\by-\delt)$);
    \coordinate (c3) at ($(start1) + 0.5*(\bx-\delt,\by-\delt)$);
    \coordinate (c4) at ($(end1) - 0.45*(\cx-\bx,\cy-\by)$);
    \coordinate (cc3) at ($(sstart1) + 0.1*(\bxq-\delt,\byq-\delt)$);
    \coordinate (cc4) at ($(eend1) - 0.4*(\cxq-\bxq,\cyq-\byq)$);
    \coordinate (c5) at ($(start2) + 0.5*(\cx-\bx,\cy-\by)$);
    \coordinate (c6) at ($(end2) - 0.45*(\dx-\cx,\dy-\cy)$);
    \coordinate (cc5) at ($(sstart2) + 0.3*(\cxq-\bxq,\cyq-\byq)$);
    \coordinate (cc6) at ($(eend2) - 0.2*(\dx-\cxq,\dy-\cyq)$);
    \coordinate (c7) at ($(start3) + 0.5*(\dx-\cx,\dy-\cy)$);
    \coordinate (c8) at ($(end3) - 0.45*(\ex-\dx,\ey-\dy)$);
    \coordinate (cc7) at ($(sstart3) + 0.5*(\dx-\cx,\dy-\cy)$);
    \coordinate (cc8) at ($(eend3) - 0.45*(\ex-\dx,\ey-\dy)$);
    \coordinate (c9) at ($(start4) + 0.5*(\ex-\dx,\ey-\dy)$);
    \coordinate (c10) at ($(end4) - 0.45*(1-\delt-\ex,1-\delt-\ey)$);
    \coordinate (cc9) at ($(sstart4) + 0.5*(\ex-\dx,\ey-\dy)$);
    \coordinate (cc10) at ($(eend4) - 0.45*(1-\delt-\ex,1-\delt-\ey)$);
    \coordinate (c11) at ($(start5) + 0.5*(1-\delt-\ex,1-\delt-\ey)$);
    \coordinate (c12) at ($(end5) - 0.45*(\delt,\delt)$);
\draw[red,thick] (A) -- (B) -- (C) -- (D) -- (E) -- (F);
\draw[red,thick] (D) -- (DD) -- (EE);
\draw[red,thick] (Bq) -- (Cq) -- (C);
    \draw[blue,thick] (start0) .. controls (c1) and (c2) .. (end0);
    \draw[blue,thick] (start1) .. controls (c3) and (c4) .. (end1);
    \draw[blue,thick] (sstart1) .. controls (cc3) and (cc4) .. (eend1);
    \draw[blue,thick] (start2) .. controls (c5) and (c6) .. (end2);
    \draw[blue,thick] (sstart2) .. controls (cc5) and (cc6) .. (eend2);
    \draw[blue,thick] (start3) .. controls (c7) and (c8) .. (end3);
    \draw[blue,thick] (sstart3) .. controls (cc7) and (cc8) .. (eend3);
    \draw[blue,thick] (start4) .. controls (c9) and (c10) .. (end4);
    \draw[blue,thick] (sstart4) .. controls (cc9) and (cc10) .. (eend4);
    \draw[blue,thick] (start5) .. controls (c11) and (c12) .. (end5);
\draw[red,thick,dotted] (0,0) -- (1*\delt*\u,1*\delt*\u);
\draw[red,thick,dotted] ({(1-\delt)*\u},{(1-\delt)*\u}) -- (1*\u,1*\u);
\node[left]  at (0,\delt*\u) {\sz $\de$};
\node[left]  at (0,\mmq*\u) {\sz $m_1$};
\node[left]  at (0,\mm*\u) {\sz $m'_1$};
\node[left]  at (0,\mmm*\u) {\sz $m_2$};
\node[left]  at (0,\mmmm*\u) {\sz $m'_2$};
\node[left]  at (0,{(1-\delt)*\u}) {\sz $1-\de$};
\node[left]  at (0,1*\u) {\sz $1$};
\node[below left]  at (0,0) {\sz $0$};
\node[below,xshift=0.25ex]  at (1*\u,0) {\sz $1$};
\node[below]  at (\delt*\u,0) {\sz $\de$};
\node[below]  at (\myu*\u,0) {\sz \rule{0pt}{5pt}$u$};
\node[below,xshift=-0.25ex]  at ({(1-\delt)*\u},0) {\sz $1-\de$};
\node[right] at (A) {\sz $A$};
\node[below right] at (B) {\sz $B'$};
\node[below right] at (Bq) {\sz $B$};
\node[below left] at (C) {\sz $C'$};
\node[below left] at (Cq) {\sz $C$};
\node[below left] at (D) {\sz $D$};
\node[above left,xshift=0.6ex,yshift=-0.5ex] at (DD) {\sz $D'$};
\node[below left] at (E) {\sz $E$};
\node[above left,xshift=0.6ex,yshift=-0.5ex] at (EE) {\sz $E'$};
\node[left,yshift=0.2ex] at (F) {\sz $F$};
\end{tikzpicture}
\]
\end{lem}

\begin{proof}
We assume that $m_2\leq m'_2$ as in the diagram, the other case being similar. We thus have $m_1\leq m'_1 < m'_1+\de < m_2\leq m'_2$. The inequality $m_2 - m'_1 > \de$ ensures by Lemma \ref{l:delta.room} that $d_y - c'_y>\de$ and hence $d_x-c'_x > \de^2$. It follows that $c'_x + \eta < d_x - \eta$, implying that the intervals $I_1 = (b_x-\eta,c'_x+\eta)$ and $I_2 = (d_x-\eta,e'_x+\eta)$ are disjoint. We prove the statements about $I_1$, the ones about $I_2$ being entirely analogous.

Let $\g=(1/\de,\de,b_x,b_y,\eta)$, $\g'=(1/\de,\de,b'_x,b'_y,\eta)\in \Gamma$.
From the inequalities in Lemma \ref{l:ineq}, we see that $b_x+\eta < b'_x+\eta < 2\de < c_x-\eta < c'_x-\eta$. Thus, $\rho_{\de u v t}=G_{\g}$ and $\rho_{\de u v' t'}=G_{\g'}$ on the interval $[b_x-\eta,2\de]$.

By Proposition \ref{roundoffcorners} (3), the functions $G_{\g}$ and $G_{\g'}$ are related by $G_{\g'} = G_\g + v$, where $v = B'-B$. Since $v$ is the direction of the support line to the graph of $G_\g$ at $b_x-\eta$ and $D^2G_\g<0$ on $(b_x-\eta,b_x+\eta)$, it follows from Corollary \ref{c:translate.conv} (1) that $G_\g < G_{\g'}$ on $(b_x-\eta,2\de]$ and the difference $G_{\g'}-G_\g$ is increasing, so its maximum value is attained at $2\de$, and hence is $m'_1-m_1$. Thus the inequality $\rho_{\de u v t} < \rho_{\de u v' t'}$ holds on $(b_x-\eta,2\de]$ with $\rho_{\de u v' t'} - \rho_{\de u v t}\leq m'_1-m_1$ on that interval.
Arguing similarly at $C$ and $C'$ (using Corollary \ref{c:translate.conv} (2) with $\rho_{\de u v' t'}$ playing the role of $f$,
we get the same inequality and bound on $[2\de,c'_x+\eta)$. Together these give the claimed properties on $I_1$.
\end{proof}

\begin{cor}\label{cor:g.incre}
For fixed $(\de,u,v)$ with $(u,v)\in T$ and $0<\de<\de(u,v)$, the integral $I(\rho_{\de u v t})(1)$ is a continuous strictly increasing function of $t$ for $t_1(u,v) \leq t \leq t_2(u,v)$.
\end{cor}

\begin{proof}
The continuity follows from Lemma \ref{l:cont.rho}.
By Proposition \ref{p:S*}, when $(u,v)\in T$, $0<\de<\de(u,v)$, and $t_1(u,v) < t < t_2(u,v)$, we have $(\de,u,v,t)\in S$, so $\rho_{\de u v t}$ is defined. Given $t_1(u,v) < t < t' < t_2(u,v)$, in terms of the notation of the lemma, we are in the case $v=v'$, so both $\rho_{\de u v t}$ and $\rho_{\de u v t'}$ have the same value of $m_1$, and $m_2<m'_2$. The assumption $m'_1+\de<\min(m_2,m'_2)$ thus says $m_1 + \de < m_2$ which is true since $(\de,u,v,t)\in S$. By Lemma \ref{l:g.incre} then, $\rho_{\de u v t}\leq \rho_{\de u v t'}$ with strict inequality on $I_2$, so $I(\rho_{\de u v t})(1) < I(\rho_{\de u v t'})(1)$.
\end{proof}

\begin{lem}\label{l:int.close}
For any $(\de,u,v,t)\in S$ and $x\in [0,1]$, $|I^n(\rho_{\de u v t})(x) - I^n(f_{u v t})(x)| < 6\de$ for $n\in\N$. In particular, $|I(\rho_{\de u v t})(1) - t| = |I(\rho_{\de u v t})(1) - I(f_{u v t})(1)| < 6\de$.
\end{lem}

\begin{proof}
We have $|\rho_{\de u v t} - f_{u v t}|<\de$ on $[0,1]\sm K$ where $K = [0,2\de]\cup [u,u+\de]\cup [1-2\de,1]$. Hence, denoting Lebesgue measure by $\mu$, we have
\begin{align*}
|I(\rho_{\de u v t})(x) - I(f_{u v t})(x)| & \leq I(|\rho_{\de u v t} - f_{u v t}|)(x) \leq I(|\rho_{\de u v t} - f_{u v t}|)(1) \\
& = \int_{[0,1]\sm K} |\rho_{\de u v t} - f_{u v t}| + \int_{K} |\rho_{\de u v t} - f_{u v t}| \\
& < \de (1-\mu (K)) + \mu(K) < \de + 5\de = 6\de.
\end{align*}
Then if $|I^n(\rho_{\de u v t}) - I^n(f_{u v t})| < 6\de$ on $[0,1]$, we get for $x\in [0,1]$,
\begin{align*}
|I^{n+1}(\rho_{\de u v t})(x) - I^{n+1}(f_{u v t})(x)| & =
|I(I^{n}(\rho_{\de u v t}) - I^{n}(f_{u v t}))(x)| \\
& \leq
|I(|I^{n}(\rho_{\de u v t}) - I^{n}(f_{u v t})|)(x)| < 6\de \qedhere.
\end{align*}
\end{proof}

\begin{lem}\label{l:def.t(u,v,t)}
Let $R=\{(\de,u,v):(u,v)\in T,\,0<\de<\de(u,v)\}$. For each $(\de,u,v)\in R$, there is a unique $t(\de,u,v)$ such that $t_1(u,v) < t(\de,u,v) < t_2(u,v)$ and $I(\rho_{\de, u, v, t(\de, u, v)})(1)=a$. The function $(\de,u,v)\mapsto t(\de, u, v)$ is continuous on $R$ and $|t(\de,u,v) - a| < 6\de$.
\end{lem}

\begin{proof}
Since $\de<\de(u,v)$ and $t_1(u,v) + 6\de(u,v) < a < t_2(u,v) - 6\de(u,v)$, we have
\[
t_1(u,v) < a-6\de < a < a+6\de < t_2(u,v).
\]
Lemma \ref{l:int.close} applied to $t=a-6\de$ and $t=a+6\de$, gives
\[
I(\rho_{\de, u, v, a-6\de})(1) < a < I(\rho_{\de, u, v, a+6\de})(1).
\]
By Lemma \ref{l:cont.rho}, the function $t\mapsto I(\rho_{\de, u, v, t})(1)$ is continuous on $[t_1(u,v),t_2(u,v)]$. By the Intermediate Value Theorem and Corollary \ref{cor:g.incre}, there exists a value $t(\de, u, v)\in(a-6\de,a+6\de)$, unique in $[t_1(u,v),t_2(u,v)]$ such that $I(\rho_{\de, u, v, t(\de, u, v)})(1)=a$. By Proposition \ref{p:cont.of.sect} applied to the function $(\de,u,v,t)\mapsto I(\rho_{\de, u, v, t})(1)$ on the open set
\[
E = \{(\de,u,v,t):(u,v)\in T,\, 0<\de<\de(u,v),\, t_1(u,v)<t<t_2(u,v)\}\sq R\times \R,
\]
$(\de,u,v)\mapsto t(\de, u, v)$ is continuous on $R$.
\end{proof}

Lemma \ref{l:def.t(u,v,t)} justifies the following definition.

\begin{defn}
Denote by $(\de,u,v)\mapsto \rho_{\de,u,v}$ the continuous map $R\to C^\infty([0,1])$ given by $\rho_{\de,u,v} = \rho_{\de, u, v, t(\de, u, v)}$.
\end{defn}

\begin{lem}\label{l:int.close.b}
For $(\de,u,v)\in R$ and $x\in [0,1]$, $|I^n(\rho_{\de u v})(x) - I^n(f_{u v})(x)| < 6\de(2-u)/(1-u)$ for $n\in\N$.
\end{lem}

\begin{proof}
From Lemma \ref{l:g.incre} we have
\[
|\rho_{\de, u, v, t(\de,u,v)} - \rho_{\de u v a}|\leq \left|\frac{t(\de,u,v)-v}{1-u} - \frac{a-v}{1-u}\right| = \frac{|t(\de,u,v)-a|}{1-u}.
\]
Using Lemmas \ref{l:int.close} and \ref{l:def.t(u,v,t)}, this then yields
\begin{align*}
|I(\rho_{\de u v})(x) & - I(f_{u v})(x)| = |I(\rho_{\de, u, v, t(\de,u,v)})(x) - I(f_{u v a})(x)| \\
& \leq |I(\rho_{\de, u, v, t(\de,u,v)})(x) - I(\rho_{\de u v a})(x)| + |I(\rho_{\de u v a})(x) - I(f_{u v a})(x)| \\
& < \frac{|t(\de,u,v)-a|}{1-u} + 6\de < \frac{6\de}{1-u} + 6\de = 6\de\cdot \frac{2-u}{1-u}.
\end{align*}
For $n>1$, proceed as in the proof of Lemma \ref{l:int.close}.
\end{proof}

Now fix $r_1$ and $r_2$ such that $u_1<r_1<r_2<u_2$. Let $P(r_1,r_2,\theta)$ be the closed parallelogram whose sides are segments of the lines having equations $x=r_1$, $x=r_2$, $y=ax-a+2(b+\theta)$, $y=ax-a+2(b-\theta)$, where $\theta$ is chosen small enough so that $P(r_1,r_2,\theta)\sq T$, as in the diagram below.
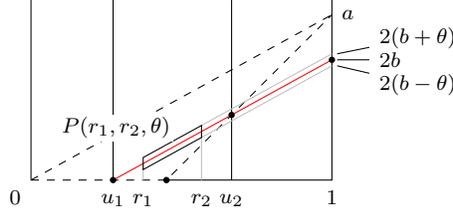
\[
\begin{tikzpicture}
\def\u{4cm}
\def\a{0.55}
\def\b{0.2}
\def\xx{((1-2*\a+2*\b)/(1-\a))}
\def\os{0.1} 
\def\ph{0.02} 
    \draw ({(1-\a)*\u},0) -- (1*\u,0) -- (1*\u,0.6*\u) (0,0.6*\u) -- (0,0);
    \draw[dashed ] (0,0) -- (1*\u,\a*\u)
        (0,0) -- ({(1-\a)*\u},0) -- (1*\u,\a*\u);
    \coordinate (p1) at (  0*\u, 0*\u);
    \coordinate (p2) at (  0.4*\u, {(-0.6*\a + 2*\b)*\u});
    \coordinate (p3) at (  1*\u,\a*\u);
    \coordinate (p4) at ( {(1-2*\b/\a)*\u}, 0*\u);
    \coordinate (p4l) at ( {(\os+1-2*\b/\a)*\u}, {(\a*(\os+1-2*\b/\a)-\a+2*\b)*\u});
    \coordinate (p4r) at ( {(-\os+\xx)*\u},{(\a*(-\os+\xx)-\a+2*\b)*\u});
    \coordinate (p5) at (1*\u,2*\b*\u);
    \coordinate (p6) at ( {(1-\a)*\u}, 0*\u);
    \coordinate (p2t) at (  {\xx*\u}, {((\xx-1)*\a + 2*\b)*\u});
    \draw (p4) -- ++(0,0.6*\u);
    \draw ({\xx*\u},0) -- ++(0,0.6*\u);
    \draw[red] (p4) -- (p5);
    \fill (p4) circle [radius = 1.2pt];
    \fill (p5) circle [radius = 1.2pt];
    \fill (p6) circle [radius = 1.2pt];
    \fill (p2t) circle [radius = 1.2pt];
    \draw[gray!50] (p4l) -- (p4l |- 0,0);
    \draw[gray!50] (p4r) -- (p4r |- 0,0);
    \draw[gray!50] ($(p4r)+(0,\ph*\u)$) -- (1*\u,{(2*\b+\ph)*\u});
    \draw[gray!50] ($(p4r)+(0,-\ph*\u)$) -- (1*\u,{(2*\b-\ph)*\u});
    \draw ($(p4l)+(0,\ph*\u)$) -- ++(0,-2*\ph*\u) -- ($(p4r)+(0,-\ph*\u)$) -- ++(0,2*\ph*\u) -- node[fill=white,inner sep=1,above,xshift=-4.5ex] {\sz $P(r_1,r_2,\theta)$} cycle;
    \node[below left]  at (0,0) {\sz $0$};
    \node[below]  at (1*\u,0) {\sz $1$};
    \node[below]  at (p4) {\sz \rule{0pt}{5pt}$u_1$};
    \node[below]  at ({\xx*\u},0) {\sz \rule{0pt}{5pt}$u_2$};
    \node[below]  at (p4l |- 0,0) {\sz \rule{0pt}{5pt}$r_1$};
    \node[below]  at (p4r |- 0,0) {\sz \rule{0pt}{5pt}$r_2$};
    \node[right]  at (1*\u,\a*\u) {\sz $a$};
    \node[right,xshift=3ex,yshift=1.3ex]
        (top) at ($(p5)+(0,\ph*\u)$) {\sz $2(b+\theta)$};
    \node[right,xshift=3ex]
        (mid) at (p5) {\sz $2b$};
    \node[right,xshift=3ex,yshift=-1.5ex]
        (bot) at ($(p5)+(0,-\ph*\u)$) {\sz $2(b-\theta)$};
    \draw ($(p5)+(0,\ph*\u)+(2pt,0)$) -- (top);
    \draw ($(p5)+(2pt,0)$) -- (mid);
    \draw ($(p5)+(0,-\ph*\u)+(2pt,0)$) -- (bot);
\end{tikzpicture}
\]
Define $v^*(x,\theta) = ax - a + 2(b+\theta)$.

Since $P(r_1,r_2,\theta)$ is a compact subset of $T$ and $\de(u,v)$ is a positive continuous function on $T$, $\de^*(r_1,r_2,\theta) = \min\{\de(u,v):(u,v)\in P(r_1,r_2,\theta)\}$ exists and is a positive 
function on the open set
\[
U = \{(r_1,r_2,\theta):u_1<r_2<r_2<u_2,\,\theta>0,\,P(r_1,r_2,\theta)\sq T\}
\]
in $\R^3$. By truncating $\de^*(r_1,r_2,\theta)$, we may ask also that
\[
\de^*(r_1,r_2,\theta) \leq \frac{\theta(1-r_2)}{6(2-r_1)}.
\]
For $(r_1,r_2,\theta)\in U$,
\[
R^*(r_1,r_2,\theta) = \{(\de,u,v):(u,v)\in P(r_1,r_2,\theta),\,0<\de<\de^*(r_1,r_2,\theta)\}
\]
is a subset of $R$, so $\rho_{\de u v}$ is defined for $(\de,u,v)\in R^*(r_1,r_2,\theta)$.

\begin{lem}\label{l:f<g>f}
Let $f,g\colon [0,1]\to \R$ be continuous functions satisfying for some $0<c<1$ that $f\leq g$ on $[0,c]$ and $g\leq f$ on $[c,1]$.
If $I(f)(1)\leq I(g)(1)$ then $I(f)\leq I(g)$.
\end{lem}

\begin{proof}
We clearly have $I(f)(x)\leq I(g)(x)$ for $0\leq x\leq c$. Suppose $I(g)(d)<I(f)(d)$ for some $d$ with $c<d<1$. Then since $g\leq f$ on $[d,1]$, we have
\[
I(g)(1) = I(g)(d) + \int_d^1 g(x)\,dx < I(f)(d) + \int_d^1 f(x)\,dx = I(f)(1),
\]
a contradiction.
\end{proof}

\begin{prop}\label{p:I2}
\begin{enumerate}
\item
For $(u,v),\,(u,v') \in T$, $0<\de<\min(\de(u,v),\de(u,v'))$ with $v<v'$, we have $I^2(\rho_{\de u v})(1) < I^2(\rho_{\de u v'})(1)$.

\item
Let $(r_1,r_2,\theta)\in U$. When $0<\de<\de^*(r_1,r_2,\theta)$ and $r_1<u<r_2$, the function $v\mapsto I^2(\rho_{\de u v})(1)$ is continuous and strictly increasing on $[v^*(u,-\theta),v^*(u,\theta)]$.
\end{enumerate}
\end{prop}

\begin{proof}
(1) Write
\[
m_1 = \frac{v}u,\,m_2 = \frac{t(\de,u,v)-v}{1-u},\,m'_1 = \frac{v'}u,\,m'_2 = \frac{t(\de,u,v')-v'}{1-u}.
\]
Since $I(\rho_{\de u v})(1) = I(\rho_{\de u v'})(1) = a$ and $m_1<m'_1$, we necessarily have $m'_2<m_2$, so the proposition applies with $f=\rho_{\de u v}$, $g=\rho_{\de u v'}$, and $c$ any number between $c'_x+\eta$ and $d'_x-\eta$, yielding $I(\rho_{\de u v}) \leq I(\rho_{\de u v'})$. We have strict inequality on $(b_x-\eta,c'_x+\eta)$, so
$I^2(\rho_{\de u v})(1) < I^2(\rho_{\de u v'})(1)$.

(2) When $0<\de<\de^*(r_1,r_2,\theta)$ and $r_1<u<r_2$, the map $v\mapsto I^2(\rho_{\de u v})(1)$ is continuous on $[v^*(u,-\theta),v^*(u,\theta)]$ since $(\de,u,v)\mapsto \rho_{\de u v}$ is continuous on $R$. By (1) it is strictly increasing.
\end{proof}

\begin{lem}\label{l:def.v(de,u)}
Let $W(r_1,r_2,\theta)=\{(\de,u):0<\de<\de^*(r_1,r_2,\theta),\, r_1<u<r_2\}$. For each $(\de,u)\in W(r_1,r_2,\theta)$, there is a unique $v(\de,u)$ such that $v^*(u,-\theta) < v(\de,u) < v^*(u,\theta)$ and $I^2(\rho_{\de, u, v(\de, u)})(1) = b$. The function $(\de,u)\mapsto v(\de, u)$ is continuous on $W(r_1,r_2,\theta)$ and $|v(\de,u) - v^*(u,0)| < 2\theta$.
\end{lem}

\begin{proof}
Let $(\de,u)\in W(r_1,r_2,\theta)$. For $0<\de<\de^*(r_1,r_2,\theta)$ and $(u,v)\in P(r_1,r_2,\theta)$, by Lemma \ref{l:int.close.b} and the choice of $\de^*(r_1,r_2,\theta)$, we have
\[
|I^2(\rho_{\de u v})(1) - I^2(f_{u v})(1)| < \frac{6\de(2-r_1)}{1-r_2} < \theta.
\]
By Proposition \ref{p:I2.of.f} (1) (with $b$ replaced by $b\pm\theta$), the double inequality $b-\theta < b < b+\theta$ can also be written
\[
I^2(f_{u,v^*(u,-\theta)})(1) < b < I^2(f_{u,v^*(u,\theta)})(1),
\]
and therefore
\[
I^2(\rho_{\de u v^*(u,-\theta)})(1) < b < I^2(\rho_{\de u v^*(u,\theta)})(1).
\]
By Proposition \ref{p:I2} (2) and the Intermediate Value Theorem, there is a unique
\[v(\de, u)\in (v^*(u,-\theta), v^*(u,\theta))\] such that $I^2(\rho_{\de, u, v(\de, u)})(1) = b$. We have
$|v(\de,u)-v^*(u,0)| < |v^*(u,\pm\theta)-v^*(u,0)| = 2\theta$.

Continuity of $(\de,u)\mapsto v(\de, u)$ on $W(r_1,r_2,\theta)$ follows from Proposition \ref{p:cont.of.sect} applied to the function $(\de,u,v)\mapsto I^2(\rho_{\de, u, v})(1)$ on the open set
$E = \{(\de,u,v): 0<\de<\de^*(r_1,r_2,\theta),\, r_1<u<r_2,\,v^*(u,-\theta)<v<v^*(u,\theta)\}\sq W(r_1,r_2,\theta)\times \R$.
\end{proof}

Lemma \ref{l:def.v(de,u)} justifies the following definition.

\begin{defn}
Denote by $(\de,u)\mapsto \rho_{\de u}$ the continuous map $W(r_1,r_1,\theta)\to C^\infty([0,1])$ given by $\rho_{\de u} = \rho_{\de, u, v(\de, u)}$.
\end{defn}

For a fixed $\de<\de^*(r_1,r_2,\theta)$, $u\mapsto \rho_{\de u}$ is continuous on $(r_1,r_2)$ and therefore so is $u\mapsto I^3(\rho_{\de, u})(1)$.

\begin{lem}\label{l:int.close.d}
For $(\de,u)\in W(r_1,r_2,\theta)$ and $x\in [0,1]$, we have $|I^n(\rho_{\de u})(x) - I^n(f_{u, v^*(u,0)})(x)| < 24(\theta+\de)/\min(r_1,1-r_2)$ for $n\in\N$.
\end{lem}

\begin{proof}
By Lemma \ref{l:g.incre}, since $|t(\de,u,v(\de,u)) - t(\de,u,v^*(u,0))| < 12\de$ by Lemma \ref{l:def.t(u,v,t)},
\begin{align*}
& |\rho_{\de, u, v(\de,u)} - \rho_{\de, u, v^*(u,0)}| \\
& \leq
\max\left(\left|\frac{v(\de,u)}{u} - \frac{v^*(u,0)}{u}\right|,\,
\left| \frac{t(\de,u,v(\de,u)) - v(\de,u)}{1-u} - \frac{t(\de,u,v^*(u,0)) - v^*(u,0)}{1-u}\right|\right) \\
& < \max\left(\frac{2\theta}{u},\frac{2\theta+12\de}{1-u}\right) \leq \frac{2\theta+12\de}{\min(u,1-u)} \leq \frac{2\theta+12\de}{\min(r_1,1-r_2)}
\end{align*}
and therefore, using Lemma \ref{l:int.close.b},
\begin{align*}
|I(\rho_{\de u})(x) & - I(f_{u, v^*(u,0)})(x)| = |I(\rho_{\de, u, v(\de,u)})(x) - I(f_{u, v^*(u,0)})(x)| \\
& \leq |I(\rho_{\de, u, v(\de,u)})(x) - I(\rho_{\de, u, v^*(u,0)})(x)| + |I(\rho_{\de, u, v^*(u,0)})(x) - I(f_{u, v^*(u,0)})(x)| \\
& < \frac{2\theta + 12\de}{\min(r_1,1-r_2)} + 6\de\cdot \frac{2-r_1}{1-r_2} \leq \frac{2\theta + 6\de(4-r_1)}{\min(r_1,1-r_2)} \leq \frac{24(\theta+\de)}{\min(r_1,1-r_2)}.
\end{align*}
For $n>1$, proceed as in the proof of Lemma \ref{l:int.close}.
\end{proof}

By Proposition \ref{p:g1g2}, we may choose $r_1,r_2,s_1,s_2$ with $u_1<r_1<s_1<s_2<r_2<u_2$, so that $H_1=I^3(f_{s_1,v^*(s_1,0)}) < c < H_2= I^3(f_{s_2,v^*(s_2,0)})$.
Choose positive numbers $\de,\theta$ small enough so that $24(\theta+\de)/\min(r_1,1-r_2) < \min(c-H_1,H_2-c)$. By taking $\theta$ smaller, we can assume that $(r_1,r_2,\theta)\in U$. Then by taking $\de$ smaller, we can assume that $\de<\de^*(r_1,r_2,\theta)$ and hence $(\de,u) \in W(r_1,r_2,\theta)$ for all $u\in [s_1,s_2]$. It then follows from $H_1<c<H_2$ and Lemma \ref{l:int.close.d} that
\[
I^3(\rho_{\de s_1})(1) < c < I^3(\rho_{\de s_2})(1).
\]
By continuity of $u\mapsto I^3(\rho_{\de u})(1)$ on $[s_1,s_2]$, there is a value of $u\in [s_1,s_2]$ for which $I^3(\rho_{\de u})(1)=c$. This $f=\rho_{\de u}$ is the desired function.

This completes the proof of Theorem \ref{t:n=3.b}.

\section{Proof of Theorem \ref{t:conj}}
\label{s:Thm9.1}

By terminating the proof of Theorem \ref{t:n=3.b} in the appropriate place, we obtain similar theorems for $n=1,2$. The case $n=0$ is just Example \ref{n=0}. Alternatively, we can deduce these theorems from Theorem \ref{t:n=3.b}. For that reason, we state these as a corollary to Theorem \ref{t:n=3.b}.

\begin{cor}\label{c:n=0,1,2}
\begin{enumerate}
\item
For each small enough $\de>0$, there is a $C^\infty$ function $f\colon [0,1]\to[0,1]$ such that $f=\s_\de$ on $[0,\de]$, $f=\tau_\de$ on $[1-\de,1]$, and $Df>0$ on $(0,1)$.

\item
Let $a$ satisfy $0<a<1$.
For each small enough $\de>0$, there is a $C^\infty$ function $f\colon [0,1]\to[0,1]$ such that $f=\s_\de$ on $[0,\de/2]$, $f=\tau_\de$ on $[1-\de/2,1]$, $Df>0$ on $(0,1)$, and $I(f)(1)=a$.

\item
Let $a,b\in\R$ satisfy $0<a<1$ and $a^2/2 < b < a/2$.
For each small enough $\de>0$, there is a $C^\infty$ function $f\colon [0,1]\to[0,1]$ such that $f=\s_\de$ on $[0,\de/2]$, $f=\tau_\de$ on $[1-\de/2,1]$, $Df>0$ on $(0,1)$, $I(f)(1)=a$, and $I^2(f)(1)=b$.
\end{enumerate}
\end{cor}

\begin{proof}
To apply Theorem \ref{t:n=3.b}, we require suitable numbers $a,b,c$. In (1), we can start with any $a$ such that $0<a<1$. In (2) and (3), we are given $a$. In (1) and (2), we then need a number $b$. Since $0<a<1$, we have $a^2<a$, so we can choose a number $b$ satisfying $a^2/2<b<a/2$. Then in all three parts, by Proposition \ref{p:g1g2} we have
\[
2b^2/(3a) < (-a^2 + 2ab -4b^2 + 2b)/(6(1-a)),
\]
so we can choose a number $c$ satisfying
\[
2b^2/(3a) < c < (-a^2 + 2ab -4b^2 + 2b)/(6(1-a)).
\]
Then Theorem \ref{t:n=3.b} applies to give the desired function $f$.
\end{proof}

We now prove Theorem \ref{t:conj} stated in the introduction as Theorem A. For convenience, we restate it here.

\begin{thm}\label{t:conj}
The statements $(P_n)$, $n=0,1,2,3$, all hold. We have the following:
\begin{align*}
W_0 & = \{a\in \R:0<a\}, \\
W_1 & = \{(a,b)\in \R^2 :0<a<b\}, \\
W_2 & = \{(a,b,c)\in \R^3 :0<2a<b,\,b^2<2ac\}, \\
W_3 & = \{(a,b,c,d)\in \R^4 :0<c<d,\,2b^2<3ac,\,6ad+c^2+4b^2 < 6ac + 2bc + 2bd\}.
\end{align*}
\end{thm}

\begin{proof}
We check $(P_n)$ in the form given in Proposition \ref{p:equiv.P_n} (2), using the integral form of the definitions from Proposition \ref{p:W.equiv}.
Each of the four formulas above has the form $W_n=S_n$ for some set $S_n$ which is clearly open. We must also show $W_n=W_n^\infty$. The inclusions $W_n^\infty\sq W_n$ are clear, so the statements are proven if we show $W_n\sq S_n$ and $S_n\sq W_n^\infty$. In the arguments below, $\al$ and $\beta$ denote arbitrary but given elements $\al\in \wt{\A}$, $\beta\in\wt{\B}$.

The case $n=0$ is covered by \cite{Bu2019}, Proposition 6.2, but we prove it here for completeness.
The inclusion $W_0\sq S_0$ is clear since it states only that if $f\in\F_0$, (i.e.,  $f$ is continuous and increasing but not constant) and $f(0)=0$, then $f(1)>0$.

To see that $S_0\sq W_0^\infty$, let $a\in S_0$, i.e., $a>0$.
We want an $f\in\F_0^\infty$ such that $f(0)=0$, $f(1)=a$ and $D^jf(0)=\al_j$, $D^jf(1)=\beta_j$, $j\in\N$. By Proposition \ref{p:endpt.der}, the functions $\s_\de$ and $\tau_\de$ fixed at the beginning of Section \ref{s:spec.higher} could have been chosen so that $D^j\s_\de(0)=\al_j/a$, $D^j\tau_\de(1)=\beta_j/a$, $j\in\N$. Then Corollary \ref{c:n=0,1,2} (1) gives, for each small enough $\de>0$, a $C^\infty$ function $g\colon [0,1]\to[0,1]$ such that $g=\s_\de$ on $[0,\de]$, $g=\tau_\de$ on $[1-\de,1]$, and $Dg>0$ on $(0,1)$.
The function $f=ag$ is as desired.

For the case $n=1$,
The inequality $W_1\sq S_1$ follows from Proposition \ref{n=1.a}. Given $f\in\F_0$ satisfying $f(0)=0$, if we write $b=f(1)$ and $a=I(f)(1)$ then the function $g=b^{-1}f$ maps into $[0,1]$. We have $I(g)(1) = b^{-1}a$. Plugging this value in for the $a$ of Proposition \ref{n=1.a}, we get $0<b^{-1}a<1$, or $0<a<b$, and hence $(a,b)\in S_1$.

To show $S_1\sq W_1^\infty$, let $(a,b)\in S_1$.
We want $f\in\F^\infty_0$ such that $f(0)=0$, $f(1)=b$, $I(f)(1)=a$, and
$D^jf(0)=\al_j$, $D^jf(1)=\beta_j$, $j\in\N$.
We have $0<a/b<1$.
By Proposition \ref{p:endpt.der}, the functions $\s_\de$ and $\tau_\de$ fixed at the beginning of Section \ref{s:spec.higher} could have been chosen so that
$D^j\s_\de(0)=\al_{j}/b$, $D^j\tau_\de(1)=\beta_{j}/b$, $j\in\N$.
By Corollary \ref{c:n=0,1,2} (2), for each small enough $\de>0$, there is a $C^\infty$ function $g\colon [0,1]\to[0,1]$ such that $g=\s_\de$ on $[0,\de/2]$, $g=\tau_\de$ on $[1-\de/2,1]$, $Dg>0$ on $(0,1)$, and $I(g)(1)=a/b$. The function $f=bg$ is as desired.

For the case $n=2$, the inequality $W_2\sq S_2$ follows from Proposition \ref{n=2.a}.
Given $f\in\F_0$ satisfying $f(0)=0$, if we write $c=f(1)$, $b=I(f)(1)$, $a=I^2(f)(1)$, then the function $g=c^{-1}f$ maps into $[0,1]$.
We have $I(g)(1) = b/c$, $I^2(g)(1) = a/c$. Plugging these values in for the $a$ and $b$ of Proposition \ref{n=1.a} (4), we get
$(b/c)^2/2 < a/c < (b/c)/2$, or $b^2 < 2ac$ and $2a < b$. Positivity of $a$ is clear from its definition. Thus, $(a,b,c)\in S_2$.

To show $S_2\sq W_2^\infty$, let $(a,b,c)\in S_2$, so $0<2a<b$ and $b^2<2ac$.
We want $f\in\F^\infty_0$ such that $f(0)=0$, $f(1)=c$, $I(f)(1)=b$, $I^2(f)(1)=a$, and
$D^jf(0)=\al_j$, $D^jf(1)=\beta_j$, $j\in\N$.
From $b^2<2ac$ and $2a<b$ we get $b^2<2ac<bc$, or $(b/c)^2/2 < a/c < (b/c)/2$. By Proposition \ref{p:endpt.der}, the functions $\s_\de$ and $\tau_\de$ fixed at the beginning of Section \ref{s:spec.higher} could have been chosen so that
$D^j\s_\de(0)=\al_{j}/c$, $D^j\tau_\de(1)=\beta_{j}/c$, $j\in\N$.
By Corollary \ref{c:n=0,1,2} (3), for each small enough $\de>0$, there is a $C^\infty$ function $g\colon [0,1]\to[0,1]$ such that $g=\s_\de$ on $[0,\de/2]$, $g=\tau_\de$ on $[1-\de/2,1]$, $Dg>0$ on $(0,1)$, $I(g)(1)=b/c$ and $I^2(g)(1)=a/c$. The function $f=cg$ is as desired.

For the case $n=3$, the inequality $W_3\sq S_3$ follows from Proposition \ref{p:n.equals.3}.
Given $f\in\F_0$ satisfying $f(0)=0$, if we write $d=f(1)$, $c=I(f)(1)$, $b=I^2(f)(1)$, $a=I^3(f)(1)$, then the function $g=d^{-1}f$ maps into $[0,1]$.
We have
$I(g)(1) = c/d$, and similarly
$I^2(g)(1) = b/d$,
$I^3(g)(1) = a/d$. Plugging these in respectively for the $a,b,c$ of Proposition \ref{p:n.equals.3}, we get
\[
\frac{2(b/d)^2}{3(c/d)} < \frac{a}{d} < \frac{-(c/d)^2 + 2(c/d)(b/d) -4(b/d)^2 + 2(b/d)}{6(1-(c/d))}.
\]
Multiplying by $d$ and simplifying gives
\[
\frac{2b^2}{3c} < a < \frac{-c^2 + 2cb -4b^2 + 2bd}{6(d-c)}.
\]
We can write these inequalities as $2b^2 < 3ac$ and $6ad - 6ac < -c^2 + 2cb -4b^2 + 2bd$, or $6ad + 4b^2 + c^2 < 2cb + 2bd + 6ac$. The property $0<c<d$ holds since $(c,d)\in W_1$ (witnessed by $f$).
Thus, $(a,b,c,d)\in S_3$.

To show $S_3\sq W_3^\infty$, let $(a,b,c,d)\in S_3$. We have
\[
0<c<d,\ 2b^2<3ac,\ \text{and}\ 6ad + 4b^2 + c^2 < 2cb + 2bd + 6ac.
\]
We want $f\in\F^\infty_0$ such that $f(0)=0$, $f(1)=d$, $I(f)(1)=c$, $I^2(f)(1)=b$, $I^3(f)(1)=a$, and
$D^jf(0)=\al_j$, $D^jf(1)=\beta_j$, $j\in\N$.
Since $0<c<d$, we have $0<c/d<1$.
The inequalities $2b^2<3ac$ and $6ad + 4b^2 + c^2 < 2cb + 2bd + 6ac$ can be re-written as
\[
\frac{2b^2}{3c} < a < \frac{-c^2 + 2bc -4b^2 + 2bd}{6(d-c)}.
\]
Dividing the numerators by $d^2$ and the denominators by $d$ gives the same inequalities with $a,b,c,d$ replaced by $a/d,b/d,c/d,1$.
By Proposition \ref{p:endpt.der}, the functions $\s_\de$ and $\tau_\de$ fixed at the beginning of Section \ref{s:spec.higher} could have been chosen so that
$D^j\s_\de(0)=\al_{j}/d$, $D^j\tau_\de(1)=\beta_{j}/d$, $j\in\N$.
By Theorem \ref{t:n=3.b}, for each small enough $\de>0$, there is a $C^\infty$ function $g\colon [0,1]\to[0,1]$ such that $g=\s_\de$ on $[0,\de/2]$, $g=\tau_\de$ on $[1-\de/2,1]$, $Dg>0$ on $(0,1)$, $I(g)(1)=c/d$, $I^2(g)(1)=b/d$ and $I^3(g)(1)=a/d$. The function $f=dg$ is as desired.
\end{proof}

\end{document}